\documentclass[reqno,a4paper]{siamart171218}
\usepackage{amssymb,amsmath,epsfig,graphics,psfrag,graphicx,color,cite,subfigure,arydshln}
 \usepackage{mathrsfs,comment}
\definecolor{darkred}{rgb}{0.82,0.15,0.20}
\definecolor{darkblue}{rgb}{0.82,0.15,0.12}

\usepackage[T1]{fontenc}
\usepackage{stmaryrd}
\SetSymbolFont{stmry}{bold}{U}{stmry}{m}{n}
\numberwithin{equation}{section}
\numberwithin{figure}{section}
\numberwithin{table}{section}
\renewcommand{\figurename}{Fig.}	
\definecolor{lightgreen}{rgb}{0.22,0.50,0.25}
\definecolor{lightblue}{rgb}{0.22,0.45,0.70}
\numberwithin{equation}{section}
\newtheorem{remark}{Remark}[section]

\def\b{\boldsymbol}
\def\bbR{\mathbb{R}}

\def\I{\int_{\Omega}}

\def\bH{\mathbf{H}}

\def\Hdiv{\mathbb{H}({\mathbf{div}},\Omega)}

\def\L2{{\mathbf{L}}^2(\Omega)}
\def\La{\mathbb{L}^2_{\mathrm{skew}}(\Omega)}

\def\HG{\mathrm{H}^1_{\Gamma}(\Omega)}
\def\HS{\mathrm{H}^1_{\Sigma}(\Omega)}

\newcommand{\norm}[1]{\left\lVert#1\right\rVert}

\def\HNdiv{\mathbb{H}_{\Sigma}(\mathbf{div},\Omega)}
\newcommand\rH{\mathrm{H}}

\newcommand\rL{\mathrm{L}}

\newcommand\bsigma{\boldsymbol{\sigma}}
\newcommand\btau{\boldsymbol{\tau}}
\newcommand\beps{\boldsymbol{\varepsilon}}

\newcommand{\bgamma}{\boldsymbol{\gamma}}
\newcommand{\brho}{\boldsymbol{\rho}}
\newcommand\bzeta{\boldsymbol{\zeta}}

\newcommand{\bdiv}{\operatorname*{\mathbf{div}}}
\newcommand{\vdiv}{\operatorname*{div}}
\newcommand{\tr}{\operatorname*{tr}}

\newcommand{\bu}{\boldsymbol{u}}

\newcommand{\bv}{\boldsymbol{v}}

\newcommand\ff{\boldsymbol{f}}
\renewcommand\gg{\boldsymbol{g}}
\newcommand\nn{\boldsymbol{n}}
\newcommand\bt{\boldsymbol{t}}
\newcommand\cero{\boldsymbol{0}}

\allowdisplaybreaks

\headers{Poroelasticity with stress-dependent diffusion}{G\'omez-Vargas, Mardal,  
Ruiz-Baier, Vinje}
\title{Twofold saddle-point formulation of Biot poroelasticity \\with stress-dependent diffusion\thanks{\textbf{Updated:} \today.\funding{BGV acknowledges support from Vicerrector\'ia de Investigaci\'on, through the project 540-C0-202, Sede de Occidente, Universidad de Costa Rica. KAM acknowledges support from the Research Council of Norway, grants 300305 and 
	301013. RRB acknowledges support from the Monash Mathematics Research Fund S05802-3951284 and from the Ministry of Science and Higher Education of the Russian Federation within the framework of state support for the creation and development of World-Class Research Centers ``Digital biodesign and personalised healthcare'' No. 075-15-2020-926.}}}

\author{
Bryan G\'omez-Vargas\thanks{Secci\'on de Matem\'atica, Sede de Occidente, Universidad de Costa Rica, San Ram\'on, Alajuela, Costa Rica (\email{bryan.gomezvargas@ucr.ac.cr}).}
\and
Kent-Andr\'e Mardal\thanks{Department of Mathematics, Division of Mechanics, University of Oslo, Norway; and Simula Research Laboratory, 1325 Lysaker, Norway (\email{kent-and@math.uio.no}).}
\and
Ricardo Ruiz-Baier\thanks{School of Mathematics, 
Monash University, 9 Rainforest Walk, Melbourne 3800 VIC,  Australia; and 
  Institute of Computer Science and Mathematical Modelling, Sechenov University, Moscow, Russian Federation; and Universidad Adventista de Chile, Casilla 7-D, Chill\'an, Chile 
  (\email{ricardo.ruizbaier@monash.edu}).} 
\and 
Vegard Vinje\thanks{Department for Numerical Analysis and Scientific Computing, Simula Research Laboratory, Lysaker, Norway (\email{vegard@simula.no}).}
}

\date{\today}
\begin{document}
\maketitle

\begin{abstract}
We introduce a stress/total-pressure formulation for poroelasticity that includes the coupling with steady nonlinear diffusion modified by stress. The nonlinear problem is written in mixed-primal form, coupling a perturbed twofold saddle-point system with an elliptic problem. The continuous formulation is analysed in the framework of abstract fixed-point theory and Fredholm alternative for compact operators. A mixed finite element method is proposed and its stability and convergence analysis are carried out. We also include a few illustrative numerical tests.  The resulting model can be used to study waste removal in the brain parenchyma, where diffusion of a tracer alone or a combination of advection and diffusion are not sufficient to explain the alterations in rates of filtration observed in porous media samples. 
\end{abstract}
\begin{keywords} Stress-altered diffusion; Poroelasticity; Perturbed saddle-point; Mixed finite elements; Brain multiphysics. 
\end{keywords}
\begin{AMS}
 65N60, 92C50.
\end{AMS}

%***********************************************************************************
%%%%%%%%%%%%%%%%%%%%%%
\section{Introduction}

Poroelastic structures are found in many applications of industrial and scientific relevance. Examples include the interaction between soft permeable tissue and blood flow, or the study of biofilm growth and distribution near fluids \cite{showalter05}. We are concerned with one particular application involving  the transport of cerebrospinal fluid (CSF) and tracer within the brain parenchyma, and how this can contribute to better explain the mechanisms that permit solutes from the brain interstitial space to move. These processes are key in eliminating waste from the brain and in maintaining a healthy function of the nervous system (see, for instance, the review \cite{pardridge16}), and which address important questions in the context of the overall glymphatic function \cite{iliff12}. 

% ...
Detailed maps of permeability and simulations using realistic hydrostatic pressure gradients and image-based reconstruction of intersitial space, suggest that diffusive effects in combination with advection along vasculature dominate over bulk flow \cite{holter17}. However it appears that (even heterogeneous) Fickian diffusion and convective effects only are not sufficient to explain the main features of the physiological propagation of solutes along the various elements of the brain structure \cite{croci19,valnes2020apparent}. We thus reconsider how other changes related to the mechanical state, such as the generation and nonuniform localisation of stresses in the brain, could impact the filtration properties at the spatio-temporal scales of interest for the application at hand. Pressure, strain, and stress changes  can affect the tissue microstructure by rearranging the local distribution of compliant arrays, opening and closing pores, as well as generating preferential paths for the accumulation of large solvent molecules \cite{zohdi99}. The combination of such effects can result in variations of averaged coefficients usually employed in classical models, such as diffusion (but also other parameters such as bulk moduli, relative permeability, and Biot-Willis coefficient \cite{ma17}). For this effect we propose a simplified phenomenological model for the steady interaction between poroelastic stress and diffusion of solutes (such as gadobutrol). The coupling of poromechanics and diffusion is a fundamental aspect not only for investigating the mechanical behaviour of the brain function, but also in the description and design of other complex systems such as tissue engineering scaffolds \cite{costanzo17}. 

Linear models of poroelasticity seem to be sufficient to capture the most important strain patterns when coupled to, e.g., immune system equations modelling the blood-brain barrier in the brain (see the comparison in \cite{lang16}). In the present model the linear poroelasticity equations are formulated using the stress as an unknown. This is of particular importance as this field will be central to the coupling with the diffusion process. The recent literature contains various formulations for Biot equations using poroelastic stress. For instance, including weakly symmetric four-field schemes \cite{baerland17}, adaptive mixed formulations \cite{elyes18}, multipoint stress-flux mixed methods \cite{ambar20}, and domain decomposition approaches \cite{jay21}. In contrast, here we employ a double saddle-point formulation designed to solve for total poroelastic stress, displacement, rotation, total pressure (drawing inspiration from \cite{lee17,oyarzua16} and generalising these formulations to include poroelastic and active stress), and fluid pressure. The Biot equations are then coupled with a nonlinear diffusion equation for the CSF tracer, and the discretisation proceeds with classical mixed-primal finite elements. Mixed-primal and fully mixed methods for stress-modified diffusion problems have been investigated in \cite{gatica18,gatica19}, and their analysis hinges on techniques using a separation of elastostatics and diffusion equations by means of a fixed-point scheme that requires additional assumptions on regularity. This is precisely the approach followed here. These models, as well as modification of transport properties in other soft tissues such as cardiac muscle \cite{propp20} allow to recover enhanced anisotropy and heterogeneity of {stress-assisted} diffusion. However, here the coupling with total poroelastic stress is used to \emph{reduce} or hinder diffusion patterns which have the potential to produce a much better match with experimental studies that indicate a decrease of effective diffusion by at least an order of magnitude with respect to the base state observed during sleep.  One of the main challenges of the work lies in deriving  a variational formulation involving total poroelastic stress, and  whose analysis does not need additional conditions on the Lam\'e parameters. The resulting form has the structure of a double saddle-point problem, where terms that explicitly involve $\lambda_s$, appear. Exploiting the compactness of these terms, an applying the Fredholm alternative together with a new extended version of the Babu\v ska-Brezzi theory, allow us to assert unique solvability robustly with respect to $\lambda_s$.

The content of this paper has been laid out as follows. Section~\ref{sec:model} provides details of the model, describing the components of the balance equations and stating boundary conditions. Section~\ref{sec:solvability} is devoted to the analysis of solvability of the weak form, using fixed-point arguments in conjunction with the Fredholm alternative and the Babu\v ska-Brezzi theory. We address in Section \ref{sec:fem} the solvability and stability analysis of the discrete problem. A priori error estimates are derived in Section~\ref{error_analysis},  and in Section~\ref{sec:results} we collect   computational results, consisting in verification of  convergence and simulation of different cases on simple geometries. 
%We close with some remarks and model extensions in Section~\ref{sec:concl}. 

%%%%%%%%%%%%%%%%%%%%%%
\section{Governing equations}\label{sec:model}
Let us consider a Lipschitz polyhedral domain $\Omega\subset \bbR^d$, $d=2,3$ occupied by the brain parenchyma, here modelled as a mixture consisting of an elastic structure and an interconnected porous space fully saturated by interstitial fluid. We will denote by $\nn$ the outward unit normal vector on the boundary $\partial\Omega$. We also define the boundaries $\Gamma$ and $\Sigma$ where different types of boundary conditions will be applied. We assume that the parenchyma is  isotropic, that the deformation process is stationary, and that its constitutive stress-strain relation is linear. As usual in the framework of linear poroelasticity, the solid-fluid mixture is described in terms of the solid displacement $\bu:\Omega\to\mathbb{R}^d$ accounting for the movement of the material from its initial undeformed configuration, and the pore fluid pressure $p:\Omega\to\mathbb{R}$ (which, through Darcy's law, also defines the specific discharge or filtration velocity). The Biot system states momentum and mass balances 
\begin{equation}
   -\vdiv(\b \sigma)  = \rho_s \ff   \quad \mathrm{and}\quad 
    c_0p+\alpha\vdiv(\b u) - \vdiv \left(\frac{\mathbf{\kappa}}{\mu_f}\nabla p
    - \rho_f \boldsymbol{g}\right) = m \qquad \text{in $\Omega$,}   \label{poro1}
\end{equation}
respectively, where 
$\ff:\Omega \to \mathbb{R}^d$ is the vector field of body loads and 
$m:\Omega\to \mathbb{R}$ is a sources/sink of fluid. 
We also consider the presence of a CSF tracer within the poroelastic parenchymal domain. We denote its concentration by $\omega:\Omega\to\mathbb{R}$ and its movement in the parenchyma is governed by 
\begin{equation}\label{eq:w0}
\phi \omega - \vdiv(D(\bsigma) \nabla \omega) = \phi \ell \qquad \text{in $\Omega$},\end{equation}
where $\bsigma= 2\mu_s\beps(\bu) + (\lambda_s \mathrm{div}(\b u)- \alpha p - \beta \omega) \mathbb{I}$ is the poroelastic Cauchy stress tensor with $\beps(\bu):=\frac{1}{2}(\nabla \bu + \nabla\bu^{\tt t})$ denoting the infinitesimal strain tensor, the Lam\'e constants $\mu_s,\lambda_s>0$ for homogeneous and isotropic materials, the Biot-Willis constant $0<\alpha\leq 1$, the $d\times d$ identity matrix $\mathbb{I}$, and $\beta$ a scaling of active stress that indicates a two-way coupling between diffusion and motion. The right-hand side $\ell:\Omega\to \mathbb{R}$ is a non-negative drainage or source coefficient that encodes absorption from the lymph nodes and/or from capillaries \cite{croci20}, and $\phi$ is porosity.  In \eqref{poro1} $c_0$ is the mass storativity coefficient, and $\kappa$ is the permeability tensor (symmetric, uniformly positive definite, and bounded). The tensor function  $D:\mathbb{R}^{d\times d} \rightarrow \mathbb{R}^{d\times d}$ is a stress-dependent diffusivity accounting for an altered diffusion acting in the poroelastic domain. This term may assume the simple form
 \begin{equation}\label{eq:D0}
D(\bsigma) = \eta_0 D_0 + \eta_1 \exp(-\eta_1 \tr\bsigma),
 \end{equation}
with $D_0$ being the effective diffusion of typical solutes such as potassium ions, and $\eta_0,\eta_1$ modulating the intensity of the stress-altered diffusion. Other constitutive equations (including power-law, polynomial, and anisotropic forms) can be found in, e.g., \cite{gatica18}, whereas pressure-dependent diffusion models can be found in \cite{grigoreva19}.  Alternative constitutive relations could be given by a dependence on the volumetric part of the effective poroelastic stress, or on the porosity (which in the linear regime is the total amount of fluid) such as $D = D_0 - \eta( c_0 p + \alpha \vdiv\bu)$.  
On the boundary $\Gamma$ we prescribe tracer concentration, displacement and normal specific discharge;   whereas on $\Sigma$ we fix zero normal total stress, zero flux for the tracer, and prescribed fluid pressure. 

Besides the main variables we also employ  the total pressure (the volumetric contributions to the poroelastic Cauchy stress) $\widetilde{p}:=\alpha p-\lambda_s\,\mathrm{div}\,\b u$, the poroelastic stress tensor $\b \sigma=2\mu_s\b \varepsilon(\b u)-\widetilde{p}\mathbb{I}-\beta\omega \mathbb{I}$, and the tensor $\b\gamma$ (rotation Lagrange multiplier assisting the weakly imposition of the symmetry of   $\b\sigma$), and the strain   $\b t:=\b \varepsilon(\b u)=\nabla \b u-\b \gamma$. Therefore the governing equations consist in finding 
%***********************************************************************************
 $\bsigma$,  $\brho$,   $\bt$, 
  $\tilde p$,   $p$, and   $\omega$, satisfying
\begin{subequations}
\label{eq:model}
\begin{align}
\bsigma  = 2\mu_s \bt - \tilde{p} \mathbb{I} - \beta\omega\mathbb{I} \quad \text{in } \Omega,\quad \bsigma  = \bsigma^{\tt t}  \quad \text{in } \Omega, \quad  
- \bdiv \bsigma  = \rho_s \ff \quad \text{in } \Omega, \quad 
 \bt  = \beps(\bu) \quad \text{in } \Omega,\label{eq:sigma}\\
\quad  \bgamma = \nabla\bu - \bt \quad \text{in } \Omega,\quad  \left(c_0 + \frac{\alpha^2}{\lambda_s}\right)p - \frac{\alpha}{\lambda_s} \tilde{p} - \vdiv \left(\frac{\mathbf{\kappa}}{\mu_f}\nabla p
- \rho_f \boldsymbol{g}\right)  = m \quad \text{in } \Omega, \\
 \quad 
 \frac{1}{\lambda_s} \tilde{p} + \tr\bt - \frac{\alpha}{\lambda_s}  p   = 0 \quad \text{in } \Omega,\quad \bu  = \cero, \quad \frac{\kappa}{\mu_f}\nabla p\cdot\nn =0 \quad \text{on } \Gamma, \quad
 \bsigma\nn  = \cero,  \quad p=0 \quad \text{on } \Sigma,\label{eq:tildep}\\
\phi \omega -\vdiv( D(\bsigma) \nabla \omega)  = \phi\ell  \quad \text{in } \Omega,   \quad \omega  = 0  \quad \text{on } \Gamma,\quad  \text{and} \quad D(\bsigma)\nabla \omega\cdot\nn   = 0 \quad \text{on } \Sigma.\label{bc:Sigma}
\end{align}\end{subequations}
%where  
%\eqref{eq:sigma} 
%is the constitutive equation for the Cauchy stress. 
The second and third equations in \eqref{eq:sigma} state, respectively, the conservation of angular and linear momentum, while the last equation in \eqref{eq:sigma} and the first equality in \eqref{eq:tildep} encode other constitutive relations. 
Homogeneous boundary data have been assumed only for sake of conciseness, but they can be generalised to the non-homogeneous case by using classical lifting arguments. An advantage of using this modified form of the system is that all equations are robust with respect to the Lam\'e parameters. 
 
%***********************************************************************************
\section{Continuous mixed-primal formulation and well-posedness analysis}\label{sec:solvability}
%***********************************************************************************
Let us define the tensor functional spaces for stress and rotations 
\[\HNdiv:=\{\btau \in \Hdiv: \btau \nn =\cero \ \mathrm{on}\, \Sigma\}, \  
\La:=\{\b \eta \in \mathbb{L}^2(\Omega): \b \eta + \b \eta^{\mathrm{t}}\!=\cero\}.\]
A weak formulation for the nonlinearly coupled problem \eqref{eq:sigma}-\eqref{bc:Sigma} can be derived as usual, and it consists in finding $\bigl(\bt,(\bsigma,\tilde{p}),(\bu,\bgamma),p,\omega\bigr) \in \mathbb{L}^2(\Omega)\times [\HNdiv\times \mathrm{L}^2(\Omega)]\times[\L2\times\La] \times \HS \times \HG$ such that 
\begin{subequations}
\begin{align}
-\bigl(c_0+\frac{\alpha^2}{\lambda_s}\bigr)\int_{\Omega}p q +\frac{\alpha }{\lambda_s}\int_{\Omega}\widetilde{p}q  - \frac{1}{\mu_f}\int_{\Omega}\boldsymbol{\kappa} \nabla p\cdot \nabla q =  \rho_f\int_{\Omega}\boldsymbol{\kappa}\b g\cdot\nabla q -\int_{\Omega}mq \quad   \forall \, q \in \HS, \label{eq:poro-equil}\\
\label{eq:strain}
-\int_{\Omega}\b t\colon \b \tau-\int_{\Omega}\b u\cdot \mathbf{div}\,\b \tau-\int_{\Omega}\b \gamma\colon\b \tau =0  \;\;\; \forall\b  \tau \!\in \mathbb{H}_{\Sigma}(\mathbf{div};\Omega),\quad -\int_{\Omega} \b v \cdot\mathbf{div}\,\b \sigma  = \rho_s\int_{\Omega} \ff\cdot \b v  \quad  \forall \, \b v \in \mathbf{L}^2(\Omega),\\
\label{eq:trace}
-\frac{1}{\lambda_s}\int_{\Omega}\widetilde{p}\,\widetilde{q}-\int_{\Omega}\mathrm{tr}(\b t)\widetilde{q}+\frac{\alpha}{\lambda_s}\int_{\Omega}p\widetilde{q}  =0\quad   \forall\, \widetilde{q} \in \mathrm{L}^2(\Omega),\quad 
-\int_{\Omega}\b \sigma\colon\b \eta =0 \quad    \forall\,\b \eta \in \mathbb{L}_{\mathrm{skew}}^2(\Omega),\\
\label{eq:constitutive-poro}
2\mu_s \int_{\Omega}\b t\colon\b r - \int_{\Omega} \b \sigma\colon\b r -\int_{\Omega} \widetilde{p}\,\mathrm{tr}(\b r)  = \beta\int_{\Omega}\omega \mathrm{tr}(\b r)\quad   \forall \, \b r \in \mathbb{L}^2(\Omega),\\
\label{eq:diff}
\int_{\Omega}\phi\omega\theta+\int_{\Omega} D(\b \sigma)\nabla \omega \cdot \nabla \theta =  \int_{\Omega}\ell \phi \theta \quad   \forall \, \theta \in \mathrm{H}^1_\Gamma(\Omega).
\end{align}
\end{subequations}

We further  assume that $\b\kappa$ is uniformly bounded and positive definite, and that  $D(\bsigma)$ is of class $C^1$, uniformly positive definite,  Lipschitz continuous and that satisfies a particular bounding property: There exist positive constants $\kappa_1, \kappa_2$, and $d_1, d_2, d_3$, such that
	\begin{equation}
\label{prop-D}	\kappa_1|\bv|^2 \leq \bv^{\mathrm{t}}\kappa(\cdot)\bv \leq \kappa_2|\bv|^2,\quad \mathrm{and}\quad   d_1|\bv|^2 \leq \bv^{\mathrm{t}}D(\cdot)\bv \leq d_2 |\bv|^2, \quad 
	|D(\bsigma)-D(\bzeta)|\leq d_3|\bsigma - \bzeta|, \end{equation}
for all $ \bv \in\mathbb{R}^d$, and for all $\bsigma,\bzeta\in\bbR^{d\times d}$.
Next, we regroup unknowns and spaces in \eqref{eq:poro-equil}-\eqref{eq:constitutive-poro}   as follows: 
\begin{equation*}\vec{\b \sigma}:=(\b \sigma, \widetilde{p}),\;\; \vec{\b \tau}:=(\b \tau, \widetilde{q}),\;\; \vec{\b u}:=(\b u, \b \gamma), \quad \vec{\b v}:=(\b v, \b \eta),\;\; \mathrm{and}\quad \mathbf{X}:= \mathbb{H}_{\Sigma}(\mathbf{div};\Omega)\times \mathrm{L}^2(\Omega),\;\; \mathbf{M}:=\mathbf{L}^2(\Omega)\times \mathbb{L}^2_\mathrm{skew}(\Omega).
\end{equation*}
Then, given $(H_{\omega},F,G,J)\in \mathbb{L}^2(\Omega)'\times\mathbf{M}'\times \mathrm{H}^1_{\Sigma}(\Omega)'\times \mathrm{H}^1_\Gamma(\Omega)'$, we seek $(\b t, \vec{\b \sigma},\vec{\b u},p, \omega)\in \mathbb{L}^2(\Omega)\times\mathbf{X}\times\mathbf{M}\times\mathrm{H}^1_{\Sigma}(\Omega)\times \mathrm{H}^1_{\Gamma}(\Omega)$ such that
\begin{subequations}
\begin{alignat}{11}
&[A(\b t), \b r]&+\quad&[B^*(\vec{\b \sigma}),\b r]&&&&&=& \;\;[H_{\omega},\b r],&	\label{VF-poro_1}\\
&[B(\b t), \vec{\b \tau}]\;\;\,&-\quad& [C(\vec{\b \sigma}),\vec{\b \tau}]\;\;\,&+&\;\; [B^*_1(\vec{\b u}),\vec{\b \tau}]\;\;&+&\;\;[B_2^*(p),\vec{\b \tau}]\;&=&\;\;\mathrm{O},&\label{VF-poro_2}\\
& &&[B_1(\vec{\b \sigma}),\vec{\b v}]&&&&&=&\;\;[F,\vec{\b v}],&\label{VF-poro_3}\\
&&&[B_2(\vec{\b \sigma}),q]&&&-&\;\;[D(p),q]&=& \;\;[G,q],& \label{VF-poro_4}\\
&&&&&&&[\mathscr{A}_{\b \sigma}(\omega),\theta]&=&\;\; [J,\theta],& \label{VF_diff}
\end{alignat}\end{subequations}
for all $(\b r,\vec{\b \tau},\vec{\b v},q,\theta )\in \mathbb{L}^2(\Omega)\times\mathbf{X}\times \mathbf{M}\times \mathrm{H}^1_{\Sigma}(\Omega)\times \mathrm{H}^1_{\Gamma}(\Omega)$, where 
  the linear bounded operators $A:\mathbb{L}^2(\Omega)\rightarrow \mathbb{L}^2(\Omega)'$, $C:\mathbf{X}\rightarrow \mathbf{X}'$, $B: \mathbb{L}^2(\Omega)\rightarrow \mathbf{X}'$, $B_1:\mathbf{X}\rightarrow \mathbf{M}'$, $B_2:\mathbf{X}\rightarrow \mathrm{H}^1_{\Sigma}(\Omega)',$ $D:\mathrm{H}^1_{\Sigma}(\Omega)\rightarrow\mathrm{H}^1_{\Sigma}(\Omega)'$, $B^*:\mathbf{X}\rightarrow\mathbb{L}^2(\Omega)'$, $B_1^*:\mathbf{M}\rightarrow\mathbf{X}'$, $B_2^*:\mathrm{H}^1_{\Sigma}(\Omega)\rightarrow \mathbf{X}'$, and $\mathscr{A}_{\b \sigma}:\mathrm{H}^1_\Gamma(\Omega)\rightarrow \mathrm{H}^1_\Gamma(\Omega)'$ are defined as
%the operators and functionals are given by
\begin{subequations}
\begin{align}
[A({\b t}),{\b r}]:=2\mu_s \int_{\Omega}\b t\colon\b r, \quad 
&[B({\b r}),\vec{\b \tau}]:= - \int_{\Omega} \b \tau\colon\b r -\int_{\Omega} \widetilde{q}\,\mathrm{tr}(\b r),\quad [B_1(\vec{\b \tau}),\vec{\b v}]:= -\int_{\Omega}\b v\cdot \mathbf{div}\,\b \tau-\int_{\Omega}\b \eta\colon\b \tau,\label{def-A-B}\\
[C(\vec{\b \sigma}),\vec{\b \tau}]:= \frac{1 }{\lambda_s}\int_{\Omega}\widetilde{p}\,\widetilde{q},&\quad 
[B_2(\vec{\b \tau}),q]:=\frac{\alpha }{\lambda_s}\int_{\Omega}\widetilde{q}q, \quad 
[D(p),q]:= \left(c_0+\frac{\alpha^2}{\lambda_s}\right)\int_{\Omega}p q + \frac{1}{\mu_f}\int_{\Omega}\boldsymbol{\kappa} \nabla p\cdot \nabla q,\label{def-C-B2-D}\\
[\mathscr{A}_{\b \sigma}(\omega),\theta]:=&\int_{\Omega}\phi\omega\theta+\int_{\Omega} D(\b \sigma)\nabla \omega \cdot \nabla \theta,\quad %\label{def-A2}\\
[H_\omega,\b r]:= \beta\int_{\Omega}\omega \mathrm{tr}(\b r),\label{def-A-H}\\ 
[F,\vec{\b v}]:=&\rho_s\int_{\Omega} \ff\cdot \b v, \quad %\label{def-F}\\
[G,q]:=-\int_{\Omega}mq + \rho_f \int_{\Omega}\boldsymbol{\kappa} \b g\cdot\nabla q, \quad %\label{def-G}\\
[J,\theta]:= \int_{\Omega} \ell\phi\theta,\label{def-F-G-G1}
\end{align}
\end{subequations}
for each ${\b t}, {\b r} \in \mathbb{L}^2(\Omega), \;\vec{\b \sigma},\vec{\b \tau}\in \mathbf{X}$, $\vec{\b u}, \vec{\b v} \in \mathbf{M},$ $p,q\in \mathrm{H}^1_\Sigma(\Omega)$ and $\omega, \theta\in \mathrm{H}^1_\Gamma(\Omega).$ The symbol $\mathrm{O}$ stands here for the null functional, and $[\cdot,\cdot]$ denotes the %corresponding 
duality pairing induced by the operators and functionals.

%%%%%%%%%%%%%%%%%%%%%%%%%%%%%%%%%%%
\subsection{Fixed-point-strategy}\label{section-a-fixed-point-strategy} In this section, we utilise a fixed-point strategy to prove that (\ref{VF-poro_1})-\eqref{VF_diff} is well-posed. Let $\textbf{S}:\mathrm{H}_{\Gamma}^1(\Omega)\to \mathbb{L}^2(\Omega)\times\mathbf{X}\times\mathbf{M}\times\mathrm{H}^1_{\Sigma}(\Omega)$ be the operator defined by 
\begin{equation*}
\textbf{S}(\omega) := (\textbf{S}_1(\omega), (\textbf{S}_2(\omega),\textbf{S}_3(\omega)),(\textbf{S}_4(\omega),\textbf{S}_5(\omega)),\textbf{S}_6(\omega))=(\b t, (\b \sigma, \widetilde{p}),(\b u,\b \gamma),p)\in \mathbb{L}^2(\Omega)\times\mathbf{X}\times\mathbf{M}\times\mathrm{H}^1_{\Sigma}(\Omega),\end{equation*}
where $(\b t, (\b \sigma, \widetilde{p}),(\b u,\b \gamma),p)$ is the unique solution of (\ref{VF-poro_1})-\eqref{VF-poro_4} with $\omega$ given. In turn, we let ${\widetilde{\mathbf S}} : \mathbb{H}_{\Sigma}(\mathbf{div},\Omega)
\to \mathrm{H}^1_{\Gamma}(\Omega)$ be the operator 
\[
{\widetilde{\mathbf S}}(\b \sigma) \,:= \, \omega 
\quad\forall\,\b \sigma\,\in\, \mathbb{H}_{\Sigma}(\mathbf{div},\Omega),\]
where $\omega$ is the unique solution of (\ref{VF_diff}) with $\b \sigma$ given. Then, we define the operator $\mbox{\textbf{{T}}}:\mathrm{H}^1_{\Gamma}(\Omega)\to \mathrm{H}^1_{\Gamma}(\Omega)$ as
\[
\mbox{\textbf{{T}}}(\omega):=\widetilde{\mathbf S}(\mathbf S_2(\omega))\quad\forall\,\omega \in \mathrm{H}^1_{\Gamma}(\Omega), 
\]
and realise that solving (\ref{VF-poro_1})-\eqref{VF_diff} is equivalent to finding  $\omega\in \mathrm{H}^1_{\Gamma}(\Omega)$ such that 
\begin{equation}\label{def-T}
\mbox{\textbf{{T}}}(\omega)=\omega. 
\end{equation}
\subsection{Well-posedness of the uncoupled problems}\label{well-posedness-analysis}
In this section, we analyse the solvability of problems defining $\mathbf{S}$ and $\widetilde{\mathbf{S}}$. In view of the analysis of Section \ref{error_analysis} and Lemma \ref{lemma_main} below, the well-posedness and stability   of the uncoupled problem related with the operator $\textbf{S}$ for a given $\omega \in \mathrm{H}^1_\Gamma(\Omega)$, will be developed considering generic functionals $\widetilde{H}_\omega, \widetilde{F}_1, \widetilde{F} $ and $\widetilde{G}$, instead of $H_\omega, \mathrm{O}, F$ and $G$, respectively, in \eqref{VF-poro_1}-\eqref{VF-poro_4}. 
 We begin by investigating the well-posedness of this uncoupled problem. First, we recall from \cite{brezzi91} the decomposition 
\[\mathbb{H}(\mathbf{div},\Omega)=\mathbb{H}_0(\mathbf{div},\Omega)\oplus\mathbb{R}\mathbb{I},\quad \text{with}\quad 
\mathbb{H}_0(\mathbf{div},\Omega):=\left\{\b \tau\in \mathbb{H}(\mathbf{div},\Omega):\ \int_{\Omega}\mathrm{tr}(\b \tau)=0\right\}.\]
That is, for each $\b \tau\in \mathbb{H}(\mathbf{div},\Omega)$ there exist unique 
\begin{equation}\label{eq:deco}
\b \tau_0:=\b \tau-{\displaystyle\left\{\frac{1}{d|\Omega|}\int_{\Omega} \mathrm{tr}(\b \tau)\right\}}\mathbb{I}\in \mathbb{H}_0(\mathbf{div},\Omega) \quad \text{and} \quad 
\hat{d}:=\frac{1}{d|\Omega|}\int_{\Omega} \mathrm{tr}(\b \tau)\in \mathbb{R},\end{equation}
such that $\b \tau=\b \tau_0+\hat{d}\mathbb{I}$. Moreover, we recall the following results which will be useful in the forthcoming analysis. We remit to \cite[Lemma 2.3]{gatica06}, and \cite[Lemma 2.4]{gatica06}, respectively, for further details.
\begin{lemma}\label{lem:inequality}
	There exists $c_1>0$, depending only on $\Omega$, such that 
	\begin{equation}\label{inequality_1}
	c_1\Vert\b \tau_0\Vert^2_{\mathbb{L}^2(\Omega)}\leq \Vert\b \tau^{\mathrm{d}}\Vert^2_{\mathbb{L}^2(\Omega)}+\Vert\mathbf{div}(\b \tau)\Vert^2_{\mathbb{L}^2(\Omega)}\quad \forall\, \b \tau \in \mathbb{H}(\mathbf{div},\Omega).	
	\end{equation}
\end{lemma} 
\begin{lemma}
	There exists $c_2>0$, depending only on $\Sigma$ and $\Omega$, such that 
	\begin{equation}\label{inequality_2}
	c_2\Vert\b \tau\Vert^2_{\mathbb{H}(\mathbf{div},\Omega)}\leq \Vert\b \tau_0\Vert^2_{\mathbb{H}(\mathbf{div},\Omega)}\quad \forall\, \b \tau \in \mathbb{H}_{\Sigma}(\mathbf{div},\Omega).	
	\end{equation}
\end{lemma}
On the other hand, we can notice that the kernel of ${B}_1$ is given by  
\begin{equation*}%\label{kernel_B1}
	\mathbb{V}:=\left\{\vec{\b \tau}\in \mathbf{X}: [B_1(\vec{\b \tau}),\vec{\b v}]=0\ \forall\,\vec{\b v}\in \mathbf{M}\right\} =\left\{\vec{\b \tau}\in \mathbf{X}: \mathbf{div}\,{\b \tau}=0 \ \mathrm{and} \ {\b \tau}={\b \tau}^{\mathrm{t}}\  \mathrm{in }\ \Omega\right\}.
\end{equation*}
The following two results will serve to establish the inf-sup conditions for $B$ and $B_1$ (cf. second and third equations in \eqref{def-A-B}).

\begin{lemma}\label{lem_inf_sup_B1}
	There exists ${\beta_1}>0$, such that 
	\begin{equation*}%\label{inf_sup_B1}\quad 
	\sup_{\vec{\b \tau}\in \mathbf{X} \setminus\{\b 0\}} \frac{ [B_1(\vec{\b \tau}),\vec{\b v}]}{\norm{\vec{\b \tau}}_{\mathbf{X}}}\geq \hat{\beta}_1\norm{\vec{\b v}}_{\mathbf{M}}\quad \forall\, \vec{\b v}\in  \mathbf{M}.
	\end{equation*}
\end{lemma}
\begin{proof}
	The (equivalent) surjectivity of  $B_1$  follows as a slight modification to the proof of \cite[Section 2.4.3]{gatica14}. 
\end{proof}
\begin{lemma}\label{inf_sup_B}
	There exists ${\beta}>0,$ such that 
	\begin{equation}\label{ineq_inf_sup_B}\quad \sup_{{\b r}\in \mathbb{L}^{2}(\Omega) \setminus\{\b 0\}} \frac{ [B({\b r}),\vec{\b \tau}]}{\norm{{\b r}}_{\mathbb{L}^2(\Omega)}}\geq \hat{\beta}\norm{\vec{\b \tau}}_{\mathbf{X}}\quad \forall\, \vec{\b \tau}\in  \mathbb{V}.
	\end{equation}
\end{lemma}
\begin{proof} It follows by using \eqref{inequality_1} and \eqref{inequality_2}. For more details see \cite[Lemma 2.5]{gatica06}.
\end{proof}
Before analysing \eqref{VF-poro_1}-\eqref{VF-poro_4}, we study the  reduced problem: For each $\omega\in \mathrm{H}^1_\Gamma(\Omega)$, find $(\hat{\b t}, \hat{\vec{\b \sigma}},\hat{p})\in \mathbb{L}^2(\Omega)\times\mathbb{V}\times\mathrm{H}^1_{\Sigma}(\Omega)$ such that 
\begin{subequations}
	\begin{alignat}{11}
	&[A(\hat{\b t}), \b r]&+\quad&[B^*(\hat{\vec{\b \sigma}}),\b r]&&&&&=& \;\;[\widetilde{H}_\omega,\b r],&	\label{VF-poro_1_reduced}\\
	&[B(\hat{\b t}), \vec{\b \tau}]\;\;\,&-\quad& [C(\hat{\vec{\b \sigma}}),\vec{\b \tau}]\;\;\,&&\;&+&\;\; [B_2^*(\hat{p}),\vec{\b \tau}]\;&=&\;\;[\widetilde{F}_1,\vec{\b \tau}],&\label{VF-poro_2_reduced}\\
	&&&[B_2(\hat{\vec{\b \sigma}}),q]&&&-&\;\;[D(\hat{p}),q]&=& \;\;[\widetilde{G},q],& \label{VF-poro_3_reduced}
	\end{alignat}\end{subequations}
for all $(\b r,\vec{\b \tau},q)\in \mathbb{L}^2(\Omega)\times\mathbb{V}\times \mathrm{H}^1_{\Sigma}(\Omega)$, where $\hat{\vec{\b \sigma}}:=(\hat{\b \sigma}, \hat{\widetilde{p}})$, and $\widetilde{H}_\omega, \widetilde{F}_1$ and $\widetilde{G}$ are the generic functionals mentioned at the beginning of Section \ref{well-posedness-analysis}. 
The unique solvability  of \eqref{VF-poro_1_reduced}-\eqref{VF-poro_3_reduced} proceeds using Fredholm's alternative. Let us recast   \eqref{VF-poro_1_reduced}-\eqref{VF-poro_3_reduced} with $\omega$ given, as the following equivalent operator problem: find $\vec{\b t}:=(\hat{\b t}, \hat{\vec{\b \sigma}}, \hat{p})\in \mathbf{H}:=\mathbb{L}^2(\Omega)\times\mathbb{V}\times\mathrm{H}^1_{\Sigma}(\Omega)$ such that 
\begin{equation}\label{operator_problem}
(\mathcal{S}+\mathcal{T})\vec{\b t}=\mathcal{F},
\end{equation}
where the linear operators $\mathcal{S},\mathcal{T}:\mathbf{H}\to \mathbf{H}^*$, and $\mathcal{F}\in \mathbf{H}^*$ are defined, for all $\vec{\b r}:=(\b r, \vec{\b \tau}, q)\in \mathbf{H}$,  as
\begin{gather*}
[\mathcal{S}(\vec{\b t}),\vec{\b r}]:=[A(\hat{\b t}), \b r]+[B^*(\hat{\vec{\b \sigma}}),\b r ]-[B(\hat{\b t}), \vec{\b \tau}]+[C( \hat{\vec{\b \sigma}}),\vec{\b \tau}]+[D(\hat{p}),q],\\
[\mathcal{T}(\vec{\b t}),\vec{\b r}]:=-[B_2^*(\hat{p}),\vec{\b \tau}]-[B_2(\hat{\vec{\b \sigma}}),q],\qquad 
[\mathcal{F},\vec{\b r}]:=[\widetilde{H}_\omega,r]-[\widetilde{F}_1,\vec{\b \tau}]-[\widetilde{G},q].\end{gather*}

The three upcoming lemmas establish that $\mathcal{S}$ is invertible,   $\mathcal{T}$ is compact, and  $\mathcal{S}+\mathcal{T}$ is injective.  Fredholm's theory will yield the well-posedness of  \eqref{operator_problem} and, equivalently, that of \eqref{VF-poro_1_reduced}-\eqref{VF-poro_3_reduced}, with $\omega\in \mathrm{H}_\Gamma^1(\Omega)$ given.
\begin{lemma}\label{F-invert}
	The operator $\mathcal{S}:\mathbf{H}\to \mathbf{H}^*$ is invertible.
\end{lemma}
Asserting the invertibility of $\mathcal{S}$ is equivalent to proving the unique solvability of the following uncoupled problems: find $(\hat{\b t}, \hat{\vec{\b \sigma}})\in \mathbb{L}^2(\Omega)\times \mathbb{V}$ such that 
\begin{subequations}
\begin{alignat}{11}
&[A(\hat{\b t}), \b r]&+\quad&[B^*(\hat{\vec{\b \sigma}}), \b r]&&&&&=& \;\;[F_{\omega},\b r]&\quad  \forall\, \b r\in \mathbb{L}^2(\Omega),\label{VF_fredholm_1}\\
&[B(\hat{\b t}), \vec{\b \tau}]\;\;\,&-\quad& [C(\hat{\vec{\b \sigma}}),\vec{\b \tau}]\;\;\,&&\; &&&=&\;\;[F_{\mathbf{X}},\vec{\b \tau}]& \forall\, \vec{\b \tau} \in \mathbb{V},\label{VF_fredholm_2}
\end{alignat}\end{subequations}
and: find $\hat{p}\in \mathrm{H}^1_\Sigma(\Omega)$ such that \begin{equation}\label{VF_Fredholm_p}
[D(\hat{p}),q]= [F_{H},q]\quad \forall\,q\in \mathrm{H}^1_{\Sigma}(\Omega),
\end{equation}
where, defining $\mathcal{F}:=(\mathcal{F}_\omega,\mathcal{F}_{\mathbf{X}},\mathcal{F}_H)$, the functionals $F_\omega, F_{\mathbf{X}}$, and $F_H$ are the ones induced by the operators $\mathcal{F}_\omega,\mathcal{F}_{\mathbf{X}},\mathcal{F}_H$, respectively.
The unique solution of   \eqref{VF_Fredholm_p} simply follows by Lax-Milgram's lemma. In turn, for   \eqref{VF_fredholm_1}-\eqref{VF_fredholm_2}, we verify the conditions given by the linear version of \cite[Lemma 2.1]{gatica03}. We first note that
\begin{equation}
[A(\b r),\b r]=2\mu_s \Vert\b r\Vert^2_{\mathbb{L}^2(\Omega)}\quad \forall\, \b r\in \mathbb{L}^2(\Omega), \label{ellipticity_TSP}
\end{equation} 
which means that the operator $A$ is elliptic on $\mathbb{L}^2(\Omega)$. Moreover, we observe that
\begin{equation}
[C(\vec{\b \tau}),\vec{\b \tau}]=\frac{1}{\lambda_s}\Vert \widetilde{q}\Vert^2_{\mathrm{L}^2(\Omega)}\geq 0\quad \forall\,\vec{\b \tau}\in \mathbb{V},\label{semi_definite_C}
\end{equation} which shows that the operator $C$ is semi positive definite on $\mathbb{V}$. In this way, having in mind \eqref{ellipticity_TSP}, \eqref{semi_definite_C}, and the inf-sup condition for $B$ given by Lemma \ref{inf_sup_B}, we simply apply \cite[Lemma 2.1]{gatica03}, and obtain the desired result. 
\begin{lemma}\label{F-compact}
	The operator $\mathcal{T}:\mathbf{H}\to \mathbf{H}^*$ is compact.
\end{lemma}
\begin{proof}
	Let us to define the operator $\mathbb{B}:\mathrm{L}^2(\Omega)\to \mathrm{H}^1_{\Sigma}(\Omega)$ as
	\[\langle \mathbb{B}(\widetilde{q}),q\rangle_{\mathrm{L}^2(\Omega)}:=\frac{\alpha}{\lambda_s}\int_{\Omega}\widetilde{q}q\quad \forall\,\widetilde{q}\in \mathrm{L}^2(\Omega),\;\forall\,q\in  \mathrm{H}^1_{\Sigma}(\Omega),\] 
	where $\langle\cdot,\cdot\rangle_{\mathrm{L}^2(\Omega)}$ stands for the ${\mathrm{L}^2(\Omega)}$-inner product. Then, thanks to the compactness of the adjoint operator $\mathbb{B}^*$ (see, \cite[Lemma 2.2]{oyarzua16}), we deduce that the following operator is also compact
	\[\mathcal{T}(\vec{\b t})=(\mathbf{0},(\b 0,\mathbb{B}(\hat{\widetilde{p}})),\mathbb{B}^*(\hat{p})).\]
\end{proof}
\begin{lemma}\label{F-injective}
	The operator $(\mathcal{S}+\mathcal{T}):\mathbf{H}\to \mathbf{H}^*$ is injective. 
\end{lemma}
\begin{proof}
	It is sufficient to show that the unique solution of the homogeneous problem of \eqref{VF-poro_1_reduced}-\eqref{VF-poro_3_reduced}
%\begin{subequations}
%	\begin{alignat}{11}
%		&[A(\hat{\b t}), \b r]&+\quad&[B^*(\hat{\vec{\b \sigma}}),\b r]&&&&&=& \;\;\mathrm{O},&	\label{VF-poro_1_injective}\\
%		&[B(\hat{\b t}), \vec{\b \tau}]\;\;\,&-\quad& [C(\hat{\vec{\b \sigma}}),\vec{\b \tau}]\;\;\,&&\;&+&\;\; [B_2^*(\hat{p}),\vec{\b \tau}]\;&=&\;\;\mathrm{O},&\label{VF-poro_2_injective}\\
%		&&&[B_2(\hat{\vec{\b \sigma}}),q]&&&-&\;\;[D(\hat{p}),q]&=& \;\;\mathrm{O},& \label{VF-poro_3_injective}
%\end{alignat}\end{subequations}
	is the null-operator in the space $\mathbf{H}.$ Thus, we begin by considering $\widetilde{H}_\omega=\widetilde{F}_1=\widetilde{G}=0$ in \eqref{VF-poro_1_reduced}-\eqref{VF-poro_3_reduced}, and then, taking $\b r=\hat{\b t}$, and $\vec{\b \tau}=\hat{\vec{\b \sigma}}$, in \eqref{VF-poro_1_reduced}, and \eqref{VF-poro_2_reduced} respectively, and applying some computations, we obtain
	\begin{equation}\label{bound_t_injective}
		2\mu_s\Vert \hat{\b t}\Vert^2_{\mathbb{L}^2(\Omega)}+\frac{1}{\lambda_s}\Vert \hat{\widetilde{p}}\Vert^2_{\mathrm{L}^2(\Omega)}-\frac{\alpha}{\lambda_s}\Vert\hat{\widetilde{p}}\Vert_{\mathrm{L}^2(\Omega)}\Vert \hat{p}\Vert_{\mathrm{H}^1(\Omega)}\leq 0.
	\end{equation}
	In turn, choosing $q=\hat{p}$ in \eqref{VF-poro_3_reduced}, we have
	\begin{equation}\label{bound_p_injective}
		\mathrm{min}\left\{c_0,\frac{\kappa_1}{\mu_f}\right\}\Vert \hat{p}\Vert^2_{\mathrm{H}^1(\Omega)}+\frac{\alpha^2}{\lambda_s}\Vert \hat{p}\Vert^2_{\mathrm{L}^2(\Omega)}-\frac{\alpha}{\lambda_s}\Vert\hat{\widetilde{p}}\Vert_{\mathrm{L}^2(\Omega)}\Vert \hat{p}\Vert_{\mathrm{H}^1(\Omega)}\leq 0.\end{equation}
	Thus, by adding \eqref{bound_t_injective} and \eqref{bound_p_injective}, and applying Young's inequality, we readily obtain $\Vert\hat{\b t}\Vert_{\mathbb{L}^2(\Omega)}=\Vert \hat{p}\Vert_{\mathrm{H}^1(\Omega)}=0$.
	%\[
	%	2\mu_s\Vert \hat{\b t}\Vert^2_{\mathbb{L}^2(\Omega)}+\mathrm{min}\left\{c_0,\frac{\kappa_1}{\mu_f}\right\}\Vert \hat{p}\Vert^2_{\mathrm{H}^1(\Omega)}\leq 0,
	%\]
	%which implies that $\Vert\hat{\b t}\Vert_{\mathbb{L}^2(\Omega)}=\Vert \hat{p}\Vert_{\mathrm{H}^1(\Omega)}=0$.
	On the other hand, from Lemma \ref{inf_sup_B}, and using \eqref{VF-poro_1_reduced}, we get 
	\begin{equation*}%\label{inf-sup-stab_injective}
		\hat{\beta}\Vert \hat{\vec{\b \sigma}}\Vert_{\mathbf{X}}\leq \sup_{\mbox{\scriptsize $\b r\in \mathbb{L}^2(\Omega) \atop \b r\neq 0$}}\frac{[B(\b r),\hat{\vec{\b \sigma}}]}{\Vert{\b r\Vert_{\mathbb{L}^2(\Omega)}}}=\sup_{\mbox{\scriptsize $\b r\in \mathbb{L}^2(\Omega) \atop \b r\neq 0$}}\frac{-[A(\hat{\b t}), \b r]}{\Vert{\b r\Vert_{\mathbb{L}^2(\Omega)}}}\leq 2\mu_s\Vert\hat{\b t}\Vert_{\mathbb{L}^2(\Omega)},
	\end{equation*}
	from which, we deduce that $\Vert \hat{\vec{\b \sigma}}\Vert_{\mathbf{X}}=0$, concluding the proof. 
	\end{proof}
As announced above, thanks to Lemmas \ref{F-invert}, \ref{F-compact} and \ref{F-injective}, and the Fredholm's alternative, we deduce the existence of a unique solution to \eqref{VF-poro_1_reduced}-\eqref{VF-poro_3_reduced}, again, with a given $\omega\in \mathrm{H}^1_{\Gamma}(\Omega)$.  
We complement the above result with the stability of \eqref{VF-poro_1_reduced}-\eqref{VF-poro_3_reduced}, for a given $\omega\in \mathrm{H}^1_{\Gamma}(\Omega)$.
\begin{lemma}\label{stability-lemma}
	For each $\omega\in \mathrm{H}^1_{\Gamma}(\Omega)$, there exists a constant $\hat{C}>0$ independent of $\lambda_s$, such that
	\begin{equation*}
	\Vert \hat{\b t}\Vert_{\mathbb{L}^2(\Omega)}+\Vert \hat{\vec{\b \sigma}}\Vert_{\mathbf{X}}+\Vert \hat{p}\Vert_{\mathrm{H}^1(\Omega)}\leq \hat{C}\left(\Vert \widetilde{H}_\omega\Vert_{\mathbb{L}^2(\Omega)'}+\Vert \widetilde{F}_1\Vert_{\mathbf{X}'}+\Vert \widetilde{G}\Vert_{\mathrm{H}^1_{\Sigma}(\Omega)'}\right).%\label{stability_result}
	\end{equation*}
\end{lemma}%Lemma \ref{F-injective} and
\begin{proof}
	We proceed similarly as in \cite[Theorem 2.1]{kumar20} (see also \cite[Theorem 4.3.1]{boffi13}), to obtain estimates and then use linearity. Firstly, we assume  that $\widetilde{F}_1=0$ and  bound  the  solution  in  terms  of $\widetilde{H}_\omega$ and $\widetilde{G}$, and secondly,  we  assume  that $\widetilde{H}_\omega=0, \, \widetilde{G}=0$ and deduce an estimate for the solution in terms of $\widetilde{F}_1$. \\\\
	\textbf{Step 1.} ($\widetilde{F}_1=0$). Proceeding exactly as in the proof of Lemma \ref{F-injective}, it can be deduced the bounds 
%	
	%Taking $\b r=\hat{\b t}$, and $\vec{\b \tau}=\hat{\vec{\b \sigma}}$, in \eqref{VF-poro_1_reduced} and \eqref{VF-poro_2_reduced}, respectively, we get 
	%\begin{equation}\label{bound_t}
	%2\mu_s\Vert \hat{\b t}\Vert^2_{\mathbb{L}^2(\Omega)}+\frac{1}{\lambda_s}\Vert \hat{\widetilde{p}}\Vert^2_{\mathrm{L}^2(\Omega)}-\frac{\alpha}{\lambda_s}\Vert\hat{\widetilde{p}}\Vert_{\mathrm{L}^2(\Omega)}\Vert \hat{p}\Vert_{\mathrm{H}^1(\Omega)} \leq \Vert \widetilde{H}_\omega\Vert_{\mathbb{L}^2(\Omega)'}\Vert\hat{\b t}\Vert_{\mathbb{L}^2(\Omega)},
	%\end{equation}
	%and also, choosing $q=\hat{p}$ in \eqref{VF-poro_3_reduced}, we obtain
	%\begin{equation} \mathrm{min}\left\{c_0,\frac{\kappa_1}{\mu_f}\right\}\Vert \hat{p}\Vert^2_{\mathrm{H}^1(\Omega)}+\frac{\alpha^2}{\lambda_s}\Vert \hat{p}\Vert^2_{\mathrm{L}^2(\Omega)}-\frac{\alpha}{\lambda_s}\Vert\hat{\widetilde{p}}\Vert_{\mathrm{L}^2(\Omega)}\Vert \hat{p}\Vert_{\mathrm{H}^1(\Omega)}\leq \Vert \widetilde{G}\Vert_{\mathrm{H}^1_{\Sigma}(\Omega)'}\Vert \hat{p}\Vert_{\mathrm{H}^1(\Omega)}.\label{bound_p}\end{equation}
%	
%	\begin{equation*}
%	2\mu_s\Vert \hat{\b t}\Vert^2_{\mathbb{L}^2(\Omega)}+\mathrm{min}\left\{c_0,\frac{\kappa_1}{\mu_f}\right\}\Vert \hat{p}\Vert^2_{\mathrm{H}^1(\Omega)}\leq \Vert \widetilde{H}_\omega\Vert_{\mathbb{L}^2(\Omega)'}\Vert\hat{\b t}\Vert_{\mathbb{L}^2(\Omega)}+\Vert \widetilde{G}\Vert_{\mathrm{H}^1_{\Sigma}(\Omega)'}\Vert \hat{p}\Vert_{\mathrm{H}^1(\Omega)},
%	\end{equation*}
%or, more conveniently written as 
\begin{subequations}	\begin{align}\label{bound_t_p}
	\Vert \hat{\b t}\Vert_{\mathbb{L}^2(\Omega)}+\Vert \hat{p}\Vert_{\mathrm{H}^1(\Omega)} & \leq C_1\left(\Vert \widetilde{H}_\omega\Vert_{\mathbb{L}^2(\Omega)'}+\Vert \widetilde{G}\Vert_{\mathrm{H}^1_{\Sigma}(\Omega)'}\right),
	\\
	%where $C_1$ is a constant independent of $\lambda_s$.
	%Finally, from the inf-sup condition for $B$ (cf. Lemma \ref{inf_sup_B}), and using \eqref{VF-poro_1_reduced}, we get 
\label{inf-sup-stab}
	\hat{\beta}\Vert \hat{\vec{\b \sigma}}\Vert_{\mathbf{X}}&\leq \sup_{\mbox{\scriptsize $\b r\in \mathbb{L}^2(\Omega) \atop \b r\neq 0$}}\frac{[B(\b r),\hat{\vec{\b \sigma}}]}{\Vert{\b r\Vert_{\mathbb{L}^2(\Omega)}}}=\sup_{\mbox{\scriptsize $\b r\in \mathbb{L}^2(\Omega) \atop \b r\neq 0$}}\frac{[\widetilde{H}_\omega,\b r]-[A(\hat{\b t}), \b r]}{\Vert{\b r\Vert_{\mathbb{L}^2(\Omega)}}}\leq \Vert\widetilde{H}_\omega\Vert_{\mathbb{L}^2(\Omega)'}+2\mu_s\Vert\hat{\b t}\Vert_{\mathbb{L}^2(\Omega)},
	\end{align}\end{subequations}
	where $C_1$ is a constant independent of $\lambda_s$. Therefore, the desired estimate then follows from \eqref{bound_t_p} and \eqref{inf-sup-stab}. \\\\
	\textbf{Step 2.} ($\widetilde{H}_\omega=0$ and $\widetilde{G}=0$). We start by defining \begin{equation*}
	[S(\hat{\vec{\b \sigma}},\hat{p}),({\vec{\b \tau}},q)]:=[C(\hat{\vec{\b \sigma}}),\vec{\b \tau}]-[B_2^*(\hat{p}),\vec{\b \tau}]-[B_2(\hat{\vec{\b \sigma}}),q]+D[(\hat{p}),q] =\frac{1}{\lambda_s}\int_{\Omega}(\hat{\widetilde{p}}-\alpha \hat{p})(\hat{\widetilde{q}}-\alpha q)+c_0\int_\Omega\hat{p} q,\end{equation*} 
and noticing, for all $\hat{\vec{\b \sigma}}, \vec{\b \tau}\in \mathbb{V}$ and $\hat{p},q \in \mathrm{H}^1_\Sigma(\Omega)$, that 
	\begin{align}\label{bound_S}
	\nonumber 	|[S(\hat{\vec{\b \sigma}},\hat{p}),(\vec{\b \tau},q)]|&\leq [S(\hat{\vec{\b \sigma}},\hat{p}),(\hat{\vec{\b \sigma}},\hat{p})]^{1/2}[S(\vec{\b \tau},q),(\vec{\b \tau},q)]^{1/2}\\
		& =\left(\frac{1}{\lambda_s}\Vert \hat{\widetilde{p}}-\alpha \hat{p}  \Vert^2_{\mathrm{L}^2(\Omega)}+c_0\Vert \hat{p}\Vert^2_{\mathrm{L}^2(\Omega)}\right)^{1/2}\left(\frac{1}{\lambda_s}\Vert \hat{\widetilde{q}}-\alpha q  \Vert^2_{\mathrm{L}^2(\Omega)}+c_0\Vert q\Vert^2_{\mathrm{L}^2(\Omega)}\right)^{1/2}.
	\end{align}
	Additionally, from the inf-sup condition \eqref{ineq_inf_sup_B} it can be deduced that
	\begin{equation}\label{inf_sup_orthogonal}
 \sup_{\vec{\b \tau}\in \mathbb{V} \setminus\{\b 0\}} \frac{ [B({\b r^\perp}),\vec{\b \tau}]}{\Vert\vec{\b \tau}\Vert_{\mathbf{X}}}\geq \hat{\beta}\norm{\b r^\perp}_{\mathbb{L}^2(\Omega)}\quad \forall\, \b r^\perp\in  \widetilde{\mathbb{V}}^\perp,
	\end{equation} 
where $\widetilde{\mathbb{V}}:=\left\{\hat{\b r}\in \mathbb{L}^2(\Omega): [B(\hat{\b r}),\vec{\b \tau}]=0\ \forall\,\vec{\b \tau}\in \mathbb{V}\right\}.$ Now, we let $\hat{\b t}_0\in \widetilde{\mathbb{V}}$ and $\hat{\b t}^\perp\in \widetilde{\mathbb{V}}^\perp$ be such that $\hat{\b t}=\hat{\b t}_0+\hat{\b t}^\perp$, and observe from \eqref{VF-poro_1_reduced}-\eqref{VF-poro_3_reduced} that there hold
\begin{subequations}
\begin{align}\label{reduced_orthogonal_t}
[A(\hat{\b t}),\hat{\b t}]+[S(\hat{\vec{\b \sigma}},\hat{p}),(\hat{\vec{\b \sigma}},\hat{p})]&=-\widetilde{F}_1(\hat{\vec{\b \sigma}}),\\
\label{reduced_orthogonal}
[B(\hat{\b t}^\perp),\vec{\b \tau}]-[S(\hat{\vec{\b \sigma}},\hat{p}),(\vec{\b \tau},q)]&=\widetilde{F}_1(\vec{\b \tau})\quad \forall\,\vec{\b \tau}\in \mathbb{V}, \quad \forall\,q\in \mathrm{H}^1_\Sigma(\Omega).
\end{align}
\end{subequations}
Thus, from \eqref{inf_sup_orthogonal} with $\b r^\perp=\hat{\b t}^\perp$,  estimate \eqref{reduced_orthogonal}, the continuity of $\widetilde{F}_1$ and the bound \eqref{bound_S}, we get
\begin{align}
\label{bound_t_orthogonal}
\nonumber	\hat{\beta}\Vert\hat{\b t}^\perp\Vert_{\mathbb{L}^2(\Omega)}&\leq \sup_{\vec{\b \tau}\in \mathbb{V} \setminus\{\b 0\}} \frac{ [B({\hat{\b t}^\perp}),\vec{\b \tau}]}{\Vert\vec{\b \tau}\Vert_{\mathbf{X}}}=\sup_{\vec{\b \tau}\in \mathbb{V} \setminus\{\b 0\}} \frac{ \widetilde{F}_1(\vec{\b \tau})+[S(\hat{\vec{\b \sigma}},\hat{p}),(\vec{\b \tau},q)]}{\Vert\vec{\b \tau}\Vert_{\mathbf{X}}}\\
	&\leq  \Vert\widetilde{F}_1\Vert_{\mathbf{X}'}+[S(\hat{\vec{\b \sigma}},\hat{p}),(\hat{\vec{\b \sigma}},\hat{p})]^{1/2}\sup_{\vec{\b \tau}\in \mathbb{V} \setminus\{\b 0\}} \frac{\left(\frac{1}{\lambda_s}\Vert \hat{\widetilde{q}}-\alpha q  \Vert^2_{\mathrm{L}^2(\Omega)}+c_0\Vert q\Vert^2_{\mathrm{L}^2(\Omega)}\right)^{1/2}}{\Vert\vec{\b \tau}\Vert_{\mathbf{X}}}.
\end{align}
Then, defining 
\begin{equation*}%\label{def_C_s}
C_S:=\sup_{\vec{\b \tau}\in \mathbb{V} \setminus\{\b 0\}} \frac{\left(\frac{1}{\lambda_s}\Vert \hat{\widetilde{q}}-\alpha q  \Vert^2_{\mathrm{L}^2(\Omega)}+c_0\Vert q\Vert^2_{\mathrm{L}^2(\Omega)}\right)^{1/2}}{\Vert\vec{\b \tau}\Vert_{\mathbf{X}}},
\end{equation*}
which can be seen as a constant independent of $\lambda_s$ if $\lambda_s\rightarrow\infty$, and noticing from \eqref{reduced_orthogonal_t} that $$[S(\hat{\vec{\b \sigma}},\hat{p}),(\hat{\vec{\b \sigma}},\hat{p})]\leq- \widetilde{F}_1(\hat{\vec{\b \sigma}}),$$ it can be deduced from \eqref{bound_t_orthogonal} that 
\begin{equation}\label{bound_t_orthogonal_final}
	\hat{\beta}\Vert\hat{\b t}^\perp\Vert_{\mathbb{L}^2(\Omega)}\leq  \Vert\widetilde{F}_1\Vert_{\mathbf{X}'}+C_S[S(\hat{\vec{\b \sigma}},\hat{p}),(\hat{\vec{\b \sigma}},\hat{p})]^{1/2}\leq \Vert\widetilde{F}_1\Vert_{\mathbf{X}'}+C_S\Vert\widetilde{F}_1\Vert_{\mathbf{X}'}^{1/2}\Vert\hat{\vec{\b \sigma}}\Vert_{\mathbf{X}}^{1/2}.
\end{equation}
Furthermore, taking $\hat{\b t}_0$ in \eqref{VF-poro_1_reduced}, applying the ellipticity and continuity of   $A$, and recalling \eqref{eq:deco},   we easily get 
%\[\Vert\hat{\b t}_0\Vert_{\mathbb{L}^2(\Omega)}\leq \widetilde{C}_1\Vert\hat{\b t}^\perp\Vert_{\mathbb{L}^2(\Omega)},\]
%which implies
\begin{equation}\label{bound_hat_t}
	\Vert\hat{\b t}\Vert_{\mathbb{L}^2(\Omega)}\leq (1+\widetilde{C}_1)\Vert\hat{\b t}^\perp\Vert_{\mathbb{L}^2(\Omega)}.
\end{equation}
In this way, from \eqref{bound_hat_t}, the inf-sup condition \eqref{ineq_inf_sup_B} and equation \eqref{VF-poro_1_reduced}, it readily follows that 
\begin{equation}\label{inf-sup-hat_sigma}
\Vert \hat{\vec{\b \sigma}}\Vert_{\mathbf{X}}\leq \hat{\beta}^{-1}\sup_{\mbox{\scriptsize $\b r\in \mathbb{L}^2(\Omega) \atop \b r\neq 0$}}\frac{[B(\b r),\hat{\vec{\b \sigma}}]}{\Vert{\b r\Vert_{\mathbb{L}^2(\Omega)}}}=\hat{\beta}^{-1}\sup_{\mbox{\scriptsize $\b r\in \mathbb{L}^2(\Omega) \atop \b r\neq 0$}}\frac{-[A(\hat{\b t}), \b r]}{\Vert{\b r\Vert_{\mathbb{L}^2(\Omega)}}}\leq 2\mu_s\hat{\beta}^{-1}\Vert\hat{\b t}\Vert_{\mathbb{L}^2(\Omega)}\leq \widetilde{C}_2\Vert\hat{\b t}^\perp\Vert_{\mathbb{L}^2(\Omega)},
\end{equation}
which combined with \eqref{bound_t_orthogonal_final} and the Young's inequality yields 
\begin{equation*}%\label{bound_hat_t_final}
	\Vert\hat{\b t}^\perp\Vert_{\mathbb{L}^2(\Omega)}\leq \widetilde{C}_3\Vert\widetilde{F}_1\Vert_{\mathbf{X}'}.
\end{equation*}
Therefore, from the latter inequality, \eqref{bound_hat_t} and \eqref{inf-sup-hat_sigma}, we derive that 
\begin{equation}\label{final_bound_hat_t_sigma}
	\Vert\hat{\b t}\Vert_{\mathbb{L}^2(\Omega)}\leq \widetilde{C}_4\Vert\widetilde{F}_1\Vert_{\mathbf{X}'}\quad \mathrm{and}\quad \Vert \hat{\vec{\b \sigma}}\Vert_{\mathbf{X}}\leq\widetilde{C}_5\Vert\widetilde{F}_1\Vert_{\mathbf{X}'}, 
\end{equation}
where $\widetilde{C}_4>0$ and $\widetilde{C}_5>0$ are constants independent of $\lambda_s$. 
Finally, the ellipticity of   $D$ and \eqref{VF-poro_3_reduced} yield 
\begin{equation}\label{bound_hat_p}
\Vert \hat{p}\Vert^2_{\mathrm{H}^1(\Omega)}\leq \widetilde{C}_6\Vert \hat{p}\Vert_{\mathrm{H}^1(\Omega)}\Vert \hat{\vec{\b \sigma}}\Vert_{\mathbf{X}}.
\end{equation}
In this way, the result follows from \eqref{final_bound_hat_t_sigma} and \eqref{bound_hat_p}, and after apply some computations.
%by and therefore 
%\begin{equation}
%\Vert p\Vert_{\mathrm{H}^1(\Omega)}\leq \widetilde{C}_7\Vert\widetilde{F}_1\Vert_{\mathbf{X}'}.
%\end{equation}
%In this way, the result follows from \eqref{final_bound_hat_t_sigma} and \eqref{bound_hat_p}, concluding the proof.
\end{proof}
We summarise the foregoing results in the following lemma.
\begin{lemma}\label{well-posed-poro}
	For each $\omega\in \mathrm{H}^1_\Gamma(\Omega)$ the problem \eqref{VF-poro_1_reduced}-\eqref{VF-poro_3_reduced} has a unique solution $(\hat{\b t}, \hat{\vec{\b \sigma}},\hat{p})\in \mathbb{L}^2(\Omega)\times\mathbb{V}\times\mathrm{H}^1_{\Sigma}(\Omega)$. Moreover, there exists $\hat{C}>0$ independent of $\omega$ and $\lambda_s$, such that
	\begin{equation*}%\label{stability_final}
	\Vert \hat{\b t}\Vert_{\mathbb{L}^2(\Omega)}+\Vert \hat{\vec{\b \sigma}}\Vert_{\mathbf{X}}+\Vert \hat{p}\Vert_{\mathrm{H}^1(\Omega)}\leq \hat{C}\left(\Vert \widetilde{H}_\omega\Vert_{\mathbb{L}^2(\Omega)'}+\Vert \widetilde{F}_1\Vert_{\mathbf{X}'}+\Vert \widetilde{G}\Vert_{\mathrm{H}^1_{\Sigma}(\Omega)'}\right).\end{equation*}
\end{lemma}
Finally, the main result is given next.
\begin{lemma}\label{lemma_main}
		For each $\omega\in \mathrm{H}^1_\Gamma(\Omega)$ the problem \eqref{VF-poro_1}-\eqref{VF-poro_4} has a unique solution $(\b t, \vec{\b \sigma},\vec{\b u},p)\in \mathbb{L}^2(\Omega)\times\mathbf{X}\times\mathbf{M}\times\mathrm{H}^1_{\Sigma}(\Omega)$. Moreover, there exists $C>0$ independent of $\omega$ and $\lambda_s$, such that
		\begin{equation}\label{stability_main}
	\Vert \b t\Vert_{\mathbb{L}^2(\Omega)}+\Vert \vec{\b \sigma}\Vert_{\mathbf{X}}+\Vert \vec{\b u}\Vert_{\mathbf{M}}+\Vert p\Vert_{\mathrm{H}^1(\Omega)} \leq C\left(\Vert \widetilde{H}_\omega\Vert_{\mathbb{L}^2(\Omega)'}+\Vert\widetilde{F}_1\Vert_{\mathbf{X}'}+\Vert \widetilde{F}\Vert_{\mathbf{M}'}+\Vert \widetilde{G}\Vert_{\mathrm{H}^1_{\Sigma}(\Omega)'}\right).\end{equation}
\end{lemma}
\begin{proof}
We proceed similarly as in \cite[Theorem 2.1]{gatica03}. Recalling from \cite[Lemma 4.1]{girault_book86}  that the result given in Lemma \ref{lem_inf_sup_B1} implies that $B_1:\mathbb{V}^\perp\rightarrow \mathbf{M}'$ and $B_1^*:\mathbf{M}\rightarrow \mathbb{V}^\circ$ are isomorphisms with \begin{equation}\label{bound_B_B*}
\Vert B_1\Vert, \quad \Vert(B_1^*)^{-1}\Vert\leq \frac{1}{\hat{\beta_1}},
\end{equation} where $\mathbb{V}^\circ$ stands for the set of functionals in $\mathbf{X}$ that vanish on the elements of $\mathbb{V}$. Now, let $\vec{\b \sigma}_0:=B_1^{-1}(\widetilde{F})\in \mathbb{V}^\perp$, and notice that by using \eqref{bound_B_B*} it can be deduce that 
\begin{equation}\label{bound_sigma0}
\Vert \vec{\b \sigma}_0\Vert_{\mathbf{X}}\leq \frac{1}{\beta_1}\Vert \widetilde{F}\Vert_{\mathbf{M}'}.
\end{equation} 
With this in mind, we define the functionals $\underline{H}_\omega:=\widetilde{H}_\omega-B^*(\vec{\b \sigma}_0)$, $\underline{F}_1:=\widetilde{F}_1+C(\vec{\b \sigma}_0)$ and $\underline{G}:=\widetilde{G}-B_2(\vec{\b \sigma}_0)$, and consider the following problem: Find $(\underline{\b t}, \underline{\vec{\b \sigma}},\underline{p})\in \mathbb{L}^2(\Omega)\times\mathbb{V}\times\mathrm{H}^1_{\Sigma}(\Omega)$ such that 
	\begin{alignat}{11}
	\nonumber&[A(\underline{\b t}), \b r]&+\quad&[B^*(\underline{\vec{\b \sigma}}),\b r]&&&&&=& \;\;[\underline{H}_{\omega},\b r],&\\
	\label{perturbed-problem}&[B(\underline{\b t}), \vec{\b \tau}]\;\;\,&-\quad& [C(\underline{\vec{\b \sigma}}),\vec{\b \tau}]\;\;\,&&\;&+&\;\; [B_2^*(\underline{p}),\vec{\b \tau}]\;&=&\;\;[\underline{F},\vec{\b \tau}],&\\
	\nonumber&&&[B_2(\underline{\vec{\b \sigma}}),q]&&&-&\;\;[D(\underline{p}),q]&=& \;\;[\underline{G},q],&
	\end{alignat}
for all $(\b r,\vec{\b \tau},q)\in \mathbb{L}^2(\Omega)\times\mathbb{V}\times \mathrm{H}^1_{\Sigma}(\Omega)$, where the involved operators are exactly the ones defining  \eqref{VF-poro_1_reduced}-\eqref{VF-poro_3_reduced}. Therefore, by noticing that $\underline{H}_\omega \in \mathbb{L}^2(\Omega)'$, $\underline{F}\in \mathbb{V}^\circ$ and $\underline{G}\in \mathrm{H}^1_\Sigma(\Omega)'$, we can simply take $\widetilde{H}_\omega$, $\widetilde{F}_1$ and $\widetilde{G}$ in \eqref{VF-poro_1_reduced}-\eqref{VF-poro_3_reduced} as $\underline{H}_\omega$, $\underline{F}$ and $\underline{G}$, respectively, and apply  Lemma \ref{stability-lemma} to assert the existence and uniqueness of a solution $(\underline{\b t}, \underline{\vec{\b \sigma}},\underline{p})\in \mathbb{L}^2(\Omega)\times\mathbb{V}\times\mathrm{H}^1_{\Sigma}(\Omega)$ to problem \eqref{perturbed-problem}. 

On the other hand, since $(\underline{F}-B(\underline{\b t})+C(\underline{\vec{\b \sigma}})-B^*_2(\underline{p}))\in \mathbb{V}^\circ$, we can take $\underline{\vec{\b u}}:=(B_1^*)^{-1}(\underline{F}-B(\underline{\b t})+C(\underline{\vec{\b \sigma}})-B^*_2(\underline{p}))\in \mathbf{M}$, from which, it is clear that 
\begin{equation}\label{def_u}
	[B_1^*(\underline{\vec{\b u}}),{\vec{\b \tau}}]=[\underline{F}-B(\underline{\b t})+C(\underline{\vec{\b \sigma}})-B^*_2(\underline{p}),{\vec{\b \tau}}]\quad\forall\,{\vec{\b \tau}}\in \mathbf{X}, 
\end{equation}
and therefore, \begin{equation}\label{bound_u_main}
\Vert{\underline{\vec{\b u}}}\Vert_{\mathbf{M}}\leq\frac{1}{\beta_1}\Vert \underline{F}-B(\underline{\b t})+C(\underline{\vec{\b \sigma}})-B^*_2(\underline{p})\Vert.
\end{equation}
In this manner, noting also that $B_1(\underline{\vec{\b \sigma}}+\vec{\b \sigma}_0)=B_1(\vec{\b \sigma}_0)=\widetilde{F}$, we finally conclude for each $\omega\in \mathrm{H}^1_\Gamma(\Omega)$, that $(\underline{\b t}, \underline{\vec{\b \sigma}}+\vec{\b \sigma}_0,\underline{\vec{\b u}},\underline{p})\in \mathbb{L}^2(\Omega)\times\mathbf{X}\times\mathbf{M}\times\mathrm{H}^1_{\Sigma}(\Omega)$ is solution of \eqref{VF-poro_1}-\eqref{VF-poro_4}.

Now, for the uniqueness, consider another solution $({\b t}, {\vec{\b \sigma}},{\vec{\b u}},{p})\in \mathbb{L}^2(\Omega)\times\mathbf{X}\times\mathbf{M}\times\mathrm{H}^1_{\Sigma}(\Omega)$ of \eqref{VF-poro_1}-\eqref{VF-poro_4}.
It is fairly directly seen, from \eqref{VF-poro_1}-\eqref{VF-poro_4}, that $({\b t}, {\vec{\b \sigma}}-\vec{\b \sigma}_0,{p})\in \mathbb{L}^2(\Omega)\times\mathbb{V}\times\mathrm{H}^1_{\Sigma}(\Omega)$ is also a solution of \eqref{perturbed-problem}, and therefore $({\b t}, {\vec{\b \sigma}}-\vec{\b \sigma}_0,{p})=(\underline{\b t}, \underline{\vec{\b \sigma}},\underline{p})$, which together with \eqref{def_u} gives $({\b t}, {\vec{\b \sigma}}, \vec{\b u},{p})=(\underline{\b t}, \underline{\vec{\b \sigma}}+\vec{\b \sigma}_0,\underline{\vec{\b u}},\underline{p})$.

On the other hand, in order to deduce the stability estimate \eqref{stability_main}  we begin by using \eqref{bound_u_main}, to obtain 
\[\Vert{{\vec{\b u}}}\Vert_{\mathbf{M}}\leq\frac{1}{\beta_1}\Vert -B({\b t})+C({\vec{\b \sigma}_0})-B^*_2({p})\Vert,
\]
from which, by using the first and second inequalities in \eqref{def-C-B2-D}, the inf-sup condition \eqref{ineq_inf_sup_B}, and the bound \eqref{bound_sigma0}, we deduce that there exists $\underline{C}_1>0$ independent of $\lambda_s$ such that 
\begin{equation}\label{bound_u_final}
\Vert{{\vec{\b u}}}\Vert_{\mathbf{M}}\leq\underline{C}_1\left\{\Vert \b t\Vert_{\mathbb{L}^2(\Omega)}+\frac{1}{\lambda_s}\Vert \widetilde{F}\Vert_{_{\mathbf{M}}'}+\frac{\alpha}{\lambda_s}\Vert p\Vert_{\mathrm{H}^1(\Omega)}\right\}.
\end{equation}

Moreover, applying   Lemma \ref{well-posed-poro} to \eqref{perturbed-problem}, we deduce the existence of $\underline{C}>0$, independent of $\lambda_s$, such that
\[
	\Vert \b t\Vert_{\mathbb{L}^2(\Omega)}+\Vert \vec{\b \sigma}-\vec{\b \sigma}_0\Vert_{\mathbf{X}}+\Vert p\Vert_{\mathrm{H}^1(\Omega)} \leq \underline{C}\left\{\Vert \widetilde{H}_\omega-B^*(\vec{\b \sigma}_0)\Vert+\Vert \widetilde{F}_1+C(\vec{\b \sigma}_0)\Vert+\Vert \widetilde{G}-B_2(\vec{\b \sigma}_0) \Vert\right\}.\]
Thus, from the definitions in   \eqref{def-A-H} and   \eqref{def-F-G-G1}, we proceed as before to get  
\begin{equation}\label{prel_bound_sigma}
\Vert \b t\Vert_{\mathbb{L}^2(\Omega)}+\Vert \vec{\b \sigma}\Vert_{\mathbf{X}}+\Vert p\Vert_{\mathrm{H}^1(\Omega)}\\
\leq \underline{C}_2\left\{\Vert \widetilde{H}_\omega\Vert_{\mathbb{L}^2(\Omega)'}+\Vert\widetilde{F}_1\Vert_{\mathbf{X}'}+\left(1+\frac{1}{\lambda_s}+\frac{\alpha}{\lambda_s}\right)\Vert \widetilde{F}\Vert_{\mathbf{M}'}+\Vert \widetilde{G}\Vert_{\mathrm{H}^1_\Sigma(\Omega)'}\right\},
\end{equation}
where $\underline{C}_2>0$ is  independent of $\lambda_s$. In this way, replacing \eqref{prel_bound_sigma} back into \eqref{bound_u_final}, we deduce that 
\begin{align*}
&\Vert \b t\Vert_{\mathbb{L}^2(\Omega)}+\Vert \vec{\b \sigma}\Vert_{\mathbf{X}}+\Vert \vec{\b u}\Vert_{\mathbf{M}}+\Vert p\Vert_{\mathrm{H}^1(\Omega)}\\
& \leq \underline{C}_3\left(1+\frac{\alpha}{\lambda_s}\right)\left\{\Vert \widetilde{H}_\omega\Vert_{\mathbb{L}^2(\Omega)'}+\Vert\widetilde{F}_1\Vert_{\mathbf{X}'}+\left(1+\frac{2}{\lambda_s}+\frac{\alpha}{\lambda_s}\right)\Vert \widetilde{F}\Vert_{\mathbf{M}'}+\Vert \widetilde{G}\Vert_{\mathrm{H}^1_\Sigma(\Omega)'}\right\}.\end{align*}
Therefore, noting that the constants $1+\frac{\alpha}{\lambda_s}$ and $1+\frac{2}{\lambda_s}+\frac{\alpha}{\lambda_s}$ are understood as independent of $\lambda_s$ in the limit $\lambda_s\to \infty$, we can apply elementary algebraic computations to deduce the existence of a constant $C$ independent of $\lambda_s$, such that \eqref{stability_main} holds.
\end{proof}

We establish now the well-posedness of the operator $\widetilde{\mathbf{S}}$. The following result asserts the unique solvability of \eqref{VF_diff} for a given $\b \sigma \in \mathbb{H}_{\Sigma}(\mathbf{div},\Omega)$.
\begin{lemma}\label{well-posed-dif}
	For each $\b \sigma \in \mathbb{H}_{\Sigma}(\mathbf{div},\Omega)$, the problem \eqref{VF_diff} has a unique solution $\omega:= \widetilde{\mathbf{S}}(\b \sigma)\in \mathrm{H}^1_{\Gamma}(\Omega)$, and there holds
	\begin{equation}\label{stabilization-diff_lemma}
		\Vert\widetilde{\mathbf{S}}(\b \sigma )\Vert:=\Vert \omega\Vert_{\mathrm{H}^1(\Omega)}\leq \widetilde{\alpha}^{-1}\Vert J\Vert_{\mathrm{H}^1_\Gamma(\Omega)'}.
	\end{equation}
\end{lemma}
	\begin{proof}
		We begin by applying the uniformly positive definiteness of $D$ (cf. \eqref{prop-D}) to obtain 
		\begin{equation}\label{ellipticity_A}
			\mathscr{A}_{\b \sigma}(\omega, \omega)\geq \phi\Vert\omega\Vert^2_{\mathrm{L}^2(\Omega)}+D_0|\omega|^2_{\mathrm{H}^1(\Omega)}\geq\widetilde{\alpha}\Vert\omega\Vert^2_{\mathrm{H}^1(\Omega)},
		\end{equation}
		which implies the ellipticity of $\mathscr{A}_{\b \sigma}(\cdot, \cdot)$, with $\widetilde{\alpha}:=\mathrm{min}\left\{D_0,\phi \right\}$. Thus, a straightforward application of Lax-Milgram's lemma (see, e.g., \cite[Thm. 1.1]{gatica14}), proves that for each $\b \sigma \in \mathbb{H}_{\Sigma}(\mathbf{div},\Omega)$, problem \eqref{VF_diff} has a unique solution $\omega= \widetilde{\mathbf{S}}(\b \sigma)\in \mathrm{H}^1_{\Gamma}(\Omega)$. Moreover, the corresponding continuous dependence on the data is formulated as
		\begin{equation}\label{stabilization-diff}
			\Vert \omega\Vert_{\mathrm{H}^1(\Omega)}\leq  \widetilde{\alpha}^{-1}\Vert J\Vert_{\mathrm{H}^1_\Gamma(\Omega)'}.
		\end{equation}
	\end{proof}
As $\mathbf{S}$ and $\widetilde{\mathbf{S}}$ are well-defined,  we can guarantee the well-posedness of the operator $\mathbf{T}$.

\subsection{Solvability of the fixed-point problem}\label{solvability_FP}
With the aim to use the Schauder fixed-point theorem on $\mathbf{T}$, in what follows we  establish sufficient conditions under which $\mathbf{T}$ maps a closed ball of $\mathrm{H}^1_\Gamma(\Omega)$ into itself. Indeed, from now on we let
\begin{equation}\label{ball_W}
	\mathbf{W}:=\left\{\omega \in \mathrm{H}^1_\Gamma(\Omega):\;\; \norm{\omega}_{\mathrm{H}^1(\Omega)}\leq r:=\widetilde{\alpha}^{-1}\Vert J\Vert_{\mathrm{H}^1_\Gamma(\Omega)'}\right\}.
\end{equation}
%Then we have the following result.
\begin{lemma}\label{bound_TW}
For the closed ball $\mathbf{W}$, it holds that $\mathbf{T}(\mathbf{W})\subseteq \mathbf{W}$. 
\end{lemma}
\begin{proof}
It suffices to recall the definition of $\mathbf{T}$ (\ref{def-T}), and then apply the estimate (\ref{stabilization-diff_lemma}).
\end{proof}
We now verify the hypotheses of the Schauder fixed-point theorem. Before
starting the result to be proved, we show the Lipschitz continuity property for the operators $\mathbf{S}$ and $\widetilde{\mathbf{S}}$, and then for $\mathbf{T}$.
\begin{lemma}\label{Lipschitz_S_WS}
	There exists a constant $C_{\mathbf{S}}>0$, independent of $\lambda_s$ such that 
	\begin{align*}
	%\label{LC_S}
	\Vert\mathbf{S}(\omega_1)-\mathbf{S}(\omega_2)\Vert &\leq C_{\mathbf{S}}  \beta\Vert\omega_1-\omega_2\Vert_{\mathrm{L}^2(\Omega)}\quad \forall\, \omega_1, \omega_2\in \mathrm{H}^1_\Gamma(\Omega),
\\
%\label{WT_LC_S}
	\Vert\widetilde{\mathbf{S}}(\b \sigma)-\widetilde{\mathbf{S}}(\widetilde{\b \sigma})\Vert &\leq \frac{d_3}{\widetilde{\alpha}}\Vert\widetilde{\b \sigma}-\b \sigma\Vert_{\mathbb{L}^2(\Omega)}\Vert\widetilde{\mathbf{S}}(\widetilde{\b \sigma})\Vert_{W^{1,\infty}(\Omega)}\quad \forall\, {\b \sigma}, \widetilde{\b \sigma}\in \mathbb{H}_\Sigma(\mathbf{div},\Omega),
	\end{align*}
	where $\beta, \widetilde{\alpha}$ and $d_3$ are the constants given in \eqref{eq:sigma}, \eqref{ellipticity_A} and \eqref{prop-D}, respectively.
\end{lemma}
\begin{proof}
	We begin with the proof for $\mathbf{S}$. Let $(\b t_1, \vec{\b \sigma}_1,\vec{\b u}_1,p_1)$ and $(\b t_2, \vec{\b \sigma}_2,\vec{\b u}_2,p_2)$ be two solutions of the problem \eqref{VF-poro_1}-\eqref{VF-poro_4}, with $\omega_1$ and $\omega_2 \in \mathrm{H}^1_\Gamma(\Omega)$ given, such that $\mathbf{S}(\omega_1):=(\b t_1, \vec{\b \sigma}_1,\vec{\b u}_1,p_1)$ and $\mathbf{S}(\omega_2):=(\b t_2, \vec{\b \sigma}_2,\vec{\b u}_2,p_2)$. Then, by applying the linearity of the involved bilinear forms and functionals, we obtain
	\begin{alignat*}{10}
	&A(\b t_1-\b t_2, \b r)&+\quad&B(\b r, \vec{\b \sigma}_1-\vec{\b \sigma}_2)&&&&&=& \;H_{\omega_1-\omega_2}(\b r),\\
	&B(\b t_1-\b t_2, \vec{\b \tau})\;\;\,&-\quad& C(\vec{\b \tau},\vec{\b \sigma}_1-\vec{\b \sigma}_2)\;\;\,&+&\;\; B_1(\vec{\b \tau},\vec{\b u}-\vec{\b u_2})\;\;&+&\;\;B_2(\vec{\b \tau},p_1-p_2)\;&=&\;0,\\
	& &&B_1(\vec{\b \sigma}_1-\vec{\b \sigma}_2,\vec{\b v})&&&&&=&\;0,\\
	&&&B_2(\vec{\b \sigma},q)&&&-&\;\;D(p,q)&=& \;0, 
	\end{alignat*}
	for each $(\b r, \vec{\b \tau}, \vec{\b v}, p)\in \mathbb{L}^2(\Omega)\times \mathbf{X}\times \mathbf{M} \times \mathrm{H}^1_{\Sigma}(\Omega)$.
	Proceeding as in the proof of Lemma \ref{well-posed-poro}, we deduce that 
	\begin{equation}\label{bound_LC_S}
	\Vert \mathbf{S}(\omega_1)-\mathbf{S}(\omega_2)\Vert=\Vert (\b t_1, \vec{\b \sigma}_1,\vec{\b u}_1,p_1)-(\b t_2, \vec{\b \sigma}_2,\vec{\b u}_2,p_2)\Vert  \leq C_{\mathbf{S}}  \beta\Vert\omega_1-\omega_2 \Vert_{\mathrm{L}^2(\Omega)}.
	\end{equation}
	In turn, for the problem defined by $\widetilde{\mathbf{S}}$, given $\b \sigma$ and $\widetilde{\b \sigma} \in \mathbb{H}_{\Sigma}(\mathbf{div},\Omega)$, we let $\omega$ and $\widetilde{\omega}$ be two solutions of problem \eqref{VF_diff}, such that $\omega:={\widetilde{\mathbf{S}}}(\b \sigma)$ and $\widetilde{\omega}:=\widetilde{\mathbf{S}}(\widetilde{\b \sigma})$. Then, adding and subtracting appropriate terms, and utilising the ellipticity of $\mathscr{A}_{\b \sigma}$, we find that
	\begin{align}
\label{LC_omega}
\nonumber	\widetilde{\alpha}\Vert \omega-\widetilde{\omega}\Vert^2_{\mathrm{H}^1(\Omega)}& \leq \mathscr{A}_{\b \sigma}(\omega,\omega-\widetilde{\omega})- \mathscr{A}_{\b \sigma}(\widetilde{\omega},\omega-\widetilde{\omega})=(\mathscr{A}_{\widetilde{\b \sigma}}-\mathscr{A}_{{\b \sigma}})(\widetilde{\omega},\omega-\widetilde{\omega})\\
	&=\int_{\Omega}(D(\widetilde{\b \sigma})-D({\b \sigma}))\nabla \widetilde{\omega}\cdot\nabla (\omega-\widetilde{\omega})\leq d_3\Vert\widetilde{\b \sigma}-\b \sigma\Vert_{\mathbb{L}^2(\Omega)}\Vert\widetilde{\omega}\Vert_{W^{1,\infty}(\Omega)}\Vert\omega-\widetilde{\omega}\Vert_{\mathrm{H}^1(\Omega)}.
	\end{align}   
	Thus, the proof ends after noticing that   from \eqref{LC_omega} we obtain 
	\begin{equation}\label{bound_TS}
	\Vert\widetilde{\mathbf{S}}(\b \sigma)-\widetilde{\mathbf{S}}(\widetilde{\b \sigma})\Vert_{\mathrm{H}^1(\Omega)}:=\Vert \omega-\widetilde{\omega}\Vert^2_{\mathrm{H}^1(\Omega)}\leq \frac{d_3}{\widetilde{\alpha}}\Vert\widetilde{\b \sigma}-\b \sigma\Vert_{\mathbb{L}^2(\Omega)}\Vert\widetilde{\mathbf{S}}(\widetilde{\b \sigma})\Vert_{W^{1,\infty}(\Omega)}.
	\end{equation}
\end{proof}

Now, we are able to show the announced property of the operator $\mathbf{T}$.
\begin{lemma}\label{lemma:LC_T}
	There exists a constant $C_{\mathbf{T}}>0$, independent of $\lambda_s$ such that 
	\begin{equation}\label{LC_T}
	\Vert\mathbf{T}(\omega_1)-\mathbf{T}(\omega_2)\Vert_{\mathrm{H}^1(\Omega)}\leq C_{\mathbf{T}}d_3\beta\Vert\mathbf{T}(\omega_2)\Vert_{W^{1,\infty}(\Omega)}\Vert\omega_1-\omega_2\Vert_{\mathrm{L}^2(\Omega)}  \quad \forall\, \omega_1, \omega_2\in \mathrm{H}^1_\Gamma(\Omega).
	\end{equation}
\end{lemma}
\begin{proof}
It suffices to recall from Section \ref{section-a-fixed-point-strategy} that $\mathbf{T}(\omega)=\widetilde{\mathbf{S}}(\mathbf{S}_2(\omega))$  for all $\omega\in \mathrm{H}^1_\Gamma(\Omega)$, and  apply Lemma \ref{Lipschitz_S_WS}.
\end{proof}
The next lemma establishes the continuity and compactness of $\mathbf{T}$ on $\mathbf{W}$.
\begin{lemma}\label{theorem:compact}
Let $\mathbf{W}$ be as in   {\rm \eqref{ball_W}}. Then, $\mathbf T:\mathbf{W}\to \mathbf{W}$ is continuous and $\overline{\mathbf T(\mathbf{W})}$ is compact.
\end{lemma}
\begin{proof}
	The continuity of $i_c: \mathrm{H}^1(\Omega)\rightarrow \mathrm{L}^2(\Omega)$ in combination with the estimate \eqref{LC_T}, give the continuity of the operator $\mathbf{T}$. Finally, thanks to the compactness of $i_c$ and the fact that every bounded sequence in a Hilbert space has a weakly convergent subsequence, we can assert the compactness of $\overline{\mathbf{T}(\mathbf{W})}$, concluding the proof.
\end{proof}

We can now establish the existence of a solution to \eqref{VF-poro_1}-\eqref{VF_diff}, which follows straightforwardly from 
Schauder's theorem together with Lemmas \ref{lemma:LC_T} and \ref{theorem:compact}. 
\begin{theorem}
Let $\mathbf{W}$ be as in \eqref{ball_W}. Then, there exists $\overline{C}>0$, independent of $\lambda_s$, such that   $\mathrm{(\ref{VF-poro_1})-\eqref{VF_diff}}$ has at least one solution $(\b t, \vec{\b \sigma},\vec{\b u},p, \omega)\in \mathbb{L}^2(\Omega)\times\mathbf{X}\times\mathbf{M}\times\mathrm{H}^1_{\Sigma}(\Omega)\times \mathrm{H}^1_{\Gamma}(\Omega)$ with $\omega\in \mathbf{W}$, satisfying   
\begin{equation}\label{bound-sigma-u-rho-T}
\Vert \b t\Vert_{\mathbb{L}^2(\Omega)}+\Vert \vec{\b \sigma}\Vert_{\mathbf{X}}+\Vert \vec{\b u}\Vert_{\mathbf{M}}+\Vert p\Vert_{\mathrm{H}^1(\Omega)}+\Vert \omega\Vert_{\mathrm{H}^1(\Omega)}\leq  \overline{C}\left(\Vert F\Vert_{\mathbf{M}'}+\Vert G\Vert_{\mathrm{H}^1_{\Sigma}(\Omega)'}+\Vert J\Vert_{\mathrm{H}^1_\Gamma(\Omega)'}\right).
\end{equation}
\end{theorem}
\begin{proof}
It only remains to prove   \eqref{bound-sigma-u-rho-T}. For that, it is suffices to apply the result \eqref{stability_main} with $\widetilde{H}_\omega=H_\omega, \widetilde{F}_1=\mathrm{O}, \widetilde{F}=F$ and $\widetilde{G}=G$, the estimate \eqref{stabilization-diff_lemma}, the corresponding bound for $\Vert H_\omega\Vert_{\mathbb{L}^2(\Omega)'}$ and the fact that $\omega\in \mathbf{W}$. 
%We omit further details.
\end{proof}
Finally, we establish the uniqueness of solution to   \eqref{VF-poro_1}-\eqref{VF_diff} based on an additional smallness assumption.
\begin{theorem}\label{theorem:solvability_continuous}
Let $(\b t, \vec{\b \sigma},\vec{\b u},p, \omega)\in \mathbb{L}^2(\Omega)\times\mathbf{X}\times\mathbf{M}\times\mathrm{H}^1_{\Sigma}(\Omega)\times \mathrm{H}^1_{\Gamma}(\Omega)\cap W^{1,\infty}(\Omega)$ be  a  solution  to  problem \eqref{VF-poro_1}-\eqref{VF_diff},  and assume that there exists $M >0$, such that
\begin{equation}\label{bound_M}
\Vert\omega\Vert_{W^{1,\infty}(\Omega)}\leq M<\frac{\widetilde{\alpha}}{d_3\beta C_{\mathbf{S}}}.
\end{equation}
Then, the solution of problem \eqref{VF-poro_1}-\eqref{VF_diff} is unique.
\end{theorem}
\begin{proof}
	It follows the same steps as in the proof of Lemma \ref{Lipschitz_S_WS} to obtain the bounds \eqref{bound_LC_S} and \eqref{bound_TS}, and therefore, the result is merely an application of the assumption \eqref{bound_M}. 
\end{proof}

%***********************************************************************************
\section{Finite element method and solvability of the discrete problem} \label{sec:fem}
%***********************************************************************************
In this section we introduce and analyse the Galerkin scheme associated with   (\ref{VF-poro_1})-\eqref{VF_diff}.
We consider generic finite dimensional subspaces 
\begin{gather}
\mathbb{H}^{\b t}_{h}\subseteq \mathbb{L}^{2}(\Omega),\quad\mathbb{H}^{\b \sigma}_{h}\subseteq \mathbb{H}_{\Sigma}(\mathbf{div},\Omega),\quad \mathrm{H}^{\widetilde{p}}_h(\Omega)\subseteq \mathrm{L}^2(\Omega),\quad  \mathbf{H}^{\bu}_{h}\subseteq \mathbf{L}^2(\Omega), \nonumber
\\
\label{arbitrary_FEM}
\mathbb{H}^{\gamma}_h(\Omega)\subseteq \mathbb{L}^2_{\mathrm{skew}}(\Omega),\quad \mathrm{H}^p_h \subseteq \mathrm{H}^1_{\Sigma}(\Omega),\quad\mathrm{and}\quad \mathrm{H}^{\omega}_h\subseteq\mathrm{H}^{1}_{\Gamma}(\Omega),
\end{gather}
which will be specified later on. Then, we define 
\begin{gather*}
\vec{\b \sigma}_h:=(\b \sigma_h, \widetilde{p}_h),\quad \vec{\b \tau}_h:=(\b \tau_h, \widetilde{q}_h),\quad \vec{\b u}_h:=(\b u_h, \b \gamma_h), \quad \vec{\b v}_h:=(\b v_h, \b \eta_h),  \quad 
\mathbf{X}_h:= \mathbb{H}_{h}^{\b \sigma}\times \mathrm{H}^{\widetilde{p}}_h,\quad \mathbf{M}_h:=\mathbf{H}^{\b u}_h\times \mathbb{H}^{\gamma}_h,\end{gather*}
and a Galerkin scheme for (\ref{VF-poro_1})-\eqref{VF_diff} reads: Find $(\b t_h, \vec{\b \sigma},\vec{\b u}_h,$
 $p_h, \omega_h)\in \mathbb{H}^{\b t}_h\times\mathbf{X}_h\times\mathbf{M}_h\times\mathrm{H}^{p}_h\times \mathrm{H}^{\omega}_h$ such that 
\begin{subequations}
	\begin{alignat}{11}
	&[A(\b t_h), \b r_h]&\,+\,&[B^*( \vec{\b \sigma}_h),\b r_h]&&&&&\;=\;& [H_{\omega_h},\b r_h]&\quad \forall\, \b r_h\in \mathbb{H}^{\b t}_h,	\label{VFh-poro_1}\\
	&[B(\b t_h), \vec{\b \tau}_h]&\,-\,& [C(\vec{\b \sigma}_h),\vec{\b \tau}_h]&\,+\,&[B_1^*(\vec{\b u}_h),\vec{\b \tau}_h]&\,+\,&[B_2^*(p_h),\vec{\b \tau}_h]&\;=\;&0& \forall\, \vec{\b \tau}_h \in \mathbf{X}_h,\label{VFh-poro_2}\\
	& &&[B_1(\vec{\b \sigma}_h),\vec{\b v}_h]&&&&&\;=\;&[F,\vec{\b v}_h]& \forall\, \vec{\b v}_h \in \mathbf{M}_h,\label{VFh-poro_3}\\
	&&&[B_2(\vec{\b \sigma}_h),q_h]&&&\,-\,&[D(p_h),q_h]&\;=\;& [G,q_h]&\forall\, p_h \in \mathrm{H}^{p}_h,\label{VFh-poro_4}\\
	&&&&&&&[\mathscr{A}_{\b \sigma_h}(\omega_h),\theta_h]&\;=\;& [J,\theta_h]&\forall\, \theta_h \in \mathrm{H}^{\omega}_h,\label{VFh_diff}
	\end{alignat}\end{subequations}
In order to address the well-posedness of (\ref{VFh-poro_1})-\eqref{VFh_diff}, we use again a fixed-point strategy. %Thus, %We now describe our fixed-point strategy to solve. 
Let us define $\textbf{S}_h:\mathrm{H}_h^{\omega}\to \mathbb{H}^{\b t}_h\times\mathbf{X}_h\times\mathbf{M}_h\times\mathrm{H}^{p}_h$ as
\begin{align*}
	\textbf{S}_h(\omega_h) & := (\textbf{S}_{1,h}(\omega_h), (\textbf{S}_{2,h}(\omega_h),\textbf{S}_{3,h}(\omega_h)),(\textbf{S}_{4,h}(\omega_h),\textbf{S}_{5,h}(\omega_h)),\textbf{S}_{6,h}(\omega_h))\\
	& \;=(\b t_h, (\b \sigma_h, \widetilde{p}_h),(\b u_h,\b \gamma_h),p_h)\in \mathbb{H}^{\b t}_h\times\mathbf{X}_h\times\mathbf{M}_h\times\mathrm{H}^{p}_h,\end{align*}
where $(\b t_h, (\b \sigma_h, \widetilde{p}_h),(\b u_h,\b \gamma_h),p_h)$, is the unique solution of (\ref{VFh-poro_1})-\eqref{VFh-poro_4} with $\omega_h$ given. In turn, let ${\widetilde{\mathbf S}}_h : \mathbb{H}_h^{\b \sigma}
\to \mathrm{H}^{\omega}_h$ be the operator defined by 
\[
{\widetilde{\mathbf S}}_h(\b \sigma_h) \,:= \, \omega_h 
\quad\forall\,\b \sigma_h\,\in\,  \mathbb{H}_h^{\b \sigma},\]
where $\omega_h$ is the unique solution of (\ref{VFh_diff}) with $\b \sigma_h$ given. Finally, by introducing the operator $\mbox{\textbf{{T}}}_h:\mathrm{H}^{\omega}_h\to \mathrm{H}^{\omega}_h$ as
\[
\mbox{\textbf{{T}}}_h(\omega_h):=\widetilde{\mathbf S}_h(\mathbf S_{2,h}(\omega_h))\quad\forall\,\omega_h \in \mathrm{H}^{\omega}_h, 
\]
we see that solving (\ref{VFh-poro_1})-\eqref{VFh_diff} is equivalent to seeking a fixed point of $\mbox{\textbf{{T}}}_h$, that is: find $\omega_h\in \mathrm{H}^{\omega}_h$ such that 
\begin{equation}\label{def-Th}
	\mbox{\textbf{{T}}}_h(\omega_h)=\omega_h. 
\end{equation}
\subsection{Well-definedness of the operator $\mathbf{T}_h$}\label{section:notation}
Here we establish the solvability of \eqref{VFh-poro_1}-\eqref{VFh_diff} by studying the equivalent fixed-point problem \eqref{def-Th}. We begin by introducing some needed notations and preliminary results, as well as specific finite element subspaces satisfying \eqref{arbitrary_FEM}.

Let us denote by $\mathcal{T}_h$ a regular partition of $\overline{\Omega}$ into triangles (or tetrahedra in 3D) $K$ of diameter $h_K$, where 
$h:=\max\left\{h_K:\ K\in \mathcal{T}_h \right\}$ is the meshsize. Given an integer $k\geq0,$ for each $K\in \mathcal{T}_h$ we let $\mathrm{P}_k(K)$ be the space of polynomial functions on $K$ of degree $\leq k$ and define the local Raviart-Thomas space of order $k$ as 
$\mathbf{RT}_k(K):=\mathbf{P}_k(K)\oplus \mathrm{P}_k(K)\,\b x$, 
where $\mathbf{P}_k(K)=[\mathrm{P}_k(K)]^d,$ and $\b x$ is a generic vector in $\mathbb{R}^d$. Now, let $b_K$ be the element bubble function defined as the unique polynomial in $\mathrm{P}_{d+1}(K)$ vanishing on $\partial K$ with $\int_{K}b_K=1$. Then, for each $K\in \mathcal{T}_h$ we consider the bubble space of order $k$, defined as 
\[\mathbf{B}_k(K):=
 \begin{cases}
\mathbf{curl}^{\mathrm{t}}(b_K\mathrm{P}_k(K))&\textrm{in}\quad\mathbb{R}^2,\\
\nabla \times (b_K\mathbf{P}_k(K))&\textrm{in} \quad \mathbb{R}^3.
\end{cases}.\]
On the other hand, we observe thanks to Lemma \ref{F-invert},
that the ellipticity of $A:\mathbb{H}^{\b t}_{h}\times \mathbb{H}^{\b t}_{h}\rightarrow \mathbb{R}$ is satisfied for any finite dimensional
subspace $\mathbb{H}^{\b t}_{h}$, and with the same constant from \eqref{ellipticity_TSP}. Therefore $\mathbb{H}^{\b t}_{h}$ is chosen such that the discrete inf-sup condition for   $B$ holds. Moreover, thanks to the discrete analogue of $B_1$ (cf. first equation in \eqref{def-A-B}), an inf-sup condition  can be ensured by using the classical PEERS$_{k}$ elements introduced in \cite{arnold84}, that is 
\begin{align}\label{peers}
\mathbb{H}^{\b \sigma}_{h}&:=\left\{{\b \tau}_h \in \mathbb{H}_{\Sigma}(\mathbf{div},\Omega):\quad {\b \tau}_{h}|_{K}\in [\mathbf{RT}_k(K)]^d\oplus[\mathbf{B}_k(K)]^d\quad \forall\, K\in \mathcal{T}_h\right\},\nonumber \\
\mathbf{H}^{\b u}_{h}&:= \left\{\b v_h \in \L2: \quad \b v_h|_{K}\in \textbf{P}_{k}(K)\quad \forall\, K\in \mathcal{T}_h\right\},\\
\mathbb{H}^{\b \gamma}_{h}&:=\left\{\b \eta_h\in \La:\quad \b \eta_h\in \mathbf{C}(\overline{\Omega})\quad \mathrm{and}\quad \b \eta_h|_{K}\in \mathbb{P}_{k+1}(K)\quad \forall\, K\in \mathcal{T}_h\right\},\nonumber
\end{align}
or by employing the well-known Arnold-Falk-Winther (AFW$_k$, \cite{afw-2007}) family of order $k\geq 0$, that is
\begin{align}\label{discrete-spaces-elasticity-AFW}
\mathbb{H}^{\b \sigma}_{h}&:=\left\{{\b \tau}_h \in \mathbb{H}_{\Sigma}(\mathbf{div},\Omega):\  {\b \tau}_{h}|_{K}\in \mathbf{BDM}_{k+1}(K)\quad \forall\, K\in \mathcal{T}_h\right\}, \\
\mathbf{H}^{\b u}_{h}&:= \left\{\b v_h \in \L2: \  \b v_h|_{K}\in \textbf{P}_{k}(K)  \quad \forall\, K\in \mathcal{T}_h\right\},\ 
\mathbb{H}^{\b \rho}_{h}:=\left\{\b \eta_h\in \La:\  \b \eta_h|_{K}\in \mathbb{P}_{k}(K) \quad \forall\, K\in \mathcal{T}_h\right\}.
\nonumber
\end{align}
Also, we notice that the kernel of ${B}_1$ is given by
\begin{equation}\label{kernel_B}
\mathbb{V}_h:=\left\{\vec{\b \tau}_h\in \mathbf{X}_h:\;\; [B_1(\vec{\b \tau}_h),\vec{\b u_h}]=0\quad \forall\,\vec{\b u}_h\in \mathbf{M}_h\right\}.
\end{equation}
Then, by using the definition given for $\mathbb{H}^{\b \sigma}_{h}$, $\mathbf{H}^{\b u}_{h},$ and $\mathbb{H}^{\b \gamma}_{h}$, \eqref{kernel_B} becomes 
$\mathbb{V}_h:={\mathbb{V}}^{\b \sigma}_h\times \mathrm{H}^p_h$, 
where 
\begin{equation}\label{kernel_sigma}
{\mathbb{V}}^{\b \sigma}_h:=\left\{{\b \tau}_h\in \mathbb{H}^{\b \sigma}_h:\;\; \mathbf{div}\,\b \tau_h=0\quad \mathrm{in}\quad \Omega\quad \mathrm{and}\quad \I\b \eta_h\colon\b \tau_h=0\quad \b \forall\,\b \eta_h\in \mathbb{H}^{\b \rho}_h\right\},
\end{equation}
and therefore, if we use the finite elements \eqref{peers}, we can proceed as in \cite[Section 2.4]{gatica13}, to extend the results given in \cite{gatica06} to the case $k\geq 1$, and define an appropriate space for $\mathbb{H}^{\b t}_{h}$ as 
\begin{align}\label{FEM_T}
\mathbb{H}^{\b t}_{h}&:=\left\{{\b r}_h \in \mathbb{L}^{2}(\Omega):\;\; {\b r}_{h}|_{K}\in \mathbb{P}_k(K)\oplus[\mathbf{B}_k(K)]^d\oplus([\mathbf{B}_k(K)]^d)^{\mathrm{d}}\quad \forall\, K\in \mathcal{T}_h\right\}, 
\end{align}
where $([\mathbf{B}_k(K)]^d)^{\mathrm{d}}$ stands for the deviatoric tensor of $[\mathbf{B}_k(K)]^d$. In turn, if   \eqref{discrete-spaces-elasticity-AFW} is employed, we simply take the part of the AFW$_k$ element that approximates $\b \sigma$ without requiring $\mathbb{H}(\mathbf{div},\Omega)$-
conformity, that is
\begin{align}\label{FEM_T1}
\mathbb{H}^{\b t}_{h}&:=\left\{{\b r}_h \in \mathbb{L}^{2}(\Omega):\;\; {\b r}_{h}|_{K}\in \mathbf{BDM}_{k+1}(K)\quad \forall\, K\in \mathcal{T}_h\right\}.
\end{align}
Additionally, note that $\mathrm{H}^{\widetilde{p}}_{h}$ does not require any specific condition. It is therefore  simply chosen as %the subspace of $\mathrm{L}^2(\Omega)$ 
\begin{equation}\label{FEM_PT1}
\mathrm{H}^{\widetilde{p}}_{h}:= \left\{\widetilde{q}_h \in \mathrm{L}^2(\Omega): \;\; \widetilde{q}_h|_{K}\in \mathrm{P}_{k}(K)\quad \forall\, K\in \mathcal{T}_h\right\}.
\end{equation}
Finally,  for  pressure and   solute concentration we consider Lagrange finite elements of degree $\leq k+1$, namely
\begin{subequations}
\begin{equation}\label{FE_pressure}
\mathrm{H}^{p}_{h}:= \left\{ q_h \in \mathrm{C}(\overline{\Omega})\,\cap\,\mathrm{H}^1_\Sigma(\Omega)\quad  q_h|_{K}\in \textrm{P}_{k+1}(K)\;\; \quad \forall\, K\in \mathcal{T}_h\right\},
\end{equation}
\begin{equation}\label{FE_concentration}
\mathrm{H}^{\omega}_{h}:= \left\{ \theta_h \in \mathrm{C}(\overline{\Omega})\,\cap\,\mathrm{H}^1_{\Gamma}(\Omega)\quad  \theta_h|_{K}\in \textrm{P}_{k+1}(K)\;\; \quad \forall\, K\in \mathcal{T}_h\right\}.
\end{equation}\end{subequations}

We now  require some preliminary results to establish the well-posedness of the operator $\mathbf{S}_h$.
\begin{lemma}
		There exists $\hat{\beta}_{1_d}>0$, independent of $h$, such that 
		\begin{equation}\label{inf_sup_B_h}\quad \sup_{\vec{\b \tau}_h\in \mathbf{X}_h \setminus\{\b 0\}} \frac{ [B_1(\vec{\b \tau}_h),\vec{\b v}_h]}{\norm{\vec{\b \tau}_h}_{\mathbf{X}}}\geq \hat{\beta}_{1_d}\norm{\vec{\b v}_h}_{\mathbf{M}}\quad \forall\, \vec{\b v}_h\in  \mathbf{M}_{h}.
		\end{equation}
\end{lemma}
\begin{lemma}\label{inequality_sigma_GS}
	There exists $c_3>0$, independent of $h$, such that 
	\begin{equation*}
	c_3\Vert\b \tau_h\Vert^2_{\mathbb{H}(\mathbf{div},\Omega)}\leq \Vert\b \tau_{0,h}\Vert^2_{\mathbb{H}(\mathbf{div},\Omega)}\quad \forall\, \b \tau_h \in \mathbb{H}_{h}^{\b \sigma}.	
	\end{equation*}
\end{lemma}
\begin{proof}
	It follows the same arguments from \cite[Lemma 2.4]{gatica06}, but now using that $\b \tau_h\b n=\b 0$ for all $\b \tau_h \in \mathbb{H}^{\b \sigma}_h$. 
\end{proof}
\begin{lemma}
There exists $\hat{\beta}_{d}>0,$ independent of $h$, such that for all \begin{equation}\label{inf_sup_B1_h}\quad \sup_{{\b r}_h\in \mathbb{H}^{\b t}_h \setminus\{\b 0\}} \frac{ [B({\b r}_h),\vec{\b \sigma}_h]}{\norm{{\b r}_h}_{\mathbb{L}^2(\Omega)}}\geq \hat{\beta}_{d}\norm{\vec{\b \sigma}_h}_{\mathbf{X}}\quad \forall\, \vec{\b \sigma}_h\in  \mathbb{V}_{h}.\end{equation}
\end{lemma}
\begin{proof}
The proof follows applying Lemmas \ref{lem:inequality} and \ref{inequality_sigma_GS}, and noticing that $\b \tau^{\mathrm{d}}_h$ and $(\widetilde{q}_h\mathbb{I}+\b \tau_h)\in \mathbb{H}^{\b t}_h$ for each $\vec{\b \tau}_h\in \mathbb{V}_h$. We refer to \cite[Lemma 2.5]{gatica06} for further details. 
\end{proof}
Now, we are in a position to establish the well-definedness of  $\mathbf{T}_h$. First we guarantee that the discrete problems defined by   $\mathbf{S}_h$ and $\widetilde{\mathbf{S}}_h$ are well-posed. In what follows, we show this result for $\mathbf{S}_h$, considering as in the continuous case, generic functionals in \eqref{VFh-poro_1}-\eqref{VFh-poro_4}, namely $\widetilde{H}_{w_h}, \widetilde{F}_{1,h}, \widetilde{F}_h$, and $\widetilde{G}_h$ instead of ${H}_{w_h}, \mathrm{O}, {F}$, and ${G}$, respectively. 
\begin{lemma}
	For each $\omega_h\in \mathrm{H}^{\omega}_h$ the problem \eqref{VFh-poro_1}-\eqref{VFh-poro_4} has a unique solution $(\b t_h, \vec{\b \sigma}_h,\vec{\b u}_h,p_h)\\\in \mathbb{H}^{\b t}_h\times\mathbf{X}_h\times\mathbf{M}_h\times\mathrm{H}^{p}_h$. Moreover, there exists $C_d>0$ independent of $h$ and $\lambda_s$, such that
	\begin{align}\label{bound_TSPh}
\nonumber
\Vert \mathbf{S}_h(\omega_h)\Vert &=\Vert \b t_h\Vert_{\mathbb{L}^2(\Omega)}+\Vert \vec{\b \sigma}_h\Vert_{\mathbf{X}}+\Vert \vec{\b u}_h\Vert_{\mathbf{M}}+\Vert p_h\Vert_{\mathrm{H}^1(\Omega)}\\ &\leq C_d\left(\Vert \widetilde{H}_{\omega_h}\Vert_{(\mathbb{H}^{\b t}_h)'}+\Vert\widetilde{F}_{1,h}\Vert_{\mathbf{X}'}+\Vert \widetilde{F}_h\Vert_{\mathbf{M}_h'}+\Vert \widetilde{G}_h\Vert_{(\mathrm{H}^{p}_{h})'}\right).\end{align}
\end{lemma}
\begin{proof}
It follows similar to the corresponding proof for the continuos case. In fact, for the well-posedness of the discrete version of problem \eqref{perturbed-problem}, we take advantage that for finite dimension, it suffices to prove that the homogeneous problem only has the trivial solution, which follows analogously to the proof of Lemma \ref{F-injective}, whereas for the corresponding stability result \eqref{bound_TSPh}, we proceed exactly as in Lemma \ref{stability-lemma}. Therefore, by noticing that $B$ satisfies the inf-sup condition \eqref{inf_sup_B_h}, the proof follows the same steps as in Lemma \ref{lemma_main} but now using the arguments given by \cite[Theorem 3.1]{gatica03}.  
\end{proof}
Now we establish the unique solvability of the nonlinear problem \eqref{VFh_diff}.
%equivalently, the well-definedness of the operator $\widetilde{\mathbf{S}}_h$ with a $\b \sigma_h\in \mathbb{H}^{\b \sigma}_h$ given.
\begin{lemma}
	For each $\b \sigma_h \in \mathbb{H}^{\b \sigma}_h$, the problem \eqref{VFh_diff} has a unique solution $\omega_h:= \widetilde{\mathbf{S}}_h(\b \sigma_h)\in \mathrm{H}^{\omega}_{h}$. Moreover, with the same constant $\widetilde{\alpha}$ given in Lemma \ref{well-posed-dif}, there holds
	\begin{equation}\label{stabilization-diff_lemma_h}
	\Vert\widetilde{\mathbf{S}}_h(\b \sigma_h )\Vert:=\Vert \omega_h\Vert_{\mathrm{H}^1(\Omega)}\leq \widetilde{\alpha}^{-1}\Vert J\Vert_{(\mathrm{H}^\omega_h)'}.
	\end{equation}
\end{lemma}
\begin{proof}
	It follows identically as in the proof of Lemma \ref{well-posed-dif} but now with elements living in $\mathrm{H}^{\omega}_h$ (cf. \eqref{FE_concentration}).
\end{proof}
\subsection{Discrete solvability analysis}
In this section we address the solvability of \eqref{def-Th}. We verify the hypotheses of the Brouwer fixed-point theorem to prove that \eqref{def-Th} has at least one fixed point. Most of the details can be omitted since they follow straightforwardly by adapting the results given in Section \ref{solvability_FP}.  

Let $\mathbf{W}_h:=\left\{\omega_h \in \mathrm{H}^\omega_h:\;\; \norm{\omega_h}_{\mathrm{H}^1(\Omega)}\leq r\right\}$, be a compact and convex subset of $\mathrm{H}^\omega_h$, with $r>0$ defined as in Lemma \ref{well-posed-dif}. We now provide the discrete analogues of Lemma \ref{bound_TW}.
\begin{lemma}\label{bound_TWh}
	For the closed ball $\mathbf{W}_h$, it holds that $\mathbf{T}_h(\mathbf{W}_h)\subseteq \mathbf{W}_h$. 
\end{lemma}
On the other hand, it is not difficult to prove that for the corresponding discrete case, the operators $\mathbf{S}$ and $\mathbf{S}_h$ satisfy the bounds given by Lemma \ref{Lipschitz_S_WS} with a constant $C_{\mathbf{S}_d}>0$ instead of $C_{\mathbf{S}}$, and therefore we have 
\begin{lemma}\label{lemma:LC_Th}
	There exists a constant $C_{\mathbf{T}_d}>0$, independent of $h$ and $\lambda_s$ such that 
	%\begin{equation}%\label{LC_Th}
\[	\Vert\mathbf{T}_h(\omega_{1,h})-\mathbf{T}_h(\omega_{2,h})\Vert_{\mathrm{H}^1(\Omega)}\leq C_{\mathbf{T}_d}d_3\beta\Vert\mathbf{T}_h(\omega_{2,h})\Vert_{W^{1,\infty}(\Omega)}\Vert\omega_{1,h}-\omega_{2,h}\Vert_{\mathrm{L}^2(\Omega)}  \quad \forall\, \omega_{1,h}, \omega_{2,h}\in \mathrm{H}^\omega_h.\]
	%\end{equation}
\end{lemma}
%We notice that $\mathrm{H}^\omega_h$ consists of piecewise polynomials (see Section \ref{section:notation}) and therefore, we can conclude that $\Vert\omega_h\Vert_{W^{1,\infty}(\Omega)}< +\infty$. 
Finally, thanks to Lemmas \ref{bound_TWh} and \ref{lemma:LC_Th}, and the continuity of $i_c$ (cf. proof of Lemma \ref{theorem:compact}) a straightforward application of the Brouwer
fixed-point theorem implies the main result of this section, stated as follows.
\begin{theorem}\label{main_discrete}
	Let $\mathbf{W}_h$ be defined as at the beginning of this section. Then, there exists $\overline{C}_d>0$ independent of $\lambda_s$ and $h$ such that the coupled problem $\mathrm{(\ref{VFh-poro_1})-\eqref{VFh_diff}}$ has at least one solution $(\b t_h, \vec{\b \sigma}_h,\vec{\b u}_h,p_h, \omega_h)$\newline $\in \mathbb{H}^{\b t}_h\times\mathbf{X}_h\times\mathbf{M}_h\times\mathrm{H}^p_h\times\mathrm{H}^\omega_h$ with $\omega_h\in \mathbf{W}_h$, satisfying the bound 
	\begin{equation*}%\label{bound-sigma-u-rho-Th}
	\Vert \b t_h\Vert_{\mathbb{L}^2(\Omega)}+\Vert \vec{\b \sigma}_h\Vert_{\mathbf{X}}+\Vert \vec{\b u}_h\Vert_{\mathbf{M}}+\Vert p_h\Vert_{\mathrm{H}^1(\Omega)}+\Vert \omega_h\Vert_{\mathrm{H}^1(\Omega)} \leq \overline{C}_d\left(\Vert F\Vert_{\mathbf{M}_h'}+\Vert G\Vert_{(\mathrm{H}^{p}_{h})'}+\Vert J\Vert_{(\mathrm{H}^\omega_h)'}\right).
	\end{equation*}
\end{theorem}
\begin{proof}
	 We proceed similarly as in the continuous case. In fact, from \eqref{bound_TSPh} and \eqref{stabilization-diff_lemma_h}, we get 
	\begin{align}\label{stabilization_works}
	&\Vert \b t_h\Vert_{\mathbb{L}^2(\Omega)}+\Vert \vec{\b \sigma}_h\Vert_{\mathbf{X}}+\Vert \vec{\b u}_h\Vert_{\mathbf{M}}+\Vert p_h\Vert_{\mathrm{H}^1(\Omega)}+\Vert \omega_h\Vert_{\mathrm{H}^1(\Omega)}\nonumber\\
	&\qquad \qquad\leq  \overline{C}_{1,d}
	\left(\Vert \widetilde{H}_{\omega_h}\Vert_{(\mathbb{H}^{\b t}_h)'}+\Vert\widetilde{F}_{1,h}\Vert_{\mathbf{X}_h'}+\Vert \widetilde{F}_h\Vert_{\mathbf{M}_h'}+\Vert \widetilde{G}_h\Vert_{(\mathrm{H}^{p}_{h})'}+\Vert J\Vert_{\mathrm{H}^1_\Gamma(\Omega)'}\right)\nonumber.
	\end{align}
	Then, the result follows by setting $\widetilde{H}_{\omega_h}=H_{\omega_h}, \widetilde{F}_{1,h}=\mathrm{O}, \widetilde{F}_h=F$ and $\widetilde{G}_h=G$, noticing that $\Vert H_{\omega_h}\Vert_{(\mathbb{H}^{\b t}_h)'}\leq \sqrt{d}\beta \Vert \omega_h\Vert_{\mathrm{H}^1(\Omega)}$, and applying the fact that $\omega_h\in \mathbf{W}_h$.
\end{proof}
\section{Error analysis}\label{error_analysis}
In this section, we derive the optimal a priori error
estimate. For this purpose, we first establish a C\'ea estimate formulated in the following theorem.
\begin{theorem}\label{theorem:a_priori}
	Let us $(\b t, \vec{\b \sigma},\vec{\b u},p, \omega)$ and $(\b t_h, \vec{\b \sigma}_h,\vec{\b u}_h,p_h, \omega_h)$ be the solutions of problems \eqref{VF-poro_1}-\eqref{VF_diff} and \eqref{VFh-poro_1}-\eqref{VFh_diff}, respectively, and assume that \begin{equation}\label{assumption_theorem}
		\hat{C}_{1,d}\beta<\frac{1}{2}, \quad \textrm{and}\quad \Vert\omega\Vert_{\mathrm{W}^{1,\infty}(\Omega)}\leq \hat{M}:=\min\{M,\widetilde{M}\}, 
	\end{equation} with $M$, $\hat{C}_{1,d}$ and $\widetilde{M}$ specified in \eqref{bound_M}, \eqref{existence_errors} and \eqref{assumption}, respectively. Then, there exists  $C_{C\textrm{\'e}a}>0$ such that
	\begin{align}\label{Cea_estimate}
\nonumber
	&\Vert \b t-\b t_h\Vert_{\mathbb{L}^2(\Omega)}+\Vert \vec{\b \sigma}-\vec{\b \sigma}_h\Vert_{\mathbf{X}}+\Vert \vec{\b u}-\vec{\b u}_h\Vert_{\mathbf{M}}+\Vert p-p_h\Vert_{\mathrm{H}^1(\Omega)}+\Vert \omega-\omega_h\Vert_{\mathrm{H}^1(\Omega)}\\
	&\quad \leq C_{C\textrm{\'e}a}\left(\mathrm{dist}(\b t,\mathbb{H}^{\b t}_h)+\mathrm{dist}(\vec{\b \sigma},\mathbb{H}^{\vec{\b \sigma}}_h)+\mathrm{dist}(\vec{\b u},\mathbf{H}^{\vec{\b u}}_h)+\mathrm{dist}(p,\mathrm{H}^{p}_h)+\mathrm{dist}(\omega,\mathrm{H}^{\omega}_h)\right).
	\end{align}
\end{theorem}
\begin{proof}
For sake of notational convenience we define $\mathbf{e}_{\b t}=\b t- \b t_h, \, \mathbf{e}_{\vec{\b \sigma}}=\vec{\b \sigma}- \vec{\b \sigma}_h, \, \mathbf{e}_{\vec{\b u}}=\vec{\b u}- \vec{\b u}_h$, $\mathrm{e}_{p}=p- p_h,$ $\mathrm{e}_{\omega}=\omega- \omega_h$. As usual, for arbitrary $(\hat{\b r_h}, \hat{\vec{\b \tau}}_h,\hat{\vec{\b v}}_h,\hat{q}_h, \hat{\theta}_h)\in \mathbb{H}^{\b t}_h\times\mathbf{X}_h\times\mathbf{M}_h\times\mathrm{H}^p_h\times\mathrm{H}^\omega_h$, they are decomposed as 
\begin{equation}\label{decomposition_errors}
	\mathbf{e}_{{\b t}}=\b \xi_{{\b t}}+\b \chi_{{\b t}}, \quad \mathbf{e}_{\vec{\b \sigma}}=\b \xi_{\vec{\b \sigma}}+\b \chi_{\vec{\b \sigma}},\quad \mathbf{e}_{\vec{\b u}}=\b \xi_{\vec{\b u}}+\b \chi_{\vec{\b u}}, \quad \mathrm{e}_{{p}}=\xi_{{p}}+\chi_{{p}}, \quad \textrm{and}\quad  \mathrm{e}_{{\omega}}=\xi_{{\omega}}+\chi_{{\omega}},
\end{equation}
where
\begin{gather*}
	\b \xi_{{\b t}}:=\b t-\hat{\b r}_h\in \mathbb{L}^2(\Omega), \  \b \chi_{{\b t}}:= \hat{\b r}_h-\b t_h\in \mathbb{H}^{\b t}_h,\ 
\b \xi_{\vec{\b \sigma}}:=\vec{\b \sigma}-\hat{\vec{\b \tau}}_h\in \mathbf{X}, \  \b \chi_{\vec{\b \sigma}}:= \hat{\vec{\b \tau}}_h-\vec{\b \sigma}_h\in \mathbf{X}_h,\ \b \chi_{\vec{\b u}}:= \hat{\vec{\b v}}_h-\vec{\b u}_h\in \mathbf{M}_h,\nonumber
\\
\xi_{{p}}:=p-\hat{q}_h\in \mathrm{H}^1_\Sigma(\Omega), \quad \chi_{{p}}:= \hat{q}_h-p_h\in \mathrm{H}^{p}_h, \quad 
\xi_{\omega}:=\omega-\hat{\theta}_h\in \mathrm{H}^1_\Gamma(\Omega), \quad \chi_{\omega}:= \hat{\theta}_h-\omega_h\in \mathrm{H}^{\omega}_h.\nonumber
\end{gather*}
Therefore, by subtracting  \eqref{VFh-poro_1}-\eqref{VFh_diff} and  \eqref{VF-poro_1}-\eqref{VF_diff}, we easily get the classical Galerkin orthogonality,
%\begin{subequations}
%	\begin{alignat}{11}
%	&[A(\mathbf{e}_{\b t}), \b r_h]&\,+\,&[B^*(\mathbf{e}_{\vec{\b \sigma}}),\b r_h]&&&&&\;=\;& [H_{\omega-\omega_h},\b r_h]&\quad \forall\, \b r_h\in \mathbb{H}^{\b t}_h,	\nonumber\\
%	&[B(\mathbf{e}_{\b t}), \vec{\b \tau}_h]&\,-\,& [C(\mathbf{e}_{\vec{\b \sigma}}),\vec{\b \tau}_h]&\,+\,&[B_1^*(\mathbf{e}_{\vec{\b u}}),\vec{\b \tau}_h]&\,+\,&[B_2^*(\mathrm{e}_{p}),\vec{\b \tau}_h]&\;=\;&\mathrm{O}& \forall\, \vec{\b \tau}_h \in \mathbf{X}_h,\nonumber\\
%	& &&[B_1(\mathbf{e}_{\vec{\b \sigma}}),\vec{\b v}_h]&&&&&\;=\;&\mathrm{O}& \forall\, \vec{\b v}_h \in \mathbf{M}_h,\nonumber\\
%	&&&[B_2(\mathbf{e}_{\vec{\b \sigma}}),q_h]&&&\,-\,&[D(\mathrm{e}_{p}),q_h]&\;=\;& \mathrm{O}&\forall\, p_h \in \mathrm{H}^{p}_h,\nonumber\\
%	&&&&&&&[\mathscr{A}_{\b \sigma_h}(\mathrm{e}_{\omega}),\theta_h]&\;=\;& [\mathscr{A}_{\b \sigma_h}(\omega)-\mathscr{A}_{\b \sigma}(\omega),\theta_h]&\forall\, \theta_h \in \mathrm{H}^{\omega}_h,\nonumber
%	\end{alignat}\end{subequations}
which together with the decompositions \eqref{decomposition_errors}, implies that
	\begin{alignat}{10}
	&[A({\b \chi}_{\b t}), \b r_h]&\,+\,&[B^*({\b \chi}_{\vec{\b \sigma}}),\b r_h]&&&&&\;=\;& [H^h+H_{\omega-\omega_h},\b r_h],\nonumber\\
	&[B({\b \chi}_{\b t}), \vec{\b \tau}_h]&\,-\,& [C({\b \chi}_{\vec{\b \sigma}}),\vec{\b \tau}_h]&\,+\,&[B_1^*({\b \chi}_{\vec{\b u}}),\vec{\b \tau}_h]&\,+\,&[B_2^*({\chi}_{p}),\vec{\b \tau}_h]&\;=\;&[\widetilde{F}^h_1,\vec{\b \tau}_h],\nonumber\\
	& &&[B_1({\b \chi}_{\vec{\b \sigma}}),\vec{\b v}_h]&&&&&\;=\;&[F^h,\vec{\b v}_h],\label{system_errors}\\
	&&&[B_2({\b \chi}_{\vec{\b \sigma}}),q_h]&&&\,-\,&[D(\chi_{p}),q_h]&\;=\;& [G^h,p_h],\nonumber\\
	&&&&&&&[\mathscr{A}_{\b \sigma_h}(\chi_{\omega}),\theta_h]&\;=\;& [J^h+\mathscr{A}_{\b \sigma_h}(\omega)-\mathscr{A}_{\b \sigma}(\omega),\theta_h],\nonumber
	\end{alignat}
for all $\b r_h\in \mathbb{H}^{\b t}_h, \vec{\b \tau}_h \in \mathbf{X}_h, \vec{\b v}_h \in \mathbf{M}_h, p_h \in \mathrm{H}^{p}_h,$ and $\theta_h \in \mathrm{H}^{\omega}_h,$ and where
\begin{gather*}
	[H^h,\b r_h]:=-[A({\b \xi}_{\b t}), \b r_h]-[B^*({\b \xi}_{\vec{\b \sigma}}),\b r_h], \quad [{F}^h,\vec{\b v}]:=-[B_1({\b \xi}_{\vec{\b \sigma}}),\vec{\b v}_h],\quad
[G^h,p_h]:=-[B_2({\b \xi}_{\vec{\b \sigma}}),q_h]+[D(\xi_{p}),q_h], \\ [J^h,\theta_h]:=-[\mathscr{A}_{\b \sigma_h}(\xi_{p}),\theta_h],\quad
	[\widetilde{F}^h_1,\vec{\b \tau}]:=-[B({\b \xi}_{\b t}), \vec{\b \tau}_h]+ [C({\b \xi}_{\vec{\b \sigma}}),\vec{\b \tau}_h]-[B_1^*({\b \xi}_{\vec{\b u}}),\vec{\b \tau}_h]-[B_2^*({\xi}_{p}),\vec{\b \tau}_h].
\end{gather*}
 Then,  proceeding as in the proof of Theorem \ref{main_discrete} but now for problem \eqref{system_errors}, we deduce that there exists $\hat{C}_{1,d}>0$, independent of $\lambda_s$ and $h$, such that
 \begin{align}\label{existence_errors}
 \nonumber
 	&\Vert \b \chi_{\b t}\Vert_{\mathbb{L}^2(\Omega)}+\Vert \b \chi_{\vec{\b \sigma}}\Vert_{\mathbf{X}}+\Vert \b \chi_{\vec{\b u}}\Vert_{\mathbf{M}}+\Vert \chi_{p}\Vert_{\mathrm{H}^1(\Omega)}+\Vert \chi_{\omega}\Vert_{\mathrm{H}^1(\Omega)}\\
 	& \;\;\leq \hat{C}_{1,d}\left(\Vert H^h+ H_{\omega-\omega_h}\Vert_{(\mathbb{H}^{\b t})'}+\Vert \widetilde{F}^h_1\Vert_{\mathbf{X}_h'}+ \Vert F^h\Vert_{\mathbf{M}_h'}+\Vert G^h\Vert_{(\mathrm{H}^{p}_{h})'}+\Vert J^h+\mathscr{A}_{\b \sigma_h}(\omega)-\mathscr{A}_{\b \sigma}(\omega)\Vert_{(\mathrm{H}^\omega_h)'}\right).
 \end{align} 
 In this way, we now proceed to bound the terms on the right-hand side of \eqref{existence_errors}. We begin by noticing thanks to the continuity of the operators $A$ and $B^*$, and the definition of $H_{\omega-\omega_h}$ (cf. \eqref{def-A-H}) that there holds
 \begin{equation}\label{error_H}
\Vert H^h+ H_{\omega-\omega_h}\Vert_{(\mathbb{H}^{\b t})'}\leq \hat{C}_1\left(\Vert\b \xi_{\b t}\Vert_{\mathbb{L}^2(\Omega)} +\Vert \b \xi_{\vec{\b \sigma}}\Vert_{\mathbf{X}}\right)+\beta\Vert \mathrm{e}_\omega\Vert_{\mathrm{H}^1(\Omega)},
 \end{equation}
 with $\hat{C}_1$ independent of $\lambda_s$ and $h$. Proceeding in a similar way as before, it can be deduced that
 \begin{align}\label{error_F_G}
 \nonumber
 	\Vert F^h\Vert_{\mathbf{M}_h'}&\leq\hat{C}_2\Vert \b \xi_{\vec{\b \sigma}}\Vert_{\mathbf{X}},\qquad \Vert G^h\Vert_{(\mathrm{H}^{p}_{h})'}\leq \hat{C}_3\left(\frac{\alpha}{\lambda_s}\Vert \b \xi_{\vec{\b \sigma}}\Vert_{\mathbf{X}}+\Vert\xi_p\Vert_{\mathrm{H}^1(\Omega)}\right),\\
 	\Vert \widetilde{F}^h_1\Vert_{\mathbf{X}_h'}&\leq \hat{C}_4\left(\Vert \b \xi_{\b t}\Vert_{\mathbb{L}^2(\Omega)}+\frac{1}{\lambda_s}\Vert \b \xi_{\vec{\b \sigma}}\Vert_{\mathbf{X}}+\Vert \b \xi_{\vec{\b u}}\Vert_{\mathbf{M}}+\frac{\alpha}{\lambda_s}\Vert\xi_{p}\Vert_{\mathrm{H}^1(\Omega)}\right),
 \end{align} 
where $\hat{C}_2, \hat{C}_3, \hat{C}_4>0$ are independent of $\lambda_s$ and $h$. On the other hand, in order to bound the last term on the right-hand side of \eqref{existence_errors}, we proceed as in \eqref{LC_omega} to get 
\[\Vert \mathscr{A}_{\b \sigma_h}(\omega)-\mathscr{A}_{\b \sigma}(\omega)\Vert_{(\mathrm{H}^\omega_h)'}\leq d_3\Vert\b \sigma-\b \sigma_h\Vert_{\mathbb{L}^2(\Omega)}\Vert \omega\Vert_{\mathrm{W}^{1,\infty}(\Omega)},\]
 and then, from the latter, we conclude that 
 \begin{equation}\label{error_J}
 	\Vert J^h+\mathscr{A}_{\b \sigma_h}(\omega)-\mathscr{A}_{\b \sigma}(\omega)\Vert_{(\mathrm{H}^\omega_h)'}\leq \hat{C}_5\Vert \xi_\omega\Vert_{\mathrm{H}^1(\Omega)}+d_3\Vert \mathbf{e}_{\vec{\b \sigma}}\Vert_{\mathbf{X}}\Vert \omega\Vert_{\mathrm{W}^{1,\infty}(\Omega)},
 \end{equation}
 with $\hat{C}_5$ independent of $\lambda_s$ and $h$. Therefore, from \eqref{error_H}, \eqref{error_F_G} and $\eqref{error_J}$, we obtain that there exists $\hat{C}_{2,d}>0$ independent of $\lambda_s$ and $h$, such that
  \begin{align}\label{existence_errors_final}
 \nonumber
 &\Vert \b \chi_{\b t}\Vert_{\mathbb{L}^2(\Omega)}+\Vert \b \chi_{\vec{\b \sigma}}\Vert_{\mathbf{X}}+\Vert \b \chi_{\vec{\b u}}\Vert_{\mathbf{M}}+\Vert \chi_{p}\Vert_{\mathrm{H}^1(\Omega)}+\Vert \chi_{\omega}\Vert_{\mathrm{H}^1(\Omega)}\\
 \nonumber 
 & \;\;\leq \hat{C}_{2,d}\left(\Vert \b \xi_{\b t}\Vert_{\mathbb{L}^2(\Omega)}+\left(1+\frac{1}{\lambda_s}+\frac{\alpha}{\lambda_s}\right)\Vert \b \xi_{\vec{\b \sigma}}\Vert_{\mathbf{X}}+\Vert \b \xi_{\vec{\b u}}\Vert_{\mathbf{M}}+\left(1+\frac{\alpha}{\lambda_s}\right)\Vert\xi_{p}\Vert_{\mathrm{H}^1(\Omega)}+\Vert \xi_\omega\Vert_{\mathrm{H}^1(\Omega)}\right)\\
 &\;\;\quad  +d_3\Vert \mathbf{e}_{\vec{\b \sigma}}\Vert_{\mathbf{X}}\widetilde{M}+\hat{C}_{1,d}\beta\Vert \mathrm{e}_\omega\Vert_{\mathrm{H}^1(\Omega)},
 \end{align} 
 with $\Vert \omega\Vert_{\mathrm{W}^{1,\infty}(\Omega)}\leq \widetilde{M}$, and where $\left(1+\frac{1}{\lambda_s}+\frac{\alpha}{\lambda_s}\right)$ and $\left(1+\frac{\alpha}{\lambda_s}\right)$ can be seen as constants independent of $\lambda_s$ if $\lambda_s\rightarrow \infty$. In this way, choosing $\widetilde{M}$ such that 
 \begin{equation}\label{assumption}
 	\widetilde{M}\leq \frac{1}{2d_3}, 
 \end{equation} 
 the desired result follows simply by   triangle's inequality in \eqref{decomposition_errors}, the estimation \eqref{existence_errors_final}, and  assumption \eqref{assumption_theorem}.
\end{proof}
 The main result of this section is given by the following lemma.
\begin{theorem}\label{approximation}
	In addition to the hypotheses of Theorems $\mathrm{\ref{theorem:solvability_continuous}}$, $\mathrm{\ref{main_discrete}}$ and $\mathrm{\ref{theorem:a_priori}}$, assume that there exist $s>0$ and $l>1/2$ such that $\b \sigma\in \mathbb{H}^s(\Omega)$, $\mathbf{div}\,\b \sigma\in \bH^s(\Omega), \,\b u \in \bH^s(\Omega), \,\b \gamma\in \mathbb{H}^s(\Omega), \,\widetilde{p}\in \mathrm{H}^s(\Omega),\, p\in \mathrm{H}^{1+s}(\Omega),\, \omega\in \mathrm{H}^{1+s}(\Omega),$ and $\b t\in \mathbb{H}^l(\Omega)$. Then, there exist $\widetilde{C}_1, \widetilde{C}_2>0$, independent of $h$, such that, with the finite element subspaces defined by $\mathrm{(\ref{discrete-spaces-elasticity-AFW})}$, $\mathrm{(\ref{FEM_T1})}$, $\mathrm{(\ref{FEM_PT1})}$, $\mathrm{(\ref{FE_pressure})},$ and $\mathrm{(\ref{FE_concentration})},$ there holds
	\begin{align}
	\begin{split}\label{rates}
	&\Vert \b t-\b t_h\Vert_{\mathbb{L}^2(\Omega)}+\Vert \vec{\b \sigma}-\vec{\b \sigma}_h\Vert_{\mathbf{X}}+\Vert \vec{\b u}-\vec{\b u}_h\Vert_{\mathbf{M}}+\Vert p-p_h\Vert_{\mathrm{H}^1(\Omega)}+\Vert \omega-\omega_h\Vert_{\mathrm{H}^1(\Omega)} \\
	&\qquad \le\widetilde{C}_1h^{\mathrm{min}\{s,k+1\}}\,
	\left\{\norm{\b \sigma}_{\mathbb{H}^s(\Omega)}+ \,\, \norm{\mathbf{div}\,\b \sigma}_{s,\Omega} \, + \, \norm{\b u}_{s,\Omega}+\norm{\b \gamma}_{\mathbb{H}^s(\Omega)}
	+\Vert\widetilde{p}\Vert_{\mathrm{H}^s(\Omega)}
	\right.\\
	&\qquad\quad \left. +\,\Vert p \Vert_{\mathrm{H}^s(\Omega)}+\norm{\omega}_{\mathrm{H}^{s+1}(\Omega)}\right\}+\widetilde{C}_2h^{\mathrm{min}\{l,k+1\}}\norm{\b t}_{\mathbb{H}^l(\Omega)}.
	\end{split}
	\end{align}
\end{theorem}
\begin{proof}
It follows as a combination of the approximation properties of the spaces (\ref{discrete-spaces-elasticity-AFW}), (\ref{FEM_T1}), (\ref{FEM_PT1}), (\ref{FE_pressure}), and (\ref{FE_concentration}) (see, e.g., \cite{boffi13, gatica14}), the C\'ea estimate \eqref{Cea_estimate} and the smallness assumption. %We omit further details.
\end{proof}
\begin{remark}
Similar results are obtained when   \eqref{peers} and \eqref{FEM_T} are employed instead of \eqref{discrete-spaces-elasticity-AFW} and \eqref{FEM_T1}, respectively. For the lowest-order case, we notice from the classical approximation properties of the finite element subspaces, that it suffices to consider $l>0$, and therefore, the last term in \eqref{rates} should be changed by $\widetilde{C}_1h^{\mathrm{min}\{s,k+1\}}\sum_{K\in \mathcal{T}_h }\norm{\b t}_{\mathbb{H}^s(K)}$ (see \cite[Theorem 3.2]{gatica06}), whereas the other terms remain unchanged. Even if approximation results for   \eqref{FEM_T} in the high-order case do not seem to be available in the literature, the numerical results in the next section demonstrate an asymptotic  $\mathcal{O}(h^{k+1})$ convergence  (see Table \ref{table:h}).
\end{remark}

%***********************************************************************************
\section{Numerical tests} \label{sec:results}
%***********************************************************************************
%****************************
\subsection{Example 1: Convergence against manufactured solutions}
%****************************
The accuracy of the  discretisation is verified through two convergence tests using manufactured solutions, and implemented using the FEniCS library \cite{logg12}. The model parameters are simply taken as   $\mu_s = \lambda_s = 1 = c_0 = \kappa = \alpha = \ell = \rho_f = \rho_s = \mu_f =  \beta = 1$, and $\phi = 0.5$, and $\gg=\cero$. Variations (of several orders of magnitude) in $\lambda_s, \kappa, \alpha, \mu_f$  will be also considered to study the robustness of the formulation with respect to these parameters. The stress-altered diffusion is \eqref{eq:D0} with $D_0 = 0.01$ and $\eta = 0.01$. We use the following closed-form smooth solutions to  \eqref{eq:model} 
\begin{gather*}
\omega = e^{x_1}+\cos(\pi x_1)\cos(\pi x_2),\ 	\bu = \frac{1}{10}\begin{pmatrix}
	-\cos(x_1)\sin(x_2)+\frac{x_1^2}{\lambda_s}\\
	\sin(x_1)\cos(x_2)+\frac{x_2^2}{\lambda_s}
	\end{pmatrix}, \  
	p = \sin(\pi x_1)\sin(\pi x_2), \   \widetilde{p} = \alpha p-\lambda_s\vdiv\bu, 
\end{gather*}
which are used to produce non-homogeneous forcing and source terms. For sake of simplicity only one type of boundary conditions is considered, i.e., $\partial\Omega = \Gamma$.%, and so only the tracer concentration is imposed essentially. 
The remaining boundary conditions in \eqref{eq:tildep} (for displacement and fluid flux) are imposed naturally. This also implies that the stress is not uniquely defined and we constraint its trace, using a real Lagrange multiplier.  
We generate seven successively refined meshes (congruent right-angled triangular partitions) for the domain $\Omega =(0,1)^2$ and calculate errors between approximate and exact solutions $e(\cdot)$  (measured in $\rH^1$-norm for tracer and fluid pressure, in the tensor $\mathbf{H}(\vdiv)$-norm for stress, and in the tensorial, vectorial and scalar $\rL^2$-norm for strain, rotation, displacement and total pressure). The mixed finite element methods are defined by the PEERS$_{k}$ and enriched piecewise polynomial spaces specified in  \eqref{peers},\eqref{FEM_T},\eqref{FE_pressure}.

\begin{table}[t!]
	\setlength{\tabcolsep}{2.5pt}
	\begin{center}  
	{\small \begin{tabular}{|ccccccccccccccc|}
	\hline 	
	DoF  & $e_0(\bt)$ & rate & $e_{\bdiv}(\bsigma)$ & rate & $e_1(\tilde{p})$ & rate  & $e_0(\bu)$ & rate & $e_0(\bgamma)$ & rate & $e_1(p)$ & rate & $e_1(\omega)$ & rate\\ 
	\hline
	\hline
\multicolumn{15}{|c|}{scheme with $k=0$}\\
				\hline
		   148 & 1.9e-01 & -- & 1.2e+0 & -- & 2.9e-01 & -- & 2.8e-02 & -- & 3.3e-01 & -- & 1.1e+0 & -- & 1.6e+0 & -- \\
   540 & 8.7e-02 & 1.15 & 6.2e-01 & 0.99 & 1.4e-01 & 1.03 & 9.0e-03 & 1.64 & 9.8e-02 & 1.76 & 7.4e-01 & 0.62 & 8.9e-01 & 0.87 \\
  2068 & 4.2e-02 & 1.05 & 3.1e-01 & 1.00 & 6.7e-02 & 1.06 & 3.9e-03 & 1.19 & 3.9e-02 & 1.41 & 4.1e-01 & 0.85 & 4.6e-01 & 0.97 \\
  8100 & 2.1e-02 & 1.03 & 1.6e-01 & 1.00 & 3.3e-02 & 1.03 & 1.9e-03 & 1.06 & 1.8e-02 & 1.35 & 2.1e-01 & 0.95 & 2.3e-01 & 0.99 \\
 32068 & 1.0e-02 & 1.02 & 7.8e-02 & 1.00 & 1.6e-02 & 1.01 & 9.4e-04 & 1.01 & 7.8e-03 & 1.18 & 1.1e-01 & 0.98 & 1.2e-01 & 1.00 \\
127620 & 5.0e-03 & 1.01 & 3.9e-02 & 1.00 & 8.2e-03 & 1.00 & 4.7e-04 & 1.00 & 3.7e-03 & 1.10 & 5.4e-02 & 1.00 & 5.8e-02 & 1.00	\\	
509188 & 2.5e-03 & 1.01 & 1.9e-02 & 1.00 & 4.1e-03 & 1.00 & 2.3e-04 & 1.00 & 1.5e-03 & 1.20 & 2.7e-02 & 1.00 & 2.9e-02 & 1.00 \\
			\hline
\multicolumn{15}{|c|}{scheme with $k=1$}\\
				\hline	
   436 & 1.3e-02 & -- & 2.3e-01 & -- & 7.7e-02 & -- & 1.8e-03 & -- & 6.9e-03 & -- & 3.8e-01 & -- & 4.8e-01 & --\\ 
  1652 & 4.6e-03 & 1.51 & 5.8e-02 & 1.97 & 2.0e-02 & 1.95 & 4.4e-04 & 2.02 & 2.4e-03 & 1.35 & 1.2e-01 & 1.70 & 1.3e-01 & 1.88 \\
  6436 & 1.3e-03 & 1.76 & 1.5e-02 & 1.99 & 5.0e-03 & 1.99 & 1.0e-04 & 2.07 & 8.7e-04 & 1.44 & 3.2e-02 & 1.88 & 3.3e-02 & 1.96 \\
 25412 & 3.6e-04 & 1.91 & 3.7e-03 & 2.00 & 1.2e-03 & 2.00 & 2.5e-05 & 2.04 & 2.5e-04 & 1.78 & 8.2e-03 & 1.95 & 8.4e-03 & 1.99 \\
100996 & 9.2e-05 & 1.96 & 9.2e-04 & 2.00 & 3.1e-04 & 2.00 & 6.3e-06 & 2.01 & 6.7e-05 & 1.91 & 2.1e-03 & 1.98 & 2.1e-03 & 2.00 \\
402692 & 2.3e-05 & 1.98 & 2.3e-04 & 2.00 & 7.8e-05 & 2.00 & 1.6e-06 & 2.00 & 1.7e-05 & 1.96 & 5.2e-04 & 1.99 & 5.3e-04 & 2.00 \\
1608196 & 5.9e-06 & 1.99 & 5.7e-05 & 2.00 & 1.9e-05 & 2.00 & 3.9e-07 & 2.00 & 4.4e-06 & 1.98 & 1.3e-04 & 2.00 & 1.3e-04 & 2.00 \\
\hline
		\end{tabular}}
		\end{center}
		
		\smallskip
		\caption{Verification of convergence for the method with PEERS$_{k}$ and enriched piecewise polynomials with degrees $k=0$ and $k=1$, and using unity parameters. Errors and convergence rates are tabulated for strain, stress, total pressure, displacement, rotation Lagrange multiplier, fluid pressure, and tracer concentration.}\label{table:h}
	\end{table}

Table~\ref{table:h} shows such error decay for polynomial orders $k=0$ and $k=1$, and we can clearly observe a convergence of $O(h^{k+1})$ for all field variables in their natural norms, which is consistent with the theoretical error bounds. Table~\ref{table:3D} shows the error history associated with the extension to 3D, using the exact solutions 
\begin{gather*}
\omega = \cos(\pi x_1)\cos(\pi x_2)\cos(\pi x_3),\ 	\bu = \frac{1}{10}\begin{pmatrix}
	\sin(x_1)\cos(x_2)\cos(x_3)+\frac{x_1^2}{\lambda_s}\\
	-2\cos(x_1)\sin(x_2)\cos(x_3)+\frac{x_2^2}{\lambda_s}\\
	\cos(x_1)\cos(x_2)\sin(x_3)+\frac{x_3^2}{\lambda_s}
	\end{pmatrix}, \  
	p = \sin(\pi x_1)\sin(\pi x_2)\sin(\pi x_3),  
\end{gather*}
and $\widetilde{p} = \alpha p-\lambda_s\vdiv\bu$, on   $\Omega=(0,1)^3$, with the same unit parameters, and focusing on the lowest-order scheme. Optimal, first-order convergence is also achieved in this case. For completeness, we also verify the error decay of the scheme resulting from using AFW$_k$ elements \eqref{discrete-spaces-elasticity-AFW} and \eqref{FEM_T1}, which also delivers optimal convergence.

\begin{table}[t!]
	\setlength{\tabcolsep}{2.5pt}
	\begin{center}  
	{\small \begin{tabular}{|ccccccccccccccc|}
	\hline 	
	DoF  & $e_0(\bt)$ & rate & $e_{\bdiv}(\bsigma)$ & rate & $e_1(\tilde{p})$ & rate  & $e_0(\bu)$ & rate & $e_0(\bgamma)$ & rate & $e_1(p)$ & rate & $e_1(\omega)$ & rate\\ 
	\hline
		\hline
\multicolumn{15}{|c|}{scheme with PEERS$_{k}$ and enriched piecewise polynomials}\\
				\hline
  1984 & 1.7e-01 & --   & 1.5e+0 & 0.00 & 2.0e-01 & -- & 2.0e-02 & -- & 1.4e-01 & -- & 1.4e+0 & -- & 1.6e+0 & -- \\
  6477 & 9.8e-02 & 0.72 & 1.0e+0 & 0.87 & 1.5e-01 & 0.74 & 1.3e-02 & 0.99 & 8.3e-02 & 1.29 & 9.4e-01 & 0.63 & 1.2e+0 & 0.73 \\
 29281 & 6.4e-02 & 0.83 & 6.4e-01 & 0.95 & 8.9e-02 & 0.98 & 7.7e-03 & 1.08 & 2.7e-02 & 2.16 & 6.7e-01 & 0.76 & 7.7e-01 & 0.90 \\
168297 & 3.5e-02 & 0.92 & 3.6e-01 & 0.99 & 4.7e-02 & 1.09 & 4.1e-03 & 1.07 & 1.0e-02 & 1.71 & 4.1e-01 & 0.84 & 4.3e-01 & 0.98 \\
1125049 & 2.1e-02 & 0.95 & 1.9e-01 & 1.00 & 2.4e-02 & 1.06 & 2.1e-03 & 1.04 & 3.7e-03 & 1.59 & 2.2e-01 & 0.96 & 2.3e-01 & 1.00 \\
8194137 & 1.2e-02 & 0.96 & 9.6e-02 & 1.01 & 1.3e-02 & 1.02 & 1.1e-03 & 1.00 & 1.6e-03 & 1.49 & 1.2e-01 & 0.97 & 1.2e-01 & 0.98 \\
\hline
\multicolumn{15}{|c|}{scheme with AFW$_k$ elements}\\
	\hline
  2551 & 7.1e-02 & -- & 1.4e+0 & -- & 1.9e-01 & -- & 1.9e-02 & -- & 3.8e-02 & -- & 1.1e+0 & -- & 1.6e+0 & -- \\
  8067 & 4.3e-02 & 1.24 & 1.0e+0 & 0.90 & 1.4e-01 & 0.78 & 1.2e-02 & 1.04 & 2.2e-02 & 1.38 & 9.4e-01 & 0.43 & 1.2e+0 & 0.74\\ 
 35383 & 2.3e-02 & 1.23 & 6.1e-01 & 0.96 & 8.5e-02 & 0.99 & 7.4e-03 & 1.02 & 1.1e-02 & 1.35 & 6.7e-01 & 0.66 & 7.7e-01 & 0.91 \\
198831 & 1.0e-02 & 1.34 & 3.4e-01 & 0.99 & 4.6e-02 & 1.06 & 4.0e-03 & 1.02 & 5.0e-03 & 1.32 & 4.1e-01 & 0.84 & 4.3e-01 & 0.98 \\
1310431 & 4.6e-03 & 1.31 & 1.8e-01 & 1.00 & 2.4e-02 & 1.03 & 2.1e-03 & 1.01 & 2.3e-03 & 1.22 & 2.3e-01 & 0.94 & 2.3e-01 & 0.99 \\
				\hline
		\end{tabular}}
		\end{center}
		
		\smallskip
		\caption{Verification of space convergence   in 3D, with polynomial degree $k=0$, and using unity parameters. 
		Errors history for strain, stress, total pressure, displacement, rotation Lagrange multiplier, fluid pressure, and tracer concentration.}\label{table:3D}
	\end{table}

Figure~\ref{fig:h} illustrates  the robustness with respect to large variations in physical parameters including nearly incompressible materials (taking $\lambda_s=10^8$), for low permeability (with $\kappa = 10^{-12}$), for weak Biot-Willis coupling ($\alpha=10^{-6}$), and with small fluid viscosity (with $\mu_f=10^{-4}$). Analogous results (not shown here) were obtained for zero storage coefficient, and for different values of porosity.  
For all these tests (both 2D and 3D), in average, four Newton-Raphson iterations were required to reach convergence across all mesh refinements.

\begin{figure}[t!]
\begin{center}
\includegraphics[width=0.24\textwidth]{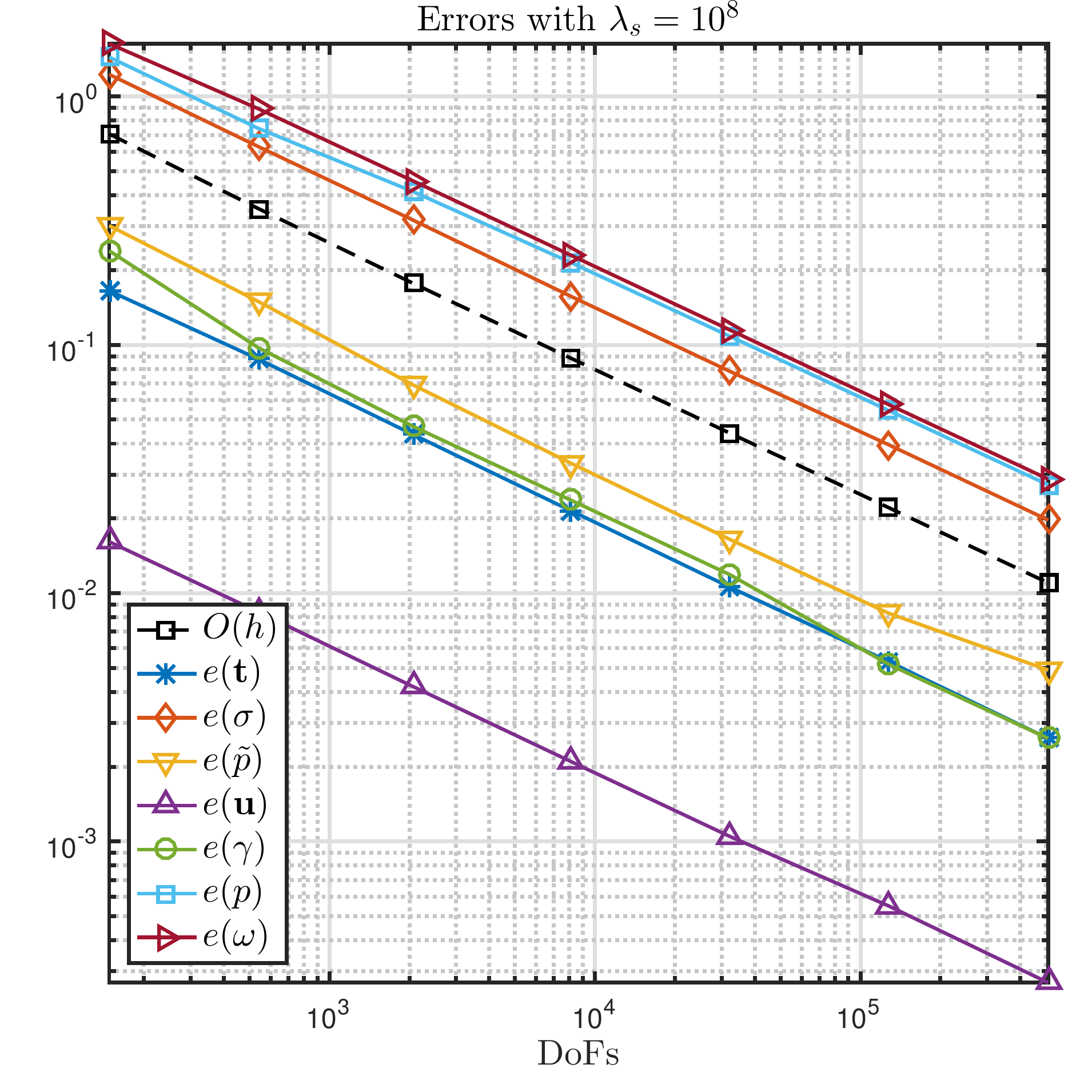}
\includegraphics[width=0.24\textwidth]{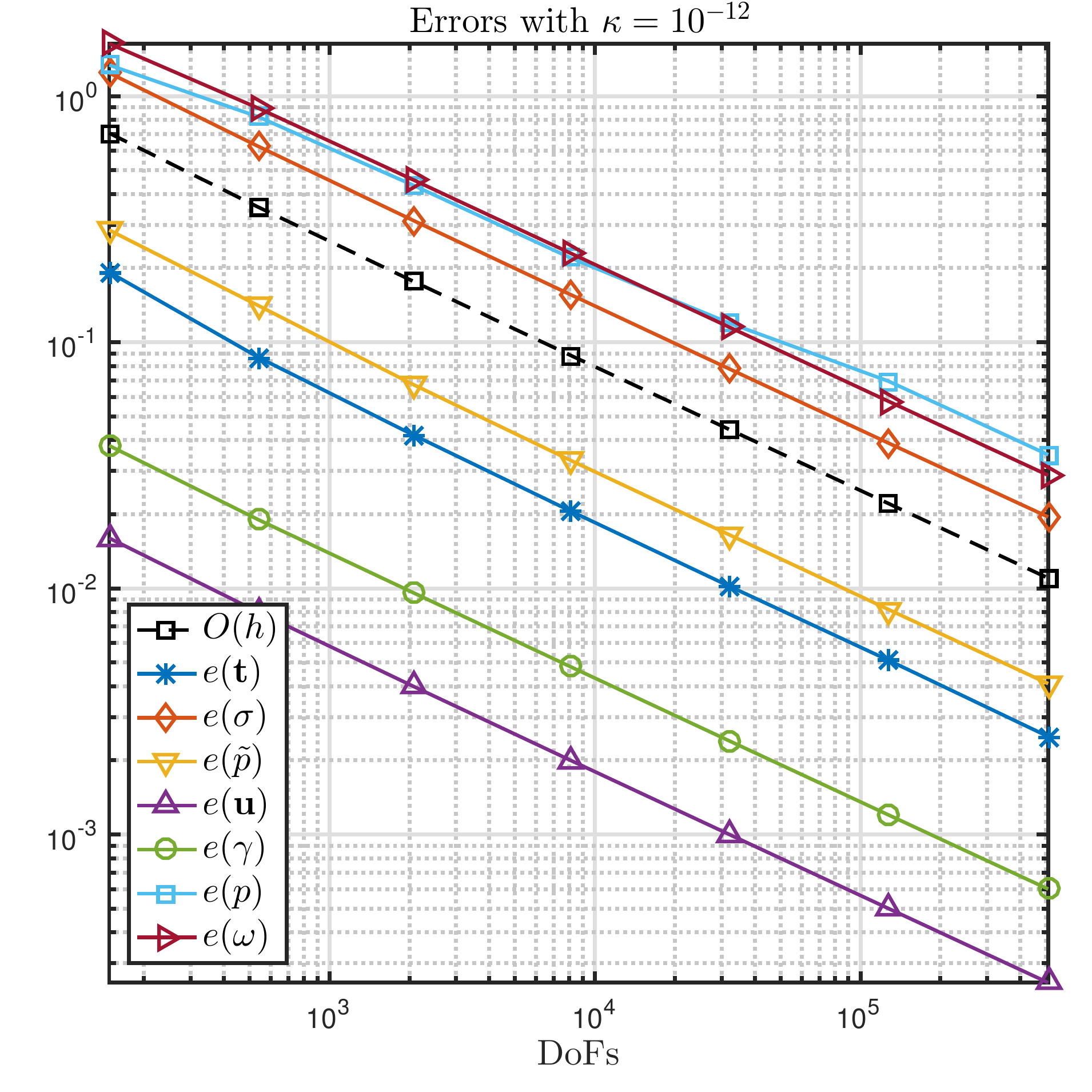}
\includegraphics[width=0.24\textwidth]{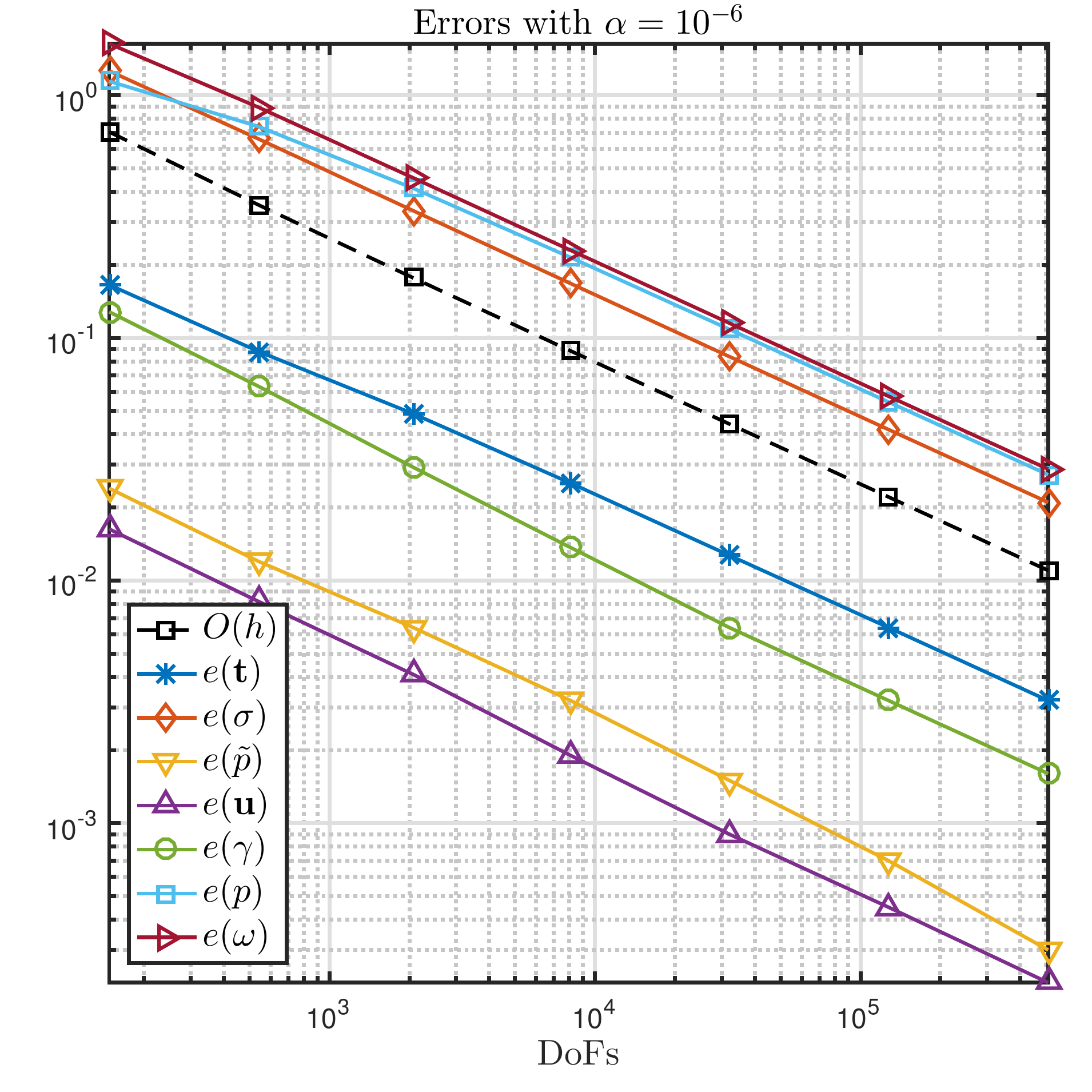}
\includegraphics[width=0.24\textwidth]{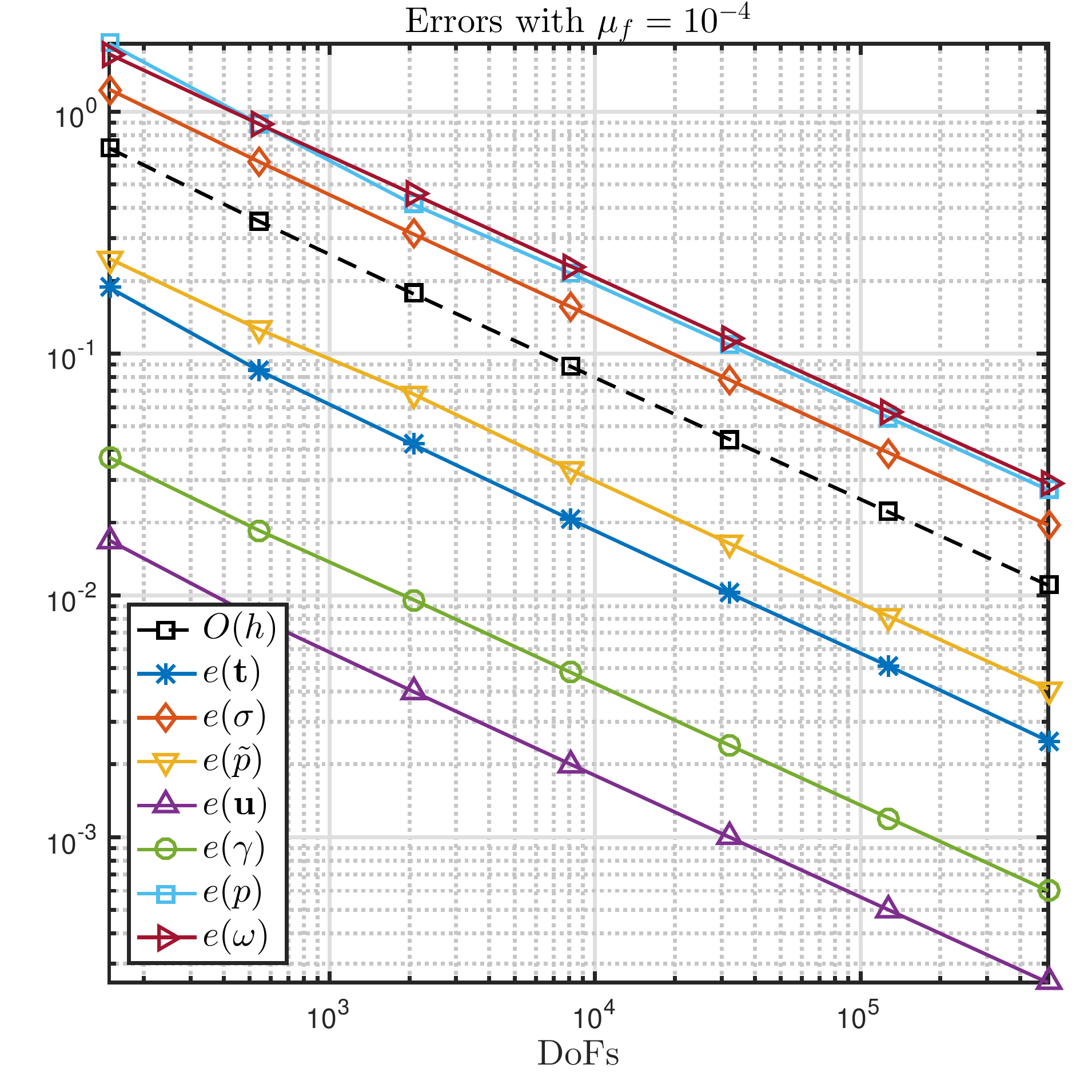}
\end{center}
\caption{Verification of space convergence   and robustness with respect to changes in model parameters $\lambda_s,\kappa,\alpha,\mu_f$.}\label{fig:h}
\end{figure}

%****************************
\subsection{Example 2: Testing the effect of stress-hindered diffusion} 
In this test we compare the filtration effects on models using different degrees of stress-modified diffusion where the diffusion due to stress is lower than the constant base-line diffusion. This simple test emphasises the significance of anisotropy and heterogeneity in the transport of tracer and it is more illustrative to consider the time-dependent case (adding a time derivative of the first two terms in the second equation of \eqref{poro1} and of the first term in \eqref{eq:w0}, which we discretise using backward Euler's method with a constant time step. For this we employ a slab of tissue of size 1\,mm$^2$, and the model parameters are as follows  (see, e.g., \cite{holter17,sykova08,budday2015mechanical})
$$
E = 800\,\text{Pa}, \quad 
\nu = 0.495,\quad 
c_0 = 2\times 10^{-8}, \quad 
\kappa = 10^{-8}\,\text{mm}^2,\quad 
\alpha = 1, \quad \ell = 0,\quad 
\rho_s = 10^{-3}\,\text{mg/mm}^3,$$ $$
\mu_f = 0.7\,\text{Pa s},\quad \beta = 0.45,\quad 
\phi = 0.2, \quad D_0 = 5.3\times 10^{-5}\text{mm}^2/\text{s},$$  
where $D_0$ is made sufficiently large to compare with stress-hindered effects. On the top of the slab we impose a traction $\bsigma\nn =  - \alpha p_{\mathrm{top}}\nn$ with $p_{\mathrm{in}} = 0.5\,\mathrm{atan}(t/10)$\,mmHg/mm,  on the bottom we clamp the tissue $\bu = \cero$, and on the vertical walls we set zero traction. The fluid pressure $p_{\mathrm{in}}$ is prescribed on the top segment and set to $p_0 = 9$\,mmHg/mm on the vertical walls, whereas we set zero normal flux on the bottom. The tracer concentration $\omega_{\mathrm{in}} = 1$\,mmol is fixed on the top and zero diffusive flux is considered elsewhere on the boundary. A coarse mesh with 16384 triangular elements is employed and the time step is $\Delta t = 50$\,s.

\begin{figure}[!t]
\begin{center}
\includegraphics[width = 0.23\textwidth]{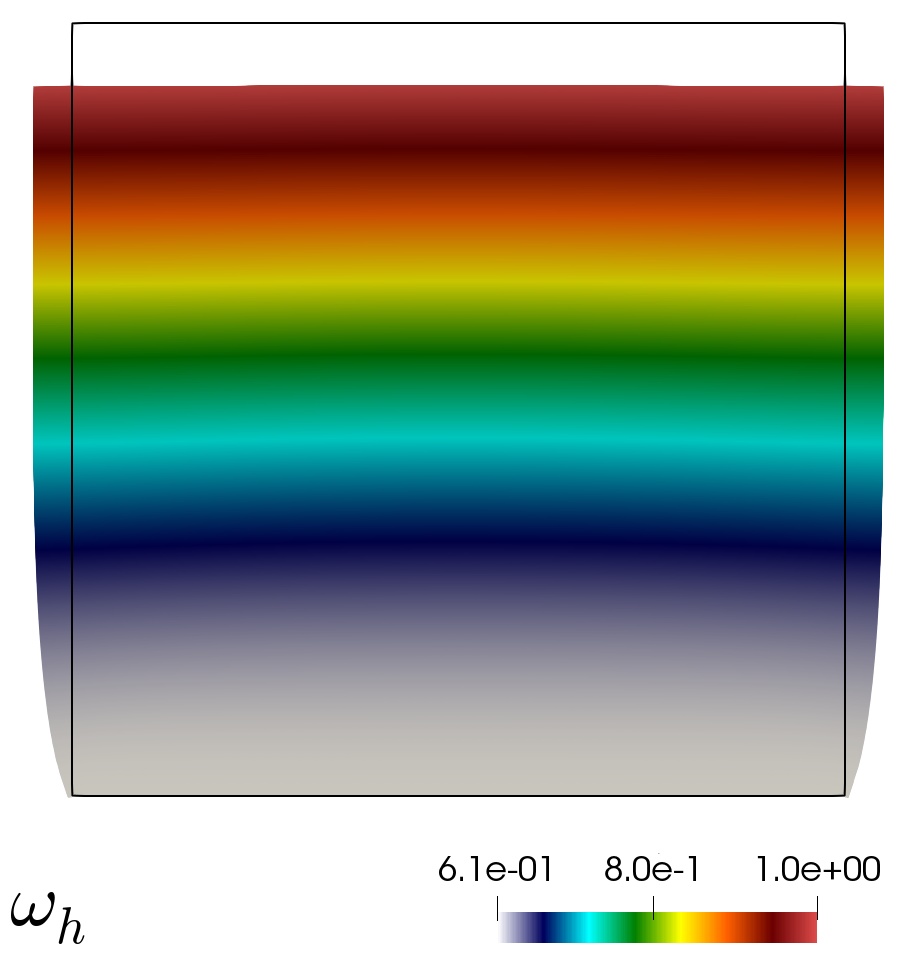}
\includegraphics[width = 0.23\textwidth]{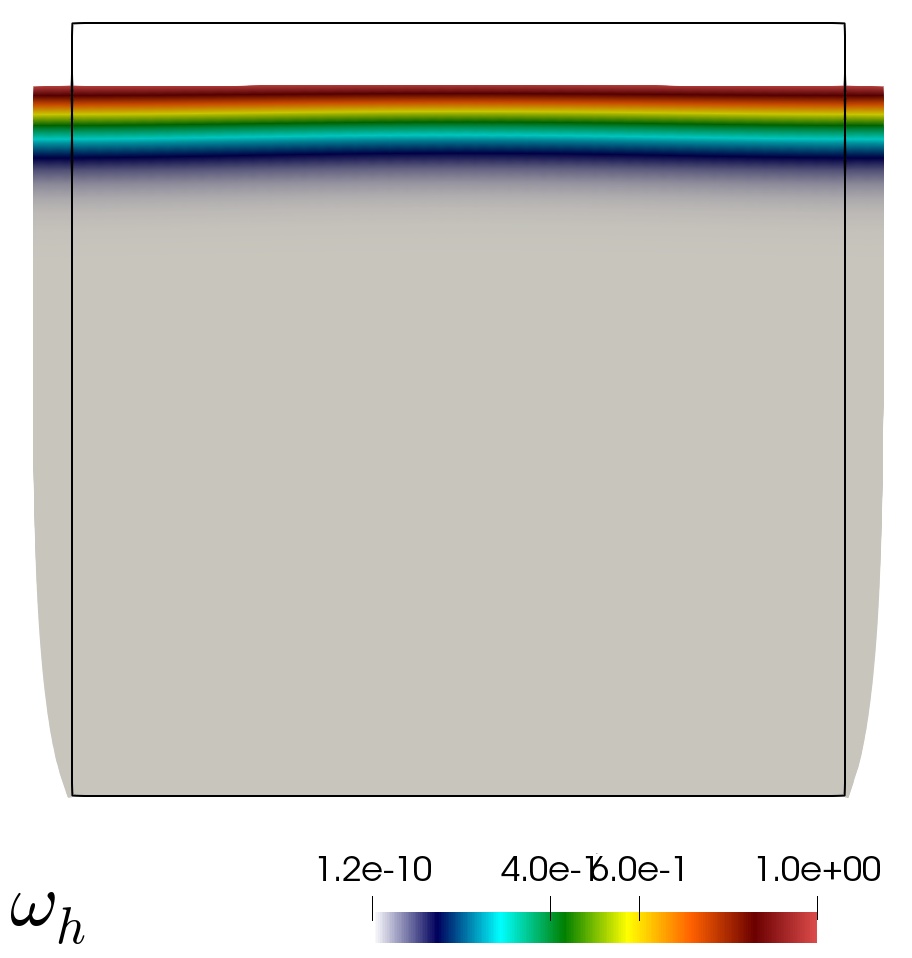}
\includegraphics[width = 0.23\textwidth]{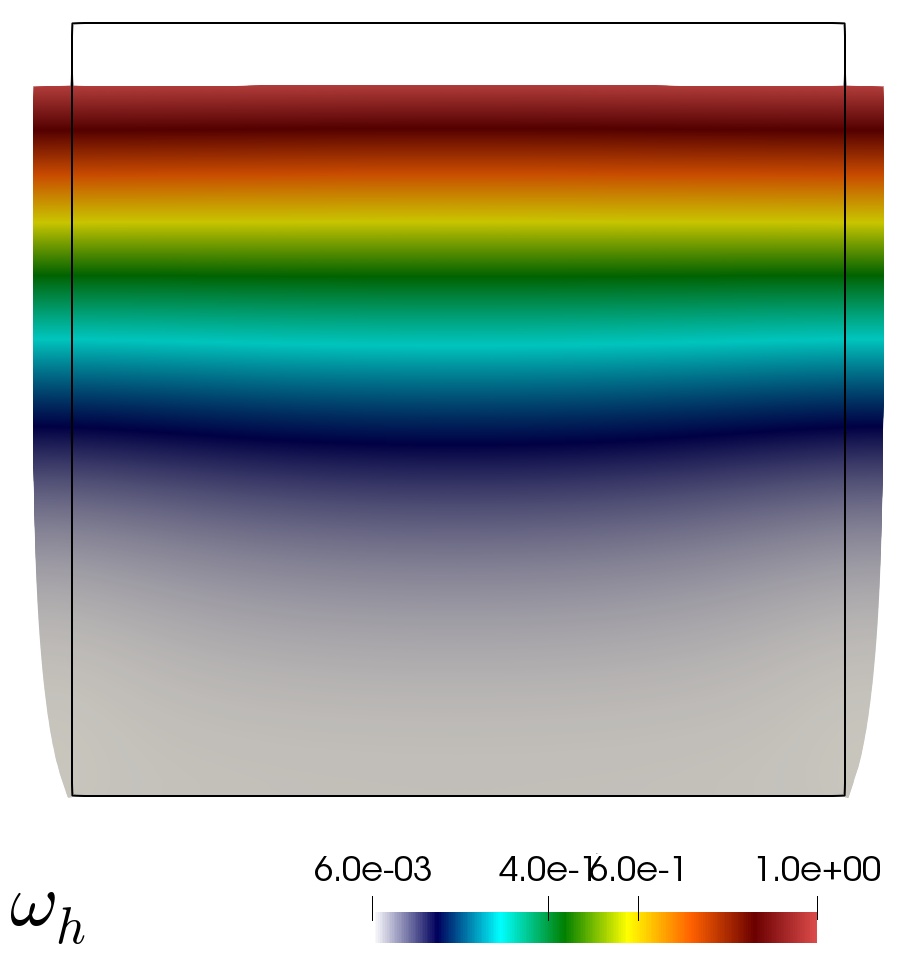}\\
\includegraphics[width = 0.23\textwidth]{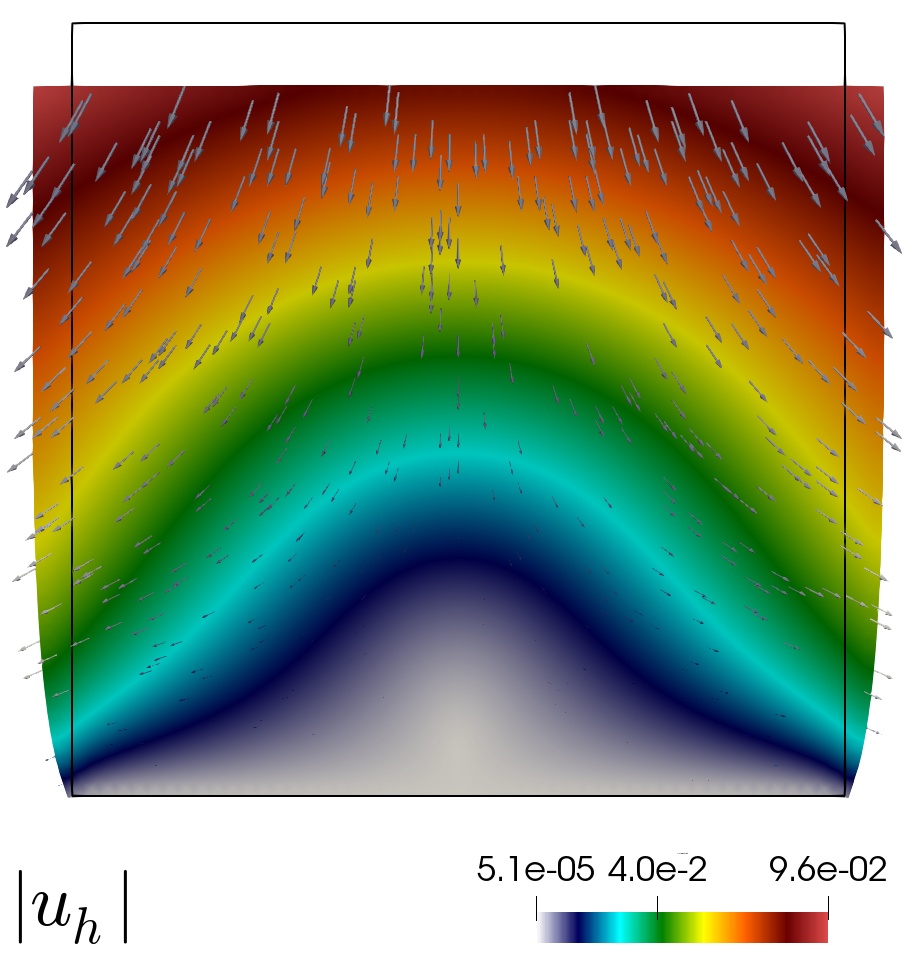}
\includegraphics[width = 0.23\textwidth]{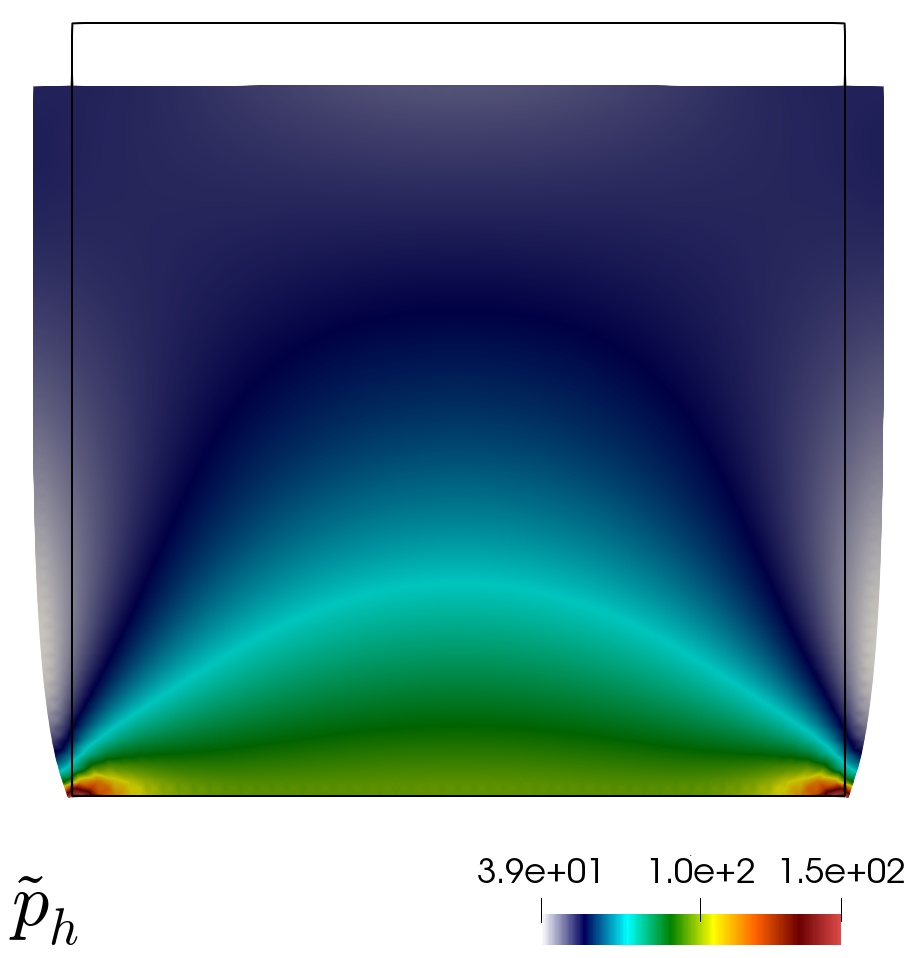}
\includegraphics[width = 0.23\textwidth]{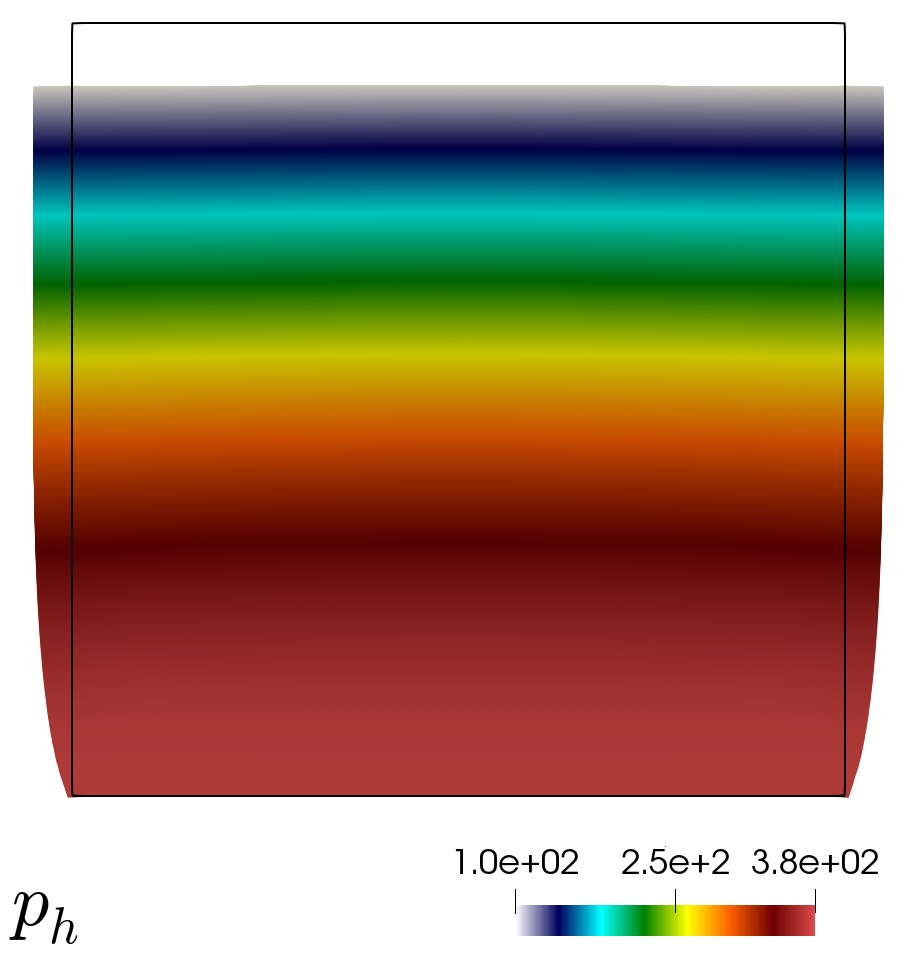}\\
\includegraphics[width = 0.23\textwidth]{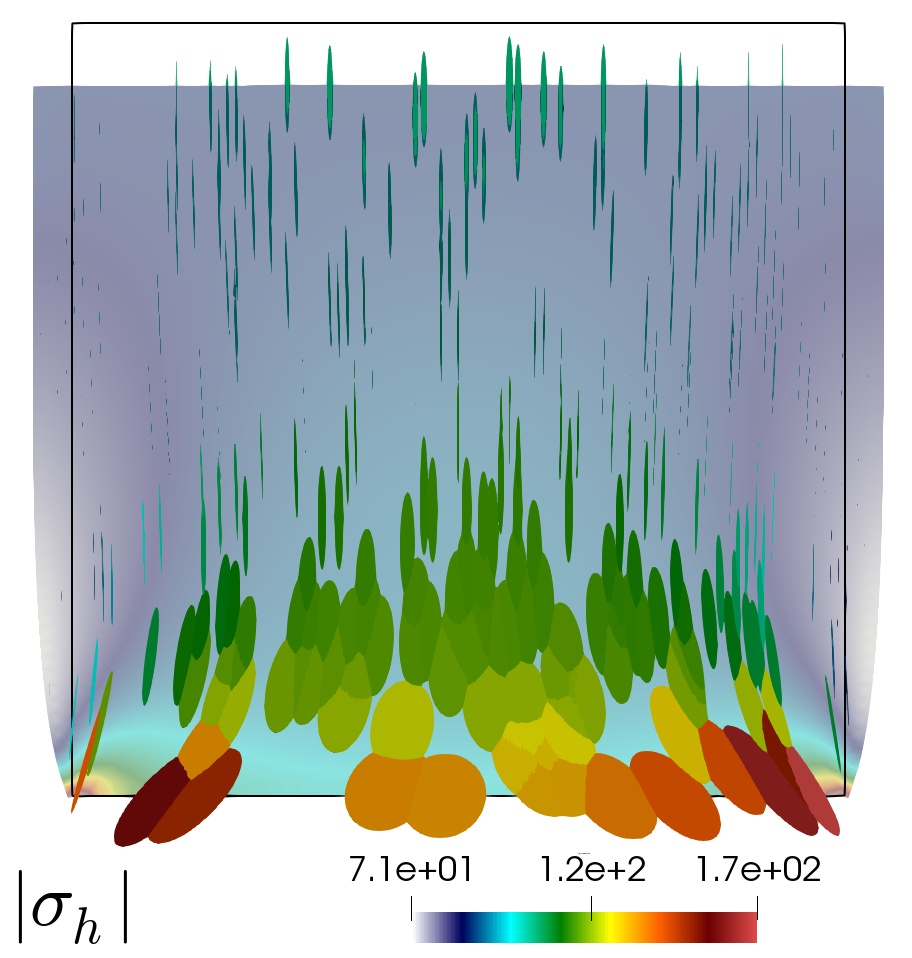}
\includegraphics[width = 0.23\textwidth]{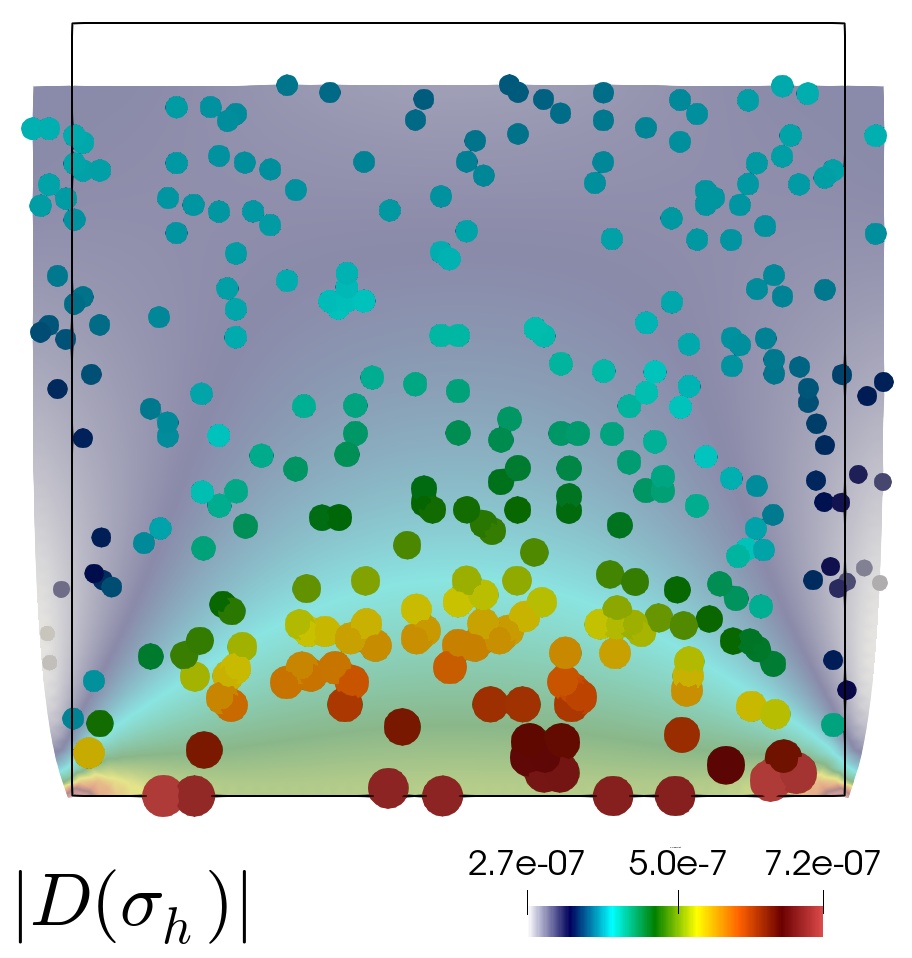}
\includegraphics[width = 0.23\textwidth]{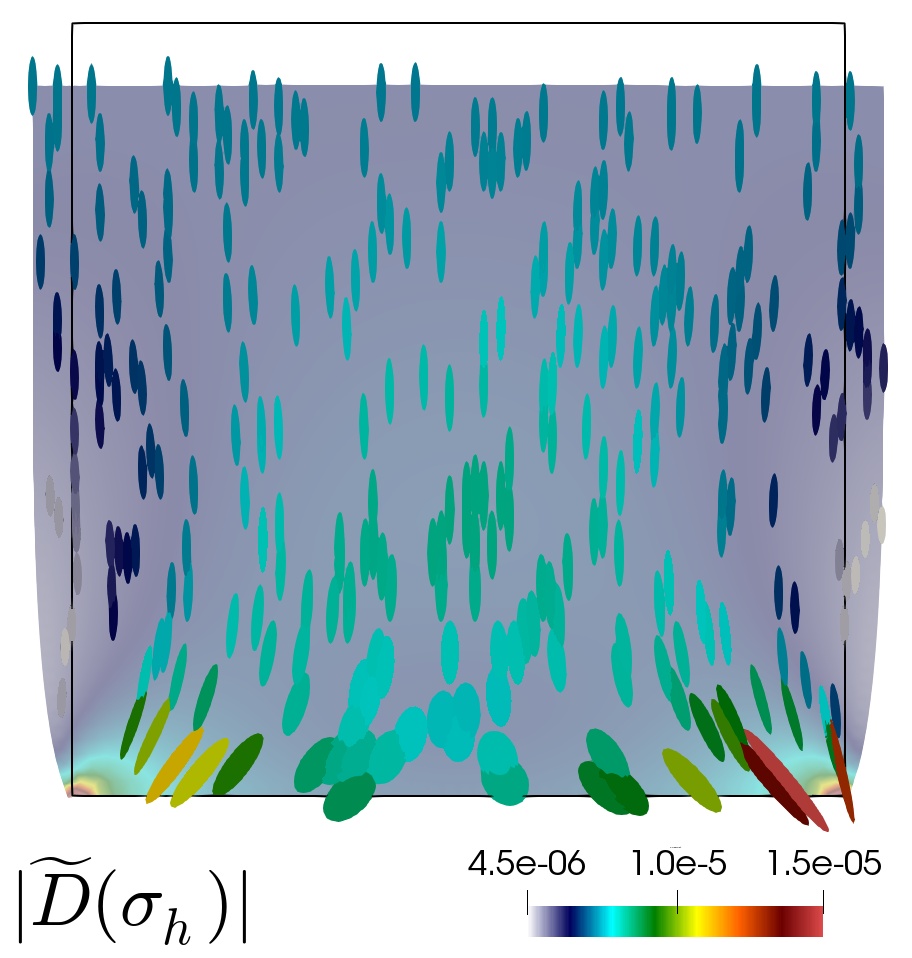}
\end{center}
\caption{Example 2. Stress-assisted diffusion. Top: Filtration of the CSF tracer into the slab of tissue for pure reaction-diffusion (left), and stress-hindered diffusion according to \eqref{eq:D} (centre), and the anisotropic form \eqref{eq:Dtilde} (right). Middle row: approximate displacement (left), total pressure (centre), and fluid pressure (right). Bottom: tensor glyph representation on the deformed configuration, of the total Cauchy stress (left), the isotropic diffusion (centre), and anisotropic one (right). All snapshots are taken at time $t=1800$\,s.} 
\label{fig:ex02}
\end{figure}

Three main scenarios are considered: (a) Pure reaction-diffusion of the tracer with $D_0$; (b) Adding the effect of isotropic stress-assisted diffusion according to 
\begin{equation}\label{eq:D} 
{D}(\bsigma) = D_0\mathbb{I} + D_0 \exp(-\eta_0\tr\bsigma)\mathbb{I},
\end{equation}
with $\eta_0 = 5\times 10^{-5}$; and (c) Modifying the functional form of the apparent diffusion to 
\begin{equation}\label{eq:Dtilde}
\tilde{D}(\bsigma) = \eta_0 D_0\mathbb{I} - \eta_2 D_0 \bsigma + \eta_2D_0 \bsigma^2,
\end{equation}
with $\eta_1 =0.02$ and $\eta_2 = 10^{-5}$. 
Figure~\ref{fig:ex02} shows the results from these tests. The first row illustrates how with pure diffusion (without stress-hinderance) the tracer penetrates faster than in the cases with stress-dependent diffusion (this is shown at $t = 1800$\,s). The deformation of the parenchymal slab using case (c) is shown in the second row, where we can observe a localisation of total pressure accumulation near the bottom corners of the domain. The bottom row indicates the degree of anisotropy of the total poroelastic stress and of the stress-altered diffusivity tensors according to \eqref{eq:D} and \eqref{eq:Dtilde}, also at $t = 1800$\,s. In case (c), apart from a slower diffusion than in case (a), a slight deviation from plane solute transport is seen (in the top-right panel of the figure), explained by the anisotropy differences exhibited in the bottom row of the figure. 

%****************************
\subsection{Example 3: Stress-hindered diffusion and transport in the brain}
Next we perform two application tests investigating the filtration properties of parenchymal brain tissue and the evolution of tracer concentration in its sleeping vs awake state. In the first example, we create  a 2D slab (1\,mm deep, 1 $\mu$m wide) of a mouse brain, and simulated tracer enrichment from the cortical surface into the cortex. The diffusion coefficient was here assumed to take the form 
\begin{equation}\label{eq:D2} 
D = D_0 - D_0\exp(-\eta|\tr(\bsigma)|).
\end{equation}

\begin{figure}[!t]
\begin{center}
\includegraphics[width = 0.75\textwidth]{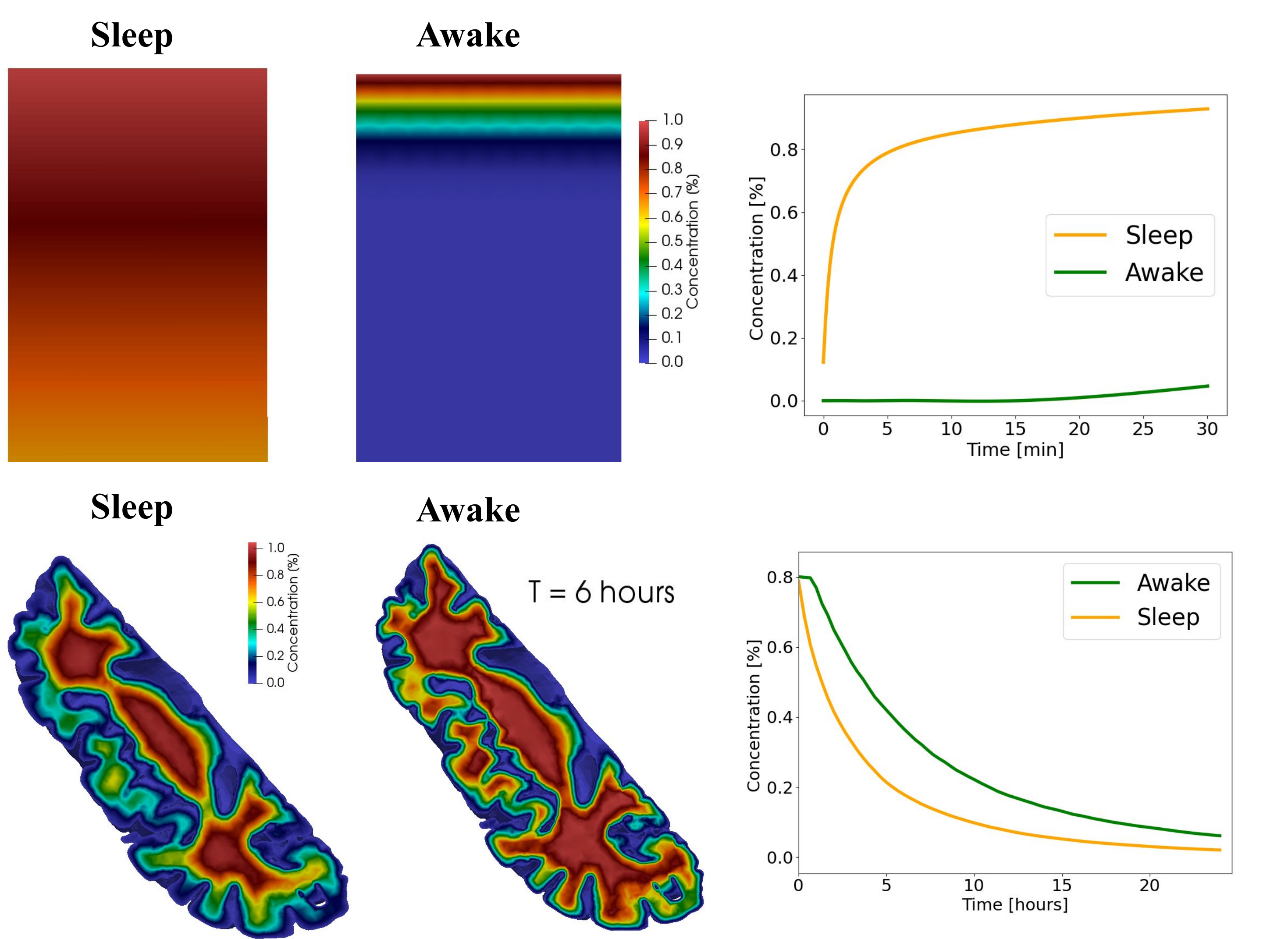}
\end{center}
\caption{Example 3. Tracer infiltration in the mouse cortex (top panel) and tracer clearance from the human brain (bottom panel) using the stress-hindered formulation as described by \eqref{eq:D2}. In the top-left we compare the sleeping versus awake state after 15 minutes of transport from the mouse cortex 300 $mu$m into the tissue. The concentration was held fixed at the cortex. To the right, the shown time evolution 100 $\mu$m below the cortex indicates that transport by stress-hindered diffusion is an order of magnitude lower than for free diffusion within the cortex. In the bottom panel we show the tracer distribution in a brain slice of a human brain. In the left panel is the tracer distribution after 12 hours of diffusive transport. In the awake state, stress hindered diffusion results in slightly slower clearance. In the right bottom panel, the average concentration for the entire piece of brain tissue is plotted over time.} 
\label{fig:ex03}
\end{figure}

At the cortex we set a given concentration of $\omega = 1$, while other boundaries were assigned homogeneous Neumann conditions and no initial tracer was assumed within the tissue. The depth (1 mm) was assumed to cover the entire width of the mouse cortex. Over time we measured the concentration (100 $\mu$m into the tissue) similarly to the results reported by \cite{xie2013sleep}. In the awake state, the porosity was set to $\phi = 0.14$, while in the awake state we used $\phi = 0.23$~\cite{xie2013sleep}. We assumed that the known change in volume fraction was associated with a similar change in the width of the cerebral cortex, which is known to change during sleep~\cite{elvsaashagen2017evidence}. Simulating the (compressed) awake state, we therefore added a force $\bsigma \nn = -p_0 \arctan(\frac{t}{50s})\nn$, where $p_0 = 0.5$\,mmHg to the top of the domain, while the bottom was fixed. The arctan function was used to ensure a smooth loading. In the sleeping state, no forces were applied, and tracers were thus allowed to diffuse freely according to the base diffusion $D_0$. In both cases the sidewalls were assigned Neumann conditions. The resulting displacement at the cortex surface was 0.09\,mm (data now shown). In Figure \ref{fig:ex03} (top panel) we show the resulting tracer concentration. We compare tracer concentrations after 15 minutes in the awake versus the sleeping state, from 0 to 300\,$\mu$m into the cortex (left), and show the time evolution of tracer concentration at a slice 100\,$\mu$m below the surface (right). The applied forces in the awake state significantly slows down transport into mouse cortex. 
Other parameters were here set as follows (see also~\cite{vinje2020intracranial, budday2015mechanical}) 
%
%\begin{gather*}
{\small
$$E = 800 \,\mathrm{Pa}, \quad 
\nu = 0.495,\quad 
c_0 = 2\times10^{-8}\,\text{Pa}^{-1}, \quad 
\kappa = 10^{-8}\,\text{mm}^2,\quad 
\alpha = 1, \quad \ell = 0, \quad \eta = 2\times 10^{-5},$$$$
\rho_s = 10^{-3}\,\text{g/mm}^3,\quad
\mu_f = 0.7 \times 10^{-3}\,\text{Pa}/\text{s},\quad \beta = 0.35,\quad 
\quad D_0 = 5.3\times 10^{-2}\text{mm}^3/\text{s}.$$}
%\end{gather*}
In the second scenario, we simulate tracer clearance from the human brain using a 3D mesh generated from the the right hemisphere of the built-in subject Bert in FreeSurfer~\cite{fischl2012freesurfer}. The mesh consisted of 400,000 cells, and the resulting system had 15M degrees of freedom. Here we have used an initial condition of $\omega = 0$ everywhere in the brain, while we set $\omega = 0$ on the brain surface. Parameters were set exactly as in the previous example, except that now $\eta = 2 \times 10^{-1}$ and the cortex pressure is $p_0 = 0.01$ mmHg. On the ventricles we imposed no-slip conditions. In Figure \ref{fig:ex03} (bottom panel) we show the tracer concentration in the sleeping versus awake brain after 12 hours of transport (left), and also show the time evolution of the average tracer concentration within the slice shown (right). Transport rates differ between the sleeping versus the awake state, but occurs on approximately the same time scale \cite{valnes2020apparent} in contrast to the mouse test case. Our results from the human brain test case thus reflects differences in clearance in the awake versus sleeping state observed experimentally~\cite{eide2021sleep}.

%****************************
%\section{Concluding remarks}\label{sec:concl}
%****************************
%We have incorporated stress-altered diffusion to the transport of solutes within poroelastic structures fully saturated with an incompressible fluid. The set of equations models how a tracer permeates through the brain tissue. Diffusion hindered by stress has a better agreement with filtration rates observed from experiments. We have used a mixed finite element formulation for the poromechanics and under standard assumptions have shown that the scheme converges optimally to the unique weak solution of the system. 

We close this section mentioning that next steps in the process of model refinement include the study of interfacial flow between the poroelastic structure (at the macroscale) coupled with the surrounding layer of CSF and the specific waste disposal system, which is still not well understood. For this we will also use multiple network models \cite{hong19,piersanti20}, the preconditioning framework for perturbed saddle-point problems from \cite{boon21}, and interface couplings such as those in \cite{bukac15,li20,taffetani20}. 
Other envisaged generalisations deal with large deformations and kinematics through a more complete mixture theory following, e.g., \cite{lang16,costanzo17}.

%*********************************************
\bibliographystyle{siam}
\bibliography{gmrv-bib}
%*********************************************
\end{document}